\documentclass{amsart}
\usepackage{amscd,amsmath,xypic,amssymb,combelow,tikz-cd,etoolbox}

\makeatletter
\patchcmd{\@settitle}{\uppercasenonmath\@title}{}{}{}
\makeatother

\xyoption{all}
\CompileMatrices

\newcommand{\nc}{\newcommand}

\makeatletter
\@addtoreset{equation}{section}
\makeatother

\newtheorem{theorem}[subsection]{Theorem}
\newtheorem{proposition}[subsection]{Proposition}

\newtheorem{definition}[subsection]{Definition}

\newtheorem{claim}[subsection]{Claim}

\newtheorem{remark}[subsection]{Remark}

\nc{\fa}{{\mathfrak{a}}}
\nc{\fb}{{\mathfrak{b}}}
\nc{\fg}{{\mathfrak{g}}}
\nc{\fh}{{\mathfrak{h}}}
\nc{\fj}{{\mathfrak{j}}}
\nc{\fn}{{\mathfrak{n}}}
\nc{\fm}{{\mathfrak{m}}}
\nc{\fu}{{\mathfrak{u}}}
\nc{\fp}{{\mathfrak{p}}}
\nc{\fr}{{\mathfrak{r}}}
\nc{\ft}{{\mathfrak{t}}}
\nc{\fsl}{{\mathfrak{sl}}}
\nc{\fgl}{{\mathfrak{gl}}}
\nc{\hsl}{{\widehat{\mathfrak{sl}}}}
\nc{\hgl}{{\widehat{\mathfrak{gl}}}}
\nc{\hg}{{\widehat{\mathfrak{g}}}}
\nc{\chg}{{\widehat{\mathfrak{g}}}{}^\vee}
\nc{\hn}{{\widehat{\mathfrak{n}}}}
\nc{\chn}{{\widehat{\mathfrak{n}}}{}^\vee}

\nc{\Mod}{{\textrm{Mod}}}

\nc{\wGL}{{\widehat{GL}^+}}

\nc{\BA}{{\mathbb{A}}}
\nc{\BC}{{\mathbb{C}}}
\nc{\BM}{{\mathbb{M}}}
\nc{\BN}{{\mathbb{N}}}
\nc{\BF}{{\mathbb{F}}}
\nc{\BH}{{\mathbb{H}}}
\nc{\BP}{{\mathbb{P}}}
\nc{\BQ}{{\mathbb{Q}}}
\nc{\BR}{{\mathbb{R}}}
\nc{\BZ}{{\mathbb{Z}}}

\nc{\ff}{{\mathbb{F}}}
\nc{\kk}{{\mathbb{K}}}
\nc{\kko}{{\mathbb{K}}}

\nc{\coh}{{\text{Coh}}}

\nc{\CA}{{\mathcal{A}}}
\nc{\CC}{{\mathcal{C}}}
\nc{\CB}{{\mathcal{B}}}
\nc{\DD}{{\mathcal{D}}}
\nc{\CE}{{\mathcal{E}}}
\nc{\CF}{{\mathcal{F}}}
\nc{\tCF}{{\widetilde{\CF}}}
\nc{\tCT}{{\widetilde{\CT}}}
\nc{\oCF}{{\overline{\CF}}}
\nc{\CG}{{\mathcal{G}}}
\nc{\CL}{{\mathcal{L}}}
\nc{\CK}{{\mathcal{K}}}
\nc{\CI}{{\mathcal{I}}}
\nc{\CM}{{\mathcal{M}}}
\nc{\CH}{{\mathcal{H}}}
\nc{\CN}{{\mathcal{N}}}
\nc{\CO}{{\mathcal{O}}}
\nc{\CP}{{\mathcal{P}}}
\nc{\CR}{{\mathcal{R}}}
\nc{\CQ}{{\mathcal{Q}}}
\nc{\CS}{{\mathcal{S}}}
\nc{\CT}{{\mathcal{T}}}
\nc{\CU}{{\mathcal{U}}}
\nc{\CV}{{\mathcal{V}}}
\nc{\CW}{{\mathcal{W}}}

\nc{\tpsi}{{\widetilde{\Psi}}}
\nc{\wpi}{{\widetilde{\pi}}}
\nc{\Ker}{{\text{Ker }}}

\nc{\CX}{{\mathcal{X}}}
\nc{\tCX}{{\widetilde{\mathcal{X}}}}
\nc{\CY}{{\mathcal{Y}}}
\nc{\tCY}{{\widetilde{\mathcal{Y}}}}

\nc{\tN}{{\widetilde{\CN}}}
\nc{\pN}{{\BP\widetilde{\CN}}}

\nc{\tT}{{T}}

\nc{\fC}{{\mathfrak{C}}}
\nc{\fZ}{{\mathfrak{Z}}}
\nc{\fU}{{\mathfrak{U}}}
\nc{\fV}{{\mathfrak{V}}}
\nc{\fS}{{\mathfrak{S}}}

\nc{\od}{{\overline{d}}}
\nc{\rg}{{\textrm{R}\Gamma}}
\nc{\erg}{{\emph{R}\Gamma}}
\nc{\id}{{\textrm{Id}}}
\nc{\rhom}{{\textrm{RHom}}}

\def\ph{\varphi}

\def\Ext{\textrm{Ext}}
\def\Hom{\textrm{Hom}}

\def\e{\varepsilon}

\def\barA{\bar{A}}
\def\barE{\bar{E}}
\def\barCS{\bar{\CS}}
\def\barT{\bar{T}}
\def\barPhi{\bar{\Phi}}

\def\barm{\bar{m}}
\def\baru{\bar{u}}
\def\barb{\bar{b}}

\def\barp{\bar{p}}
\def\barbu{\bar{\bu}}

\def\bary{\bar{y}}
\def\barz{\bar{z}}
\def\barR{\bar{R}}
\def\barW{\bar{W}}
\def\barLambda{\bar{\Lambda}}
\def\barchi{\bar{\chi}}
\def\bartau{\bar{\tau}}
\def\barzeta{\bar{\zeta}}

\def\barbla{\bar{\bla}}
\def\barbmu{\bar{\bmu}}
\def\barbnu{\bar{\bnu}}

\def\tp{\tilde{p}}

\def\pt{\textrm{pt}}
\def\and{\textrm{ }\&\textrm{ }}

\def\sym{\textrm{Sym}}
\def\esym{\emph{\sym}}

\def\Tr{\textrm{Tr}}

\def\tCF{\widetilde{\CF}}

\def\tpi{\widetilde{\pi}}

\def\res{\textrm{Res }}

\def\tp{\tilde{p}}

\def\loccit{\emph{loc. cit. }}
\def\loccitt{\emph{loc. cit.}}
\def\op{{\textrm{op}}}

\def\tab{\text{} \\}

\def\bu{{\textbf{u}}}
\def\ebu{{\emph{\bu}}}

\def\la{{\lambda}}
\def\lamu{{\lambda \backslash \mu}}
\def\lanu{{\lambda \backslash \nu}}
\def\munu{{\mu \backslash \nu}}
\def\bla{{\boldsymbol{\la}}}
\def\bmu{{\boldsymbol{\mu}}}
\def\bnu{{\boldsymbol{\nu}}}
\def\blamu{{\boldsymbol{\lamu}}}
\def\blanu{{\boldsymbol{\lanu}}}
\def\bmunu{{\boldsymbol{\munu}}}

\def\sq{\square}
\def\bsq{\blacksquare}

\def\Tan{\text{Tan}}

\def\Id{\text{Id}}
\def\wW{\widetilde{W}}
\def\wCA{\widehat{\CA}}

\def\wCH{\widehat{\CH}}

\def\tCB{\widetilde{\CB}}
\def\tCH{\widetilde{\CH}}

\def\wwedge{\widetilde{\wedge}}
\def\wGamma{\widetilde{\Gamma}}
\def\wLambda{\widetilde{\Lambda}}

\def\tp{\widetilde{p}}

\def\ext{\text{ext}}
\def\diag{\text{diag}}

\def\resy{\underset{y=\infty}\res}

\begin{document}

\title[The $q$--AGT--W relations via shuffle algebras]{\large{\textbf{THE $q$--AGT--W RELATIONS VIA SHUFFLE ALGEBRAS}}}
\author[Andrei Negu\cb t]{Andrei Negu\cb t}
\address{MIT, Department of Mathematics, Cambridge, MA, USA}
\address{Simion Stoilow Institute of Mathematics, Bucharest, Romania}
\email{andrei.negut@gmail.com}

\maketitle

\renewcommand{\thefootnote}{\fnsymbol{footnote}} 
\renewcommand{\thefootnote}{\arabic{footnote}} 

\renewcommand{\thefootnote}{\fnsymbol{footnote}} 
\renewcommand{\thefootnote}{\arabic{footnote}} 

\begin{abstract} \noindent We construct the action of the $q$--deformed $W$--algebra on its level $r$ representation geometrically, using the moduli space of $U(r)$ instantons on the plane and the double shuffle algebra. We give an explicit LDU decomposition for the action of $W$--algebra currents in the fixed point basis of the level $r$ representation, and prove a relation between the Carlsson-Okounkov Ext operator and intertwiners for the deformed $W$--algebra. We interpret this result as a $q$--deformed version of the AGT--W relations.  

\end{abstract}

\section{Introduction}

\noindent Fix $r \in \BN$. The moduli space $\CM$ of rank $r$ framed sheaves on $\BP^2$ is an algebro-geometric incarnation of the moduli space of $U(r)$ instantons, where the Nekrasov partition function naturally appears. More precisely, it has been known from the work of \cite{CNO}, \cite{CO}, \cite{ON}, \cite{Nek} that the partition function of $5d$ $U(r)^k$--gauge theory with bi-fundamental hypermultiplets $m_1,...,m_k$, in the presence of full $\Omega$--background, is:
\begin{equation}
\label{eqn:nek part}
Z_{m_1,...,m_k}(x_1,...,x_k) = \Tr \left( A_{m_1}(x_1) \circ ... \circ A_{m_k}(x_k) \Big|_{\bu^{k+1} = \bu^1} \right) 
\end{equation}
where the Ext operator (see \eqref{eqn:a correspondence} for the precise geometric definition) is:
\begin{equation}
\label{eqn:ext 0}
A_{m_i}(x_i) : K_{\bu^{i+1}} \longrightarrow K_{\bu^i} 
\end{equation}
and $K_{\bu^i}$ denotes the equivariant $K$--theory of the moduli space of rank $r$ sheaves, with equivariant weights encoded in the vector of parameters $\bu^i = (u_1^i,...,u_r^i)$. The partition function of linear quiver gauge theory can be recovered from the operators $A_m(x)$ and their matrix coefficients, as explained in \cite{Bo}. This partition function has been studied extensively and from many different points of view, see e.g. \cite{NY}, \cite{NS}.

\tab 
The main purpose of the present paper is to mathematically state and prove a connection between the rank $r$ Nekrasov partition function and conformal blocks for the $_{q}W$--algebra (more commonly called ``defomed $W$--algebra") of type $\fgl_r$. The proof of our main Theorem \ref{thm:main} uses two main mathematical tools: expressing $_qW$--algebras via shuffle algebras, and performing intersection--theoretic computations with the Ext operator \eqref{eqn:ext 0}. We interpret our result as a $q$--deformed version of the well-known AGT--W relations between gauge theory and conformal field theory (these were introduced by Alday, Gaiotto and Tachikawa and extended by Wyllard in the undeformed case, and formulated in the $q$--deformed case by Awata and Yamada, see \cite{AY}, \cite{Bo}, \cite{Ta}, \cite{T} among other references. The physical literature on the subject is vast, see for example \cite{Aga}, \cite{KP}, \cite{NPS} for other points of view).  

\noindent The algebra we study is the tensor product of the $_qW$--algebra of type $\fsl_r$ (\cite{AKOS}, \cite{FF}) and a $_q$Heisenberg algebra. By close analogy with \loccitt, we show in Section \ref{sec:miura} that the defining currents of our $_qW$--algebra are ``elementary symmetric functions":
\begin{equation}
\label{eqn:miura 0}
W_k(z) = \sum_{1\leq i_1 < ... < i_k \leq r} : \exp \left[ b^{i_1}(x) \right] \exp \left[ b^{i_2}\left(\frac xq\right) \right] ... \exp \left[ b^{i_k} \left(\frac x{q^{k-1}} \right) \right] : 
\end{equation}
in a family of bosonic fields $b^1(z),...,b^r(z)$ which satisfy the commutation relations \eqref{eqn:bosons}. The nice thing about the $\fgl_r$ case is that one can send $r \rightarrow \infty$, and the resulting limit can be interpreted as the upper half of the double shuffle algebra (as in Section \ref{sec:algebra}). For fixed $r$, the definition \eqref{eqn:miura 0} implies the following relations:
$$
W_0(x) = 1, \qquad W_k(x) = 0\quad \text{for all }k>r
$$
and:
\begin{equation}
\label{eqn:relations}
W_k(x) W_{k'}(y) \cdot f_{kk'}\left(\frac yx \right) -  W_{k'}(y) W_k(x) \cdot f_{k'k}\left(\frac xy \right) =
\end{equation}
$$
= \sum_{i=\max(0,k'-k)+1}^{k'} \delta\left(\frac {y}{xq^i} \right) \left[ W_{k'-i}(x) W_{k+i}(y) f_{k'-i,k+i}\left(\frac yx \right) \Big|_{x = \frac y{q^i}} \right] \theta(\min(i,k-k'+i))
$$
$$
- \sum^{k}_{i = \max(0,k-k')+1} \delta\left(\frac {x}{yq^i} \right) \left[ W_{k-i}(y) W_{k'+i}(x) f_{k-i,k'+i}\left(\frac xy \right) \Big|_{y = \frac x{q^i}} \right] \theta(\min(i,k'-k+i))
$$
The quantity $\theta(s)$ is defined for all $s\in \BN$ in \eqref{eqn:def theta}, while the power series $f_{kk'}(z)$ is defined in \eqref{eqn:def f} (note that we always expand it in $|z| \ll 1$). In Proposition \ref{prop:iso}, we will explain how formulas \eqref{eqn:relations} differ from those of \cite{AKOS} and \cite{FF}.

\tab 
Our strategy is quite well-known to mathematicians and physicists: to recast the AGT--W relations as a connection between the operator $A_m(x)$ and intertwiners for the $_qW$--algebra. This starts with Theorem \ref{thm:geom} below, which states that for arbitrary $r \in \BN$ and generic equivariant parameters $\bu = (u_1,...,u_r)$, the $K$--theory group $K_\bu$ of the moduli space of rank $r$ sheaves is isomorphic to the Verma module of the $_qW$--algebra, with highest weight prescribed by the equivariant parameters $\bu$. Note that our construction and proof are purely geometric, and do not use the isomorphism between the level $r$ representation and a tensor product of $r$ Fock spaces (which was used e.g. in \cite{AFHKSY}). This geometric definition is fruitful because it can be extended to moduli of sheaves on other surfaces, see \cite{W gen} and \cite{AGT}: \\

\begin{theorem}
\label{thm:main}

For arbitrary $r \geq 1$ and collections $\ebu = (u_1,...,u_r)$, $\ebu' = (u_1',...,u_r')$ of equivariant parameters, let $u = u_1...u_r$ and $u'=u_1'...u_r'$. The Ext operator:
$$
A_m(x) : K_{\ebu'} \longrightarrow K_{\ebu} 
$$
is connected by the simple relation \eqref{eqn:def phi} with the vertex operator:
$$
\Phi_m(x) : K_{\ebu'} \longrightarrow K_{\ebu} 
$$
which satisfies the following commutation relations with the $_qW$--algebra currents:
\begin{equation}
\label{eqn:phi wk}
\left[ \Phi_m(x), W_k(y) \right]_{m^k} \cdot \prod_{i=1}^k \left(1 - \frac {m^r u x}{q^{r-i} u' y} \right) = 0
\end{equation}
where $[A,B]_s = AB - s BA$ denotes the $s$--commutator. 
\end{theorem}

\noindent (see Remark \ref{rem:future} for a stronger form of relation \eqref{eqn:phi wk}, to be proved in \cite{AGT}). Let us say a few words about our methods. Theorem \ref{thm:geom} below was proved in the undeformed case in \cite{MO} and \cite{SV2}, by different means. Our proof of the deformed case is new, and uses the embedding of the $_qW$--algebra into the double shuffle algebra $\CA$ of \cite{BS}, \cite{FO}, \cite{FT}, \cite{SV}. More specifically, we recast relations \eqref{eqn:relations} in terms of the shuffle algebra, recall the action of the shuffle algebra on $K_\bu$ from \cite{Mod}, and show that relations \eqref{eqn:relations} hold in $K_\bu$ by a shuffle algebra computation. In particular, we prove the following formula for the LDU decomposition of $_qW$--algebra currents in the basis of fixed points indexed by $r$--tuples of partitions $\bla = (\lambda^1,...,\lambda^r)$ (note that this basis is quite different, and in a sense ``orthogonal", to the basis arising from the isomorphism $K_\bu \cong \text{Fock}^{\otimes r}$, which we will review in Subsection \ref{sub:stable basis}): \\

\begin{theorem}
\label{thm:ldu}
Let $D_x$ denote the $q$--difference operator $f(x) \leadsto f (x q)$. The action of the $_qW$--algebra on its level $r$ representation (interpreted as the $K$--theory of the moduli space of framed sheaves, as in Subsections \ref{sub:sheaf} and \ref{sub:tor act}) is given by:
\begin{equation}
\label{eqn:w currents}
\sum_{k=0}^\infty \frac {W_k(x)}{(-yD_x)^k} = T \left(x^{-1}, yD_x \right)^\leftarrow E\left( y D_x \right) \ T\left(xq, y D_x \right)^\rightarrow
\end{equation}
where the three factors are lower triangular, diagonal and upper triangular, respectively (see Remark \ref{rem:paragraph} for how to place the non-commuting symbols $x$ and $D_x$ in \eqref{eqn:w currents}). These operators are interpreted geometrically in Section \ref{sec:geom}, when we obtain: 
$$
\langle \bmu | W_k(x) |\bla \rangle = \sum^{d_\rightarrow - d_\leftarrow = d}_{k_\leftarrow + k_0 + k_\rightarrow = k}  \sum_{\bnu \subset \bla \cap \bmu} \langle \bmu |T_{d_\leftarrow,k_\leftarrow}^\leftarrow | \bnu \rangle \langle \bnu | E_{0,k_0} | \bnu \rangle \langle \bnu | T_{d_\rightarrow, k_\rightarrow}^\rightarrow | \bla \rangle  q^{(k-1)d_\rightarrow}
$$
with the three matrix coefficients in the right-hand side given by \eqref{eqn:lower}, \eqref{eqn:upper} and \eqref{eqn:diagonal}, respectively. The basis $|\bla \rangle \in K$ consists of torus fixed points in $K$--theory.

\end{theorem}

\tab 
We will prove formula \eqref{eqn:phi wk} by a geometric computation, which takes up most of Section \ref{sec:ext} and closely follows that in the undeformed case in \cite{Ext}. Note that in \cite{Bo0} and \cite{Bo}, the authors computed the commutation relations of the Ext operator with the degree one generators of the double shuffle algebra (some of these relations had appeared in \cite{Mod}, but without any of the very interesting physics studied in \cite{Bo}). In the present paper, we take the orthogonal viewpoint of computing the commutation relations of the Ext operator with the $_qW$--currents directly. Thus, while the present work uses the double shuffle algebra as a technical step, the final result is presented only in terms of the $_qW$--algebra action on $K$--theory. Finally, in Section \ref{sec:classical}, we show how to degenerate the $_qW$--algebra to the usual one, in which case the limit of $\Phi_m(x)$ is precisely the vertex operator studied in \cite{FL}.  \\

\noindent Let us note that the three factors that make up the LDU decomposition of the currents \eqref{eqn:w currents} are purely geometric, and one may ask what the analogous construction means for moduli spaces of sheaves on a more general (projective or toric) surface. Understanding the corresponding $_qW$--algebra will be the subject of \cite{W gen}.

\tab
I would like to thank the wonderful audience of my mini--course at Institut Henri Poincar\'e for suggesting this problem (special thanks to Amir-Kian Kashani-Poor for the invitation), and especially Francesco Sala for everything he taught me about AGT--W. Many thanks are due to Alexander Tsymbaliuk for his patience and knowledge. I gratefully acknowledge the support of NSF grant DMS--1600375. \\

\section{The shuffle algebra and deformed $W$--algebras}
\label{sec:algebra}

\subsection{}
\label{sub:def shuf}



Let $q_1,q_2$ be indeterminates, and set $q=q_1q_2$ and $\ff = \BQ(q_1,q_2)$. One of the basic objects of the present paper is the rational function:
\begin{equation}
\label{eqn:def zeta}
\zeta(x) = \frac {(1-q_1x)(1-q_2x)}{(1-x)(1-qx)} 
\end{equation}
which we observe satisfies the relation:
\begin{equation}
\label{eqn:inv zeta}
\zeta(x) = \zeta \left(\frac 1{xq} \right)
\end{equation}
Note that we can write $\zeta$ in the form:
\begin{equation}
\label{eqn:zeta exponential}
\zeta(x) = \exp \left[ \sum_{n=1}^\infty \frac {(1-q_1^n)(1-q_2^n)x^n}n \right]
\end{equation}
Consider an infinite set of variables $z_1,z_2,...$, and take the $\ff-$vector space:
\begin{equation}
\label{eqn:big}
V = \bigoplus_{k \geq 0} \ff(z_{1},...,z_{k})^{\sym}
\end{equation}
of rational functions which are symmetric in the variables $z_1,...,z_k$, for any $k$. We endow $V$ with an $\ff-$algebra structure by the \textbf{shuffle product}:
$$
R(z_{1},...,z_{k}) * R'(z_{1},...,z_{k'}) =
$$
\begin{equation}
\label{eqn:mult}
= \frac 1{k! k'!} \cdot \textrm{Sym} \left[R(z_{1},...,z_{k})R'(z_{k+1},...,z_{k+k'}) \prod_{i=1}^k \prod_{j = k+1}^{k+k'} \zeta \left( \frac {z_i}{z_j} \right) \right]
\end{equation}
where \textrm{Sym} denotes the symmetrization operator:
$$
\textrm{Sym}\left( R(z_1,...,z_k) \right) = \sum_{\sigma \in S(k)} R(z_{\sigma(1)},...,z_{\sigma(k)})
$$
The \textbf{shuffle algebra} $\CS \subset V$ is defined as the set of rational functions of the form:
\begin{equation}
\label{eqn:shuf}
R(z_{1},...,z_{k}) = \frac {r(z_{1},...,z_{k})}{\prod_{1 \leq i \neq j \leq k} (z_{i} - z_{j}q)} 
\end{equation}
where $r$ is a symmetric Laurent polynomial that satisfies the \textbf{wheel conditions}:
\begin{equation}
\label{eqn:wheel}
r(z_1,...,z_k) \Big|_{\left\{ \frac {z_1}{z_2}, \frac {z_2}{z_3}, \frac {z_3}{z_1} \right\} = \left\{q_1,q_2, \frac 1q \right\}} = 0
\end{equation}
This condition is vacuous for $k\leq 2$. It is straightforward to show that $\CS$ is an algebra, and that the shuffle product preserves the grading by the number of variables $k$ in \eqref{eqn:big}, and also the grading by the homogeneous degree $d$ of rational functions. We will use the following notation for the graded and bigraded pieces of $\CS$:
$$
\CS = \bigoplus_{k=0}^\infty \CS_k, \qquad \qquad \CS_k = \bigoplus_{d\in \BZ} \CS_{k,d}
$$

\begin{theorem}
\label{thm:shuf} \emph{(\cite{Shuf})}
The shuffle algebra $\CS$ is generated by the degree one elements $R(z) = z^d \in \CS_1$, as $d$ goes over $\BZ$, under the shuffle product \eqref{eqn:mult} (when dealing with rational functions in a single variable, we will denote the variable by $z = z_1$). \\
\end{theorem}

\subsection{}
\label{sub:explicit shuffle}

The following elements of the shuffle algebra $\CS$ were constructed in \cite{Shuf} (note that the multiplication in \loccit is defined with respect to the function $\zeta(x^{-1})^{-1}$ instead of $\zeta(x)$, but the fact that the two structures are isomorphic is easily observed). Below, we consider any $k>0$ and $d\in \BZ$, and write $n = \gcd(k,d)$, $a = \frac kn$:
\begin{equation}
\label{eqn:pkd}
P_{k,d}  = \sym \left[ \frac {\prod_{i=1}^k z_i^{\left \lfloor \frac {id}k \right \rfloor - \left \lfloor \frac {(i-1)d}k \right \rfloor}}{\prod_{i=1}^{k-1} \left(1 - \frac {qz_{i+1}}{z_i} \right)} \sum_{s=0}^{n-1} q^{s} \frac {z_{a(n-1)+1}...z_{a(n-s)+1}}{{z_{a(n-1)} ...z_{a(n-s)}}} \prod_{i < j} \zeta \left( \frac {z_i}{z_j} \right) \right]
\end{equation}
\begin{equation}
\label{eqn:hkd}
H_{k,d} =  \sym \left[ \frac {\prod_{i=1}^k z_i^{\left \lfloor \frac {id}k \right \rfloor - \left \lfloor \frac {(i-1)d}k \right \rfloor}}{\prod_{i=1}^{k-1} \left(1 - \frac {qz_{i+1}}{z_i} \right)} \prod_{i < j} \zeta \left( \frac {z_i}{z_j} \right) \right]
\end{equation}
\begin{equation}
\label{eqn:ekd}
E_{k,d} = (-q)^{n-1} \cdot \sym \left[ \frac {\prod_{i=1}^k z_i^{\left \lceil \frac {id}k \right \rceil - \left \lceil \frac {(i-1)d}k \right \rceil + \delta_i^k - \delta_i^1}}{\prod_{i=1}^{k-1} \left(1 - \frac {qz_{i+1}}{z_i} \right)} \prod_{i < j} \zeta \left( \frac {z_i}{z_j} \right) \right]
\end{equation}
\begin{equation}
\label{eqn:qkd}
Q_{k,d}  = \left(1 - \frac 1q \right) \cdot
\end{equation}
$$
\cdot \ \sym \left[ \frac {\prod_{i=1}^k z_i^{\left \lfloor \frac {id}k \right \rfloor - \left \lfloor \frac {(i-1)d}k \right \rfloor}}{\prod_{i=1}^{k-1} \left(1 - \frac {qz_{i+1}}{z_i} \right)} \sum_{s=0}^{n-1} \frac {z_{a(n-1)+1}...z_{a(n-s)+1}}{{z_{a(n-1)} ...z_{a(n-s)}}} \prod_{i < j} \zeta \left( \frac {z_i}{z_j} \right) \right] \qquad
$$
It was shown in \cite{Shuf} that $P_{k,d}$ has ``minimal degree", in a sense made explicit in \loccitt, among all symmetric rational functions in $k$ variables which have homogeneous degree $d$ and satisfy the wheel conditions \eqref{eqn:wheel}. As for $H_{k,d}$, $E_{k,d}$ and $Q_{k,d}$, they are in relation to $P_{k,d}$ as complete, elementary and plethystically modified complete symmetric functions, respectively, are in relation to power sum functions. Specifically, this means that for all coprime integers $a$ and $b$, we have:
\begin{equation}
\label{eqn:def h}
\sum_{n=0}^\infty \frac {H_{an,bn}}{x^n} = \exp \left[ \sum_{n=1}^\infty \frac {P_{an,bn}}{n x^n} \right]
\end{equation}
\begin{equation}
\label{eqn:def e}
\sum_{n=0}^\infty \frac {E_{an,bn} (-1)^n}{x^n} = \exp \left[ - \sum_{n=1}^\infty \frac {P_{an,bn}}{n x^n} \right] \qquad
\end{equation}
\begin{equation}
\label{eqn:def q}
\sum_{n=0}^\infty \frac {Q_{an,bn}}{x^n} \ = \exp \left[ \sum_{n=1}^\infty \frac {P_{an,bn}}{n x^n} \cdot (1 - q^{-n}) \right]
\end{equation}
We henceforth take \eqref{eqn:hkd}, \eqref{eqn:ekd}, \eqref{eqn:qkd} as the definition of the shuffle elements $H_{k,d}$, $E_{k,d}$, $Q_{k,d}$, and we will prove formulas \eqref{eqn:def h}, \eqref{eqn:def e}, \eqref{eqn:def q} in the Appendix. \\

\subsection{}
\label{sub:double}

It was shown in \cite{Shuf} that $\CS$ has a coproduct and a bialgebra pairing (strictly speaking, this is true only after enlarging the shuffle algebra by adding commuting elements $a_1,a_2,...$ as below, and we refer the reader to \loccit for the details), which allow us to construct its Drinfeld double. The double is defined as the algebra:
\begin{equation}
\label{eqn:def double}
\CA = \CA^\leftarrow \otimes \CA^{\diag} \otimes \CA^\rightarrow
\end{equation}
where for a central element $c$ and commuting elements $\{a_n\}_{n \in \BZ\backslash 0}$ we set:
$$
\CA^\leftarrow = \CS \qquad \qquad \CA^{\diag} = \BF[...,a_{-2},a_{-1},a_1,a_2,..., c^{\pm 1}] \qquad \qquad \CA^\rightarrow = \CS^\op
$$
Therefore, associated to any shuffle element $R$ as in \eqref{eqn:shuf}, we have two elements $R^\leftarrow \in \CA^\leftarrow$ and $R^\rightarrow \in \CA^\rightarrow$, which satisfy the following rules for all $R_1, R_2 \in \CS$:
$$
(R_1 * R_2)^\leftarrow = R_1^\leftarrow * R_2^\leftarrow
$$
$$
(R_1*R_2)^\rightarrow = R_2^\rightarrow * R_1^\rightarrow
$$
We impose the following relations between the three subalgebras of \eqref{eqn:def double}:
\begin{equation}
\label{eqn:comm1}
[R^\leftarrow(z_{1},...,z_{k}), a_n] \ = \ R^\leftarrow(z_{1},...,z_{k}) (z_1^n+...+z_k^n)
\end{equation}
\begin{equation}
\label{eqn:comm2}
[R^\rightarrow(z_{1},...,z_{k}), a_n] = - R^\rightarrow(z_{1},...,z_{k}) (z_1^n+...+z_k^n)
\end{equation}
for all $R \in \CS, n \in \BZ \backslash 0$, as well as:
\begin{equation}
\label{eqn:comm3}
\left[ (z^d)^\rightarrow, (z^{d'})^\leftarrow \right] = \frac {(1-q_1)(1-q_2)}{q^{-1} - 1} \left( A_{d+d'}  \delta_{d+d' \geq 0} - \frac 1c \cdot A_{d+d'}  \delta_{d+d'\leq 0}  \right)
\end{equation}
where the generators $A_n$ are defined in terms of the $a_n$ by:
\begin{equation}
\label{eqn:defp}
\sum_{n=0}^\infty \frac {A_{\pm n}}{x^{\pm n}} = \exp \left[ \pm \sum_{n=1}^\infty \frac {a_{\pm n}}{n x^{\pm n}} (1 - q_1^n)(1 - q_2^n)(1 - q^{-n}) \right]
\end{equation}
According to Theorem \ref{thm:shuf}, the $\BF$--algebras $\CA^\rightarrow$, $\CA^\leftarrow$ are generated by the elements $(z^d)^\rightarrow$, $(z^{d'})^\leftarrow$, respectively. Therefore, \eqref{eqn:comm3} gives an inductive recipe for expressing any product $R^\rightarrow * (R')^\leftarrow$ in terms of products of the form $(\bar{R}')^\leftarrow * a * \bar{R}^\rightarrow$ for various $\bar{R'}$, $\bar{R} \in \CS$ and $a \in \CA^{\diag}$. This underlies the decomposition \eqref{eqn:def double}. \\

\begin{remark}
\label{rem:second}

In most of the existing literature, the double shuffle algebra is defined as $\langle \CA,{c'}^{\pm 1} \rangle$, where $c'$ is a second central element which controls the commutation of the elements $a_n$ and $a_{-n}$. Since the central element $c'$ will act by 1 in all the representations considered in this paper, we will ignore it. \\

\end{remark}

\subsection{} 
\label{sub:shuffle generators} 

Since the double shuffle algebra $\CA$ contains $\CA^\leftarrow = \CS$ and $\CA^\rightarrow = \CS^{\op}$, we have two copies of each of the shuffle elements \eqref{eqn:pkd}--\eqref{eqn:qkd} inside it. Specifically, we denote them by:
\begin{align*}
&P_{-k, d} = P_{k, d}^\leftarrow, & &H_{-k, d} = H_{k, d}^\leftarrow, & &E_{-k, d} = E_{k, d}^\leftarrow, & &Q_{-k, d} = Q_{k, d}^\leftarrow  &\in \CA^\leftarrow \\
&P_{k, d} = P_{k, d}^\rightarrow, & &H_{k, d} = H_{k, d}^\rightarrow, & &E_{k, d} = E_{k, d}^\rightarrow, & &Q_{k, d} = Q_{k, d}^\rightarrow &\in \CA^\rightarrow 
\end{align*}
for all $k>0$ and $d\in \BZ$, and:
$$
P_{0,\pm d} = \pm (1-q_1^d)(1-q_2^d) a_{\pm d} \in \CA^{\diag}
$$
for all $d>0$. The elements $P_{k,d}$ for $(k,d) \in \BZ^2 \backslash (0,0)$ mimic those introduced by Burban and Schiffmann in the elliptic Hall algebra of \cite{BS}. More specifically, we claim that the elements $P_{k,d}$ introduced in the present paper satisfy the relations between the generators studied in \loccitt, namely \eqref{eqn:relation 1} and \eqref{eqn:relation 2} below:
\begin{equation}
\label{eqn:relation 1}
[P_{k,d}, P_{k',d'}] = \delta_{k+k'}^0 n(1-q_1^n)(1-q_2^n)\cdot \frac {1-c^{k}}{1-q^{-n}}
\end{equation}
if $kd'=k'd$ and $k<0$, where $n = \gcd(k,d)$. The second relation states that whenever $kd'>k'd$ and the triangle $T$ with vertices $(0,0), (k,d), (k+k',d+d')$ contains no lattice points inside nor on one of the edges, then we have the relation:
\begin{equation}
\label{eqn:relation 2}
[P_{k,d}, P_{k',d'}] = \frac {(1 - q_1^n) (1 - q_2^n)}{q^{-1} - 1} Q_{k+k',d+d'} \ \cdot \
\end{equation}
$$
\cdot \ \begin{cases}
c^{k} & \text{if } (k,d) \in \BZ^-, (k',d') \in \BZ^+, (k+k',d+d') \in \BZ^+ \\
c^{-k'} & \text{if } (k,d) \in \BZ^-, (k',d') \in \BZ^+, (k+k',d+d') \in \BZ^- \\
1 & \text{otherwise}
\end{cases}
$$
where $n = \gcd(k,d)\gcd(k',d')$ (by the assumption on the triangle $T$, we note that at most one of the pairs $(k,d), (k',d'), (k+k',d+d')$ can fail to be coprime), and we divide the lattice plane into its left and right halves:
$$
\BZ^+ = \{(k,d) \in \BZ^2 \text{ s.t. } k>0 \text{ or } k=0,d>0\}
$$
$$
\BZ^- = \{(k,d) \in \BZ^2 \text{ s.t. } k<0 \text{ or } k=0,d<0\}
$$
Note that the passage from our generators $P_{k,d}$ to the generators $u_{k,d}$ of \cite{BS} is: 
$$
P_{k,d} = (1-q_1^n)(1-q_2^n)c^{-\frac k2 \delta_{k>0}} u_{k,d}
$$
for all $(k,d) \in \mathbb{Z}^2 \backslash (0,0)$ with $n = \gcd(k,d)$. Moreover, their $\kappa_{k,d}$ equals our $c^{\frac k2}$. The contents of the present Subsection may thus be summarized by the following: \\

\begin{proposition}

\emph{(Follows by combining \cite{BS}, \cite{Shuf}, \cite{SV})}: Relations \eqref{eqn:relation 1} and \eqref{eqn:relation 2}  hold between the generators $\{P_{k,d}\}_{(k,d) \in \BZ^2 \backslash (0,0)}$ in the algebra $\CA$. Moreover, they generate the full ideal of relations, so one may alternatively define $\CA$ as the $\BF$--algebra generated by $\{P_{k,d}\}_{(k,d) \in \BZ^2 \backslash (0,0)}$ modulo relations \eqref{eqn:relation 1} and \eqref{eqn:relation 2}. \\

\end{proposition}

\subsection{}
\label{sub:other shuffle algebras}

Let us make several observations about the algebra $\CA$. First of all, for each rational slope $\frac ba$, relation \eqref{eqn:relation 1} implies that there exists a $_q$Heisenberg subalgebra:
$$
\CA^{\frac ba} = \BF \left \langle P_{an,bn}, c^{\pm 1} \right \rangle_{n\in \BZ \backslash 0} \subset \CA
$$ 
with central charge $c^{a}$. Moreover, the line of slope $\frac ba$ splits the lattice plane into two open half-planes, and so we will write:
\begin{equation}
\label{eqn:less}
\CA^{< \frac ba} = \BF \left \langle P_{k,d} \right \rangle_{kb < da}
\end{equation}
modulo relations \eqref{eqn:relation 2} between the generators $P_{k,d}$ for $kb<da$, and:
\begin{equation}
\label{eqn:more}
\CA^{> \frac ba} = \BF \left \langle P_{k,d} \right \rangle_{kb > da}
\end{equation}
modulo relations \eqref{eqn:relation 2} between the generators $P_{k,d}$ for $kb>da$. It was shown in \cite{BS} that the algebra $\CA$ exhibits a triangular decomposition:
\begin{equation}
\label{eqn:triangular}
\CA = \CA^{< \frac ba} \otimes \CA^{\frac ba} \otimes \CA^{> \frac ba}
\end{equation}
and the decomposition \eqref{eqn:def double} is precisely the case when $\frac ba = \frac 10$. \\

 


\begin{proposition}
\label{prop:half is shuffle}

There exist isomorphisms $\CA^{ < \frac ba} \cong \CS$ and $\CA^{ > \frac ba} \cong \CS^{\emph{op}}$ for all coprime $a,b$ with $a \geq 0$, and so the decomposition \eqref{eqn:triangular} reads:
\begin{equation}
\label{eqn:triangular 2}
\CA = \CS \otimes \left(q-\text{Heisenberg algebra}\right) \otimes \CS^{\emph{op}}
\end{equation}
Moreover, if we choose integers $a',b'$ such that $a'b - ab' = 1$, then the relations between the generators $P_{k,d}, P_{k',d'}, P_{k'',d''}$ with $kb < da$, $k'b = d'a$, $k''b > d''a$ are completely determined by the particular cases of relations \eqref{eqn:relation 1} and \eqref{eqn:relation 2} when the indices of the $P$'s that appear in the left-hand side of these relations are:
\begin{equation}
\label{eqn:apply}
(k,d) \in -(a',b') + \BZ(a,b), \ (k',d') \in \BZ(a,b), \ (k'',d'') \in (a',b') + \BZ(a,b) 
\end{equation}
$$$$
\end{proposition}

\begin{proof} Note that positive/negative multiples of $(a',b')$ plus integer multiples of $(a,b)$ span half of the lattice plane on either side of the line of slope $\frac ba$. The assignments:
\begin{align*}
&P_{-k(a',b')+d(a,b)} \mapsto P_{-k,d} \\
&P_{k(a',b')+d(a,b)} \mapsto P_{k,d} \cdot c^{(ka'+da)\delta_{ka'+da<0}}
\end{align*}
where $k$ ranges over $\BN$ and $d$ ranges over $\BZ$, induce isomorphisms:
$$
\CA^{<\frac ba} \cong \CA^\leftarrow = \CS \qquad \text{and} \qquad \CA^{>\frac ba} \cong \CA^\rightarrow = \CS^{\op}
$$
because relations \eqref{eqn:relation 2} between the generators $P_{\pm ka'+da, \pm kb'+db}$ are mapped to the exact same relations between the generators $P_{\pm k,d}$ (proving this is a straight-forward accounting of powers of $c$, and we leave it to the interested reader; note that the assumption $a \geq 0$ is important). This proves the first statement of the Proposition. Since the shuffle algebra is generated by its degree 1 part (Theorem \ref{thm:shuf}), then for any $k,k' > 0$ and $d,d' \in \BZ$, we may write:
\begin{align*}
&P_{-k(a',b')+d(a,b)} = \sum_{e_1,...,e_k \in \BZ} \text{constant} \cdot P_{-(a',b')+e_1(a,b)}... P_{-(a',b')+e_s(a,b)} \\
&P_{k'(a',b')+d'(a,b)} = \sum_{e''_1,...,e''_k \in \BZ} \text{constant} \cdot P_{(a',b')+e''_1(a,b)}... P_{(a',b')+e_s''(a,b)}
\end{align*}
Therefore, when expressed as an element of $\CA^{<\frac ba} \cdot \CA^{\frac ba} \cdot \CA^{>\frac ba}$, the commutator: 
$$
[P_{-k(a',b')+d(a,b)}, P_{k'(a',b')+d'(a,b)}]
$$
is an expression in $P_{-(a',b')+e(a,b)}$, $P_{(a',b')+e''(a,b)}$ and their commutators. This expression must equal the triangular decomposition \eqref{eqn:triangular} of the right-hand side of relations \eqref{eqn:relation 1}--\eqref{eqn:relation 2}, whenever the lattice points $(-k,d)$ and $(k',d')$ satisfy the hypotheses of these relations. Therefore, we conclude that the commutation relations \eqref{eqn:relation 1}--\eqref{eqn:relation 2} in general can be deduced from their particular cases \eqref{eqn:apply}.

\end{proof}

\subsection{}
\label{sub:vertical shuffle algebra}

In this paper, an important role will be played by the case $\frac ba = \frac 01$ of the decomposition \eqref{eqn:triangular}. The corresponding algebras \eqref{eqn:less}--\eqref{eqn:more} will be denoted:
$$
\CA^\uparrow := \CA^{<0} = \BF \left \langle P_{d,k} \right \rangle^{k>0}_{d \in \BZ} \subset \CA
$$
$$
\CA^\downarrow := \CA^{>0} = \BF \left \langle P_{d,k} \right \rangle^{k<0}_{d \in \BZ} \subset \CA
$$
and $\CA^0 = \BF \left \langle p_n,c^{\pm 1} \right \rangle_{n\in \BZ \backslash 0} \subset \CA$, where we set:
\begin{equation}
\label{eqn:def bosons}
p_{-n} = P_{-n,0}, \qquad p_n = \frac {c^n}{q^n} \cdot P_{n,0} 
\end{equation}
for all $n>0$. We interpret the symbols $p_{\pm n}$ as power sum functions, and think of: 
\begin{equation}
\label{eqn:abbreviations}
h_{-n} = H_{- n,0}, \qquad h_{n} =\frac {c^n}{q^n} \cdot H_{n,0}
\end{equation}
as the corresponding complete symmetric functions. The interaction between the positive and negative $_q$Heisenberg algebra generators is governed by:
\begin{equation}
\label{eqn:heisenberg}
[p_{-n},p_{n}] = n(1-q_1^n)(1-q_2^n) \cdot \frac {1-c^{n}}{1-q^{n}} \qquad \forall \ n>0
\end{equation}
As an application of Proposition \ref{prop:half is shuffle}, we have an isomorphism $\CA^\uparrow \cong \CS$. In other words, to any shuffle element $R \in \CS$, there corresponds an element $R^\uparrow \in \CA^\uparrow$, determined inductively by Theorem \ref{thm:shuf} and the initial assignment:
\begin{equation}
\label{eqn:identify}
(z^d)^\uparrow = P_{d,1} \qquad \forall d\in \BZ
\end{equation}
Finally, we will encounter the subalgebra  generated by both $\CA^\uparrow$ and $\CA^0$: 
$$
\CA^{\uparrow \ext} =  \BF \left \langle P_{d,k}, c^{\pm 1} \right \rangle^{k \geq 0}_{d \in \BZ} \subset \CA
$$
The superscript ext stands for ``extended", and will be used repeatedly throughout the paper, when we will ``extend" various algebras by adding the $_q$Heisenberg algebra generators \eqref{eqn:def bosons}, modulo certain relations defined on a case-by-case basis. \\

\subsection{}
\label{sub:rep theory}

We call a representation $\CA \curvearrowright F$ \textbf{good} if the following conditions are met: \\

\begin{enumerate}
 
\item $F = \bigoplus_{n = 0}^\infty F_n$ is non-negatively graded \\

\item $P_{k,d} \in \CA$ acts on $F$ with degree $-k$ for all $(k,d) \in \BZ^2 \backslash (0,0)$ \\

\end{enumerate}

\begin{definition}
\label{def:level}

The level of an irreducible $\CA$--module $F$ is the constant $\log_q c$. 

\end{definition}

\tab 
All the good representations we will encounter in this paper will have level $\in \BN$, and will moreover be quasifinite in the terminology of \cite{FJMM} (we refer the reader to \loccit for a thorough treatment of the representation theory of $\CA$). In \cite{BS} and \cite{Shuf}, it is shown that a linear basis of $\CA^\uparrow \cong \CS$ is given by the products:
\begin{equation}
\label{eqn:compositions}
P_v = P_{d_1,k_1} ... P_{d_t,k_t} 
\end{equation}
where: 
\begin{equation}
\label{eqn:sequence}
v = \Big\{ (d_1,k_1),...,(d_t,k_t), \ k_i \in \BN, d_i \in \BZ \Big\}
\end{equation}
denotes an arbitrary sequence of lattice points, ordered such that:
\begin{equation}
\label{eqn:sequence inequality}
\frac {d_1}{k_1} \leq ... \leq \frac {d_t}{k_t}
\end{equation}
By convention, if we have $d_i/k_i = d_{i+1}/k_{i+1}$ for some $i$, then we place $(d_i,k_i)$ before $(d_{i+1},k_{i+1})$ in the ordering if and only if $k_i<k_{i+1}$ (this convention is not crucial at all, because of \eqref{eqn:relation 1}). Property (2) of a good representation $F$ implies that among the infinitely many compositions \eqref{eqn:compositions} with $k_1,...,k_t > 0$ and fixed $k_1+...+k_t$, {\bf only finitely many of them} have $\langle f'|P_v|f \rangle \neq 0$ for any pair of elements $f,f' \in F$. This implies that the action $\CA \curvearrowright F$ extends to an action of the completion:
\begin{equation}
\label{eqn:completion}
\wCA^\uparrow \curvearrowright F
\end{equation}
where we define:
\begin{equation}
\label{eqn:def completion}
\wCA^\uparrow = \Big \{ \text{infinite sums of }P_v\text{'s as in \eqref{eqn:compositions}, for bounded }k_1+...+k_t \in \BN \Big \}
\end{equation}
Analogous remarks apply to the completion $\wCA^{\uparrow \ext}$ of $\CA^{\uparrow \ext}$, which we define to be spanned by infinite sums of the form:
\begin{equation}
\label{eqn:nanana}
P_{d_1,k_1} ... P_{d_t,k_t} 
\end{equation}
going over all $d_i \in \BZ$, $k_i \geq 0$ such that $k_1+...+k_t$ is bounded above, as well as:
$$
- \infty \leq \frac {d_1}{k_1} \leq ... \leq \frac {d_t}{k_t} \leq \infty
$$
(if $k_i = k_{i+1} = 0$, then we place $P_{d_i,k_i}$ before $P_{d_{i+1},k_{i+1}}$ in \eqref{eqn:nanana} iff $d_i < d_{i+1}$). \\

\subsection{}
\label{sub:fock}

The basic representation of the algebra $\CA$ is the Fock space:
$$
F_u = \BF [p_{-1},p_{-2},...]
$$
with action given by $c = q$, as well as:
\begin{equation}
\label{eqn:vertex 0}
p_{-n} ( m ) = p_{-n} \cdot m, \qquad p_n ( m ) = - n(1-q_1^n)(1-q_2^n) \cdot \partial_{p_{-n}} (m)
\end{equation}
while the actions of $\CA^\uparrow$ and $\CA^\downarrow$ are completely determined by the formulas:
\begin{equation}
\label{eqn:vertex 1}
\sum_{d \in \BZ} \frac {P_{d,1}}{x^d} = u \cdot \exp \left[ \sum_{n=1}^\infty \frac {p_{-n}}{nx^{-n}} \right] \exp \left[ \sum_{n=1}^\infty \frac {p_n}{nx^n} \right]
\end{equation}
\begin{equation}
\label{eqn:vertex 2}
\sum_{d \in \BZ} \frac {P_{d,-1}}{ q^{d\delta_{d<0}}x^d} = \frac qu \cdot \exp \left[ - \sum_{n=1}^\infty \frac {p_{-n}}{nx^{-n}} \right] \exp \left [- \sum_{n=1}^\infty \frac {p_n}{nx^nq^n} \right] 
\end{equation}
Indeed, the identifications $\CA^\uparrow \cong \CS$ and $\CA^\downarrow \cong \CS^{\op}$ of Proposition \ref{prop:half is shuffle} when $\frac ba = \frac 01$, together with Theorem \ref{thm:shuf}, establish the fact that the algebras $\CA^\uparrow$ and $\CA^\downarrow$ are generated by $P_{d,1}$ and $P_{d,-1}$, respectively. Therefore, formulas \eqref{eqn:vertex 0}--\eqref{eqn:vertex 2} determine the entire action $\CA \curvearrowright F_u$. We must now show this action is well-defined. \\

\begin{proposition}
\label{prop:fock act}

Formulas \eqref{eqn:vertex 0}, \eqref{eqn:vertex 1}, \eqref{eqn:vertex 2} define an action $\CA \curvearrowright F_u$, in such a way that $F_u$ is a good level 1 representation. \\

\end{proposition}

\begin{proof} It is clear that the operators \eqref{eqn:vertex 0} give rise to an action $\CA^0 \curvearrowright F_u$, because relation \eqref{eqn:heisenberg} holds for $c=q$. Moreover, we may define an action:
$$
\CA^\uparrow \cong \CS \curvearrowright F_u
$$
which to any shuffle element $R(z_1,...,z_k) \in \CS$ associates the operator:
\begin{equation}
\label{eqn:vertex 3}
R^\uparrow = u^k \oint ... \oint \frac {R(z_1,...,z_k)}{\prod_{1\leq i \neq j \leq k} \zeta \left( \frac {z_i}{z_j} \right)} \exp \left[ \sum_{i=1}^k \sum_{n=1}^\infty \frac {p_{-n}}{nz_i^{-n}} \right] \exp \left[ \sum_{i=1}^k \sum_{n=1}^\infty \frac {p_n}{nz_i^n} \right] 
\end{equation}
The contours of integration are defined to be $|z_i| = 1$ for all $1 \leq i \leq k$, and we make the assumption $|q_1|,|q_2|<1$ in order to avoid poles on the contours of integration from the denominators of $R$ and $\zeta$. We leave the following observations as exercises to the interested reader (see \cite{Ops} for details): \\

\begin{itemize}

\item the composition of the operators \eqref{eqn:vertex 3} respects the shuffle product \eqref{eqn:mult}: 
$$
R_1^\uparrow \circ R_2^\uparrow = (R_1 * R_2)^\uparrow
$$

\item when $R = z_1^d$, formula \eqref{eqn:vertex 3} specializes to \eqref{eqn:vertex 1}. \\

\end{itemize}

\noindent Therefore, \eqref{eqn:vertex 1} induces an action $\CA^\uparrow \curvearrowright F_u$. Analogously, \eqref{eqn:vertex 2} induces $\CA^\downarrow \curvearrowright F_u$. To prove that these actions give rise to an action of:
$$
\CA = \CA^\uparrow \otimes \CA^0 \otimes \CA^\downarrow
$$
on $F_u$, we must show that the relations between the generators of the three tensor factors are respected. According to the last statement of Proposition \ref{prop:half is shuffle}, it is enough to check relations \eqref{eqn:relation 1} and \eqref{eqn:relation 2} between the generators $P_{d,1}$, $P_{d,0}$, $P_{d,-1}$ for all $d \in \BZ$. Specifically, these relations state:
\begin{equation}
\label{eqn:rel1}
[P_{d,1}, P_{\pm e,0}] = \pm (1-q_1^e)(1-q_2^e) P_{d\pm e,1}
\end{equation}
\begin{equation}
\label{eqn:rel2}
\left[ \frac {P_{d,-1}}{q^{d\delta_{d<0}}}, \frac {P_{\pm e,0}}{q^{-e\delta_{\pm e <0}}} \right] = \mp (1-q_1^e)(1-q_2^e) \frac {P_{d \pm e,-1}}{q^{(d\pm e)\delta_{d\pm e<0}}}
\end{equation}
for all $d,d' \in \BZ$ and $e \in \BN$, as well as:
\begin{equation}
\label{eqn:rel3}
[P_{d,1}, P_{d',-1}] = \frac {(1-q_1)(1-q_2)}{1-q^{-1}} \left( \delta_{d+d' \geq 0} \frac {Q_{d+d',0}}{q^{-d' \delta_{d'<0}}} - \delta_{d+d' \leq 0} \frac {Q_{d+d',0}}{q^{d' \delta_{d'>0}}} \right)
\end{equation}
where the operators $Q_{\pm n,0}$ are exponentials of the operators $p_{\pm n}$, given by:
$$
\sum_{n=0}^\infty \frac {Q_{ \pm n,0}}{x^{\pm n}} = \exp \left[ \sum_{n=1}^\infty \frac {p_{\pm n}}{nx^{\pm n}}(1-q^{-n}) \right]
$$
Apply \eqref{eqn:vertex 1} and \eqref{eqn:heisenberg} to obtain:
$$
\left[\sum_{d\in \BZ} \frac {P_{d,1}}{x^d}, P_{-e,0} \right] = \left[ u \cdot \exp \left( \sum_{n=1}^\infty \frac {p_{-n}}{nx^{-n}} \right) \exp \left( \sum_{n=1}^\infty \frac {p_n}{nx^n} \right), p_{-e} \right]  = 
$$
$$
= \frac {[p_e,p_{-e}]}{e x^e} \cdot u \exp \left( \sum_{n=1}^\infty \frac {p_{-n}}{nx^{-n}} \right) \exp \left( \sum_{n=1}^\infty \frac {p_n}{nx^n} \right)  = -(1-q_1^e)(1-q_2^e) x^{-e} \sum_{d\in \BZ} \frac {P_{d,1}}{x^d} \cdot 
$$
Picking out the coefficient of $x^{-d}$ gives precisely \eqref{eqn:rel1} when the sign is $\pm = -$. The case when the sign is $\pm = +$, as well as relation \eqref{eqn:rel2}, are analogous and we leave them as an exercise to the interested reader. As for \eqref{eqn:rel3}, we have:
\begin{equation}
\label{eqn:equality}
\sum_{d\in \BZ} \frac {P_{d,1}}{x^d} \cdot \sum_{d' \in \BZ} \frac {P_{d',-1}}{q^{d' \delta_{d'<0}}y^{d'}}  = 
\end{equation}
$$
= q \exp \left[ \sum_{n=1}^\infty \frac {p_{-n}}{nx^{-n}} \right] \exp \left[ \sum_{n=1}^\infty \frac {p_n}{nx^n} \right] \exp \left[ - \sum_{n=1}^\infty \frac {p_{-n}}{ny^{-n}} \right] \exp \left [- \sum_{n=1}^\infty \frac {p_n}{n(yq)^n} \right] =
$$ 
$$
= q \exp \left[ \sum_{n=1}^\infty \frac {p_{-n}}n (x^n - y^n) \right] \exp \left [\sum_{n=1}^\infty \frac {p_n}n \left( \frac 1{x^n} - \frac 1{(yq)^n} \right) \right] \frac {(x - yq_1)(x - yq_2)}{(x - y)(x - yq)}
$$
where in the final equality we used the identity: 
$$
\exp(a)\exp(b) = \exp(b)\exp(a)\exp([a,b]),
$$
that holds whenever the commutator $[a,b]$ commutes with both $a$ and $b$. The product of exponentials in the third line of \eqref{eqn:equality} is normally-ordered, which means that its matrix coefficients in $F_u$ are Laurent polynomials in $x$ and $y$. Meanwhile, the product of exponentials on the second line only gives a well-defined operator on $F_u$ only if we expand the power series in the region $|y| \ll |x|$ (see Remark \ref{rem:currents}). We conclude that the equality \eqref{eqn:equality} makes sense if we expand the rational function:
$$
\zeta \left(\frac yx \right) = \frac {(x - yq_1)(x - yq_2)}{(x - y)(x - yq)}
$$
in non-negative powers of $y/x$. Therefore, \eqref{eqn:equality} implies:
\begin{equation}
\label{eqn:integral}
P_{d,1}P_{d',-1} =  q^{1+d' \delta_{d'<0}} \int_{|y| \ll |x|} x^d y^{d'} \zeta\left( \frac yx \right) \exp \left[ \dots \right] \exp \left [ \dots \right] \frac {dx}{2\pi i x} \frac {dy}{2\pi i y}
\end{equation}
where $\exp \left[ \dots \right] \exp \left [ \dots \right]$ denotes the normal-ordered product of exponentials on the third line of \eqref{eqn:equality}. By definition, the contours over which the variables $x$ and $y$ run in \eqref{eqn:integral} are circles centered at the origin, with the former of much larger radius than the latter. Similarly, the product $P_{d',-1} P_{d,1}$ is given by the same integrand as \eqref{eqn:integral}, but the contours respect the inequality $|x| \ll |y|$. We conclude that:
\begin{equation}
\label{eqn:chet}
[P_{d,1}, P_{d',-1}] =  q^{1+d' \delta_{d'<0}} \int_{|y| \ll |x|} - \int_{|x| \ll |y|} 
\end{equation}
$$
x^d y^{d'} \zeta\left( \frac yx \right) \exp \left[ \sum_{n=1}^\infty \frac {p_{-n}}n (x^n - y^n) \right] \exp \left [\sum_{n=1}^\infty \frac {p_n}n \left( \frac 1{x^n} - \frac 1{(yq)^n} \right) \right] \frac {dx}{2\pi i x} \frac {dy}{2\pi i y}
$$
By the residue theorem, the value of \eqref{eqn:chet} is equal to the sum of the residues of the integrand at poles of the form $x = y \alpha$, for various constants $\alpha \notin \{0,\infty\}$. Since the product of exponentials is normally-ordered, its matrix coefficients in $F_u$ are Laurent polynomials in $x$ and $y$, and therefore do not produce poles at $x = y\alpha$. And in fact, the only poles we encounter in the integral \eqref{eqn:chet} come from $\zeta$, and they are $x = y$ and $x = yq$. The corresponding sum of residues is:
$$
[P_{d,1}, P_{d',-1}] = q^{1+d' \delta_{d'<0}} \frac {(1-q_1)(1-q_2)}{1-q} \cdot
$$
$$
\left( \oint y^{d+d'} q^{d} \exp \left[\sum_{n=1}^\infty \frac {p_{-n}}{ny^{-n}} (q^n-1) \right] \frac {dy}{2\pi i y} -  \oint y^{d+d'} \exp \left[\sum_{n=1}^\infty \frac {p_n}{ny^n} (1-q^{-n}) \right] \frac {dy}{2\pi i y} \right)
$$
Note that the expression above precisely matches the right-hand side of \eqref{eqn:rel3}. The fact that $F_u$ is good is clear, since the two conditions of Subsection \ref{sub:rep theory} are easily seen to be respected. That the level is 1 is precisely restating the fact that $c = q$.

\end{proof}

\subsection{}
\label{sub:q-W currents}




Since $z^{d\uparrow} = P_{d,1}$ are the images of the degree 1 shuffle elements under $\CS \cong \CA^\uparrow$, the left-hand side of \eqref{eqn:vertex 1} plays an important role in the double shuffle algebra $\CA$:
\begin{equation}
\label{eqn:w1}
W_1(x) := \delta \left( \frac {z}x \right)^\uparrow = \sum_{d \in \BZ} \frac {{z^d}^\uparrow}{x^d} \in \CA^\uparrow[[x^{\pm 1}]]
\end{equation}
will be called the ``first $_qW$--current". The reason for the terminology in quotes is that we will now define ``higher $_qW$--currents" in terms of the double shuffle algebra $\CA$, and in Proposition \ref{prop:w} will prove that they satisfy the relations in \cite{AKOS}, \cite{FF}:\footnote{Strictly speaking, our currents differ from those of \loccit by certain exponentials in the bosons $p_{\pm n}$, which we explicitly give in \eqref{eqn:jon}. The difference, as we will see in Section \ref{sec:miura}, stems from the fact that our $_qW$--algebra corresponds to $\fgl_r$, while that of \loccit corresponds to $\fsl_r$}
\begin{equation}
\label{eqn:w k}
W_k(x) = \eta_k \cdot \sym \left[ \delta \left(\frac {z_1}x \right) ... \delta \left( \frac {z_k}{xq^{1-k}} \right) \right]^\uparrow \in \wCA^\uparrow[[x^{\pm 1}]]
\end{equation}
where:
$$
\eta_k = \prod_{1\leq i < j \leq k} \zeta(q^{j-i})
$$
We set $W_0(x) = 1$ by convention, and often use the following notation:
\begin{equation}
\label{eqn:modes}
W_k(x) = \sum_{d\in \BZ} \frac {W_{d,k}}{x^d}
\end{equation}
for the coefficients of $W_k(x)$ as elements of $\wCA^\uparrow$. The fact that the $W_{d,k}$ lie in the completion of $\CA^\uparrow$ is not immediately obvious, but is implied by the following result (together with \eqref{eqn:def for l}, \eqref{eqn:def for u}, \eqref{eqn:lower upper} and the definition of the completion in \eqref{eqn:def completion}): \\



\begin{proposition}
\label{prop:explicit shuffle}

In the completion $\wCA^\uparrow[[x^{\pm 1}]]$, we have the relation:
\begin{equation}
\label{eqn:explicit shuffle}
W_k(x) =  \mathop{\sum^{k_\leftarrow, k_0, k_\rightarrow \geq 0}_{k_\leftarrow + k_0 + k_\rightarrow = k}}^{d_\leftarrow,d_\rightarrow \geq 0} \frac {T_{d_\leftarrow,k_\leftarrow}^\leftarrow E_{0,k_0} T_{d_\rightarrow, k_\rightarrow}^\rightarrow}{x^{d_\rightarrow - d_\leftarrow}} \cdot  q^{(k-1)d_\rightarrow}
\end{equation}
where we set $T_{0,k}^\leftarrow = T_{0, k}^\rightarrow = T_{k,0}^\leftarrow = T_{k,0}^\rightarrow = \delta_k^0$, while for $d,k>0$ we define:
\begin{equation}
\label{eqn:awesome shuffle}
T_{d,k}(z_1,...,z_d) = \esym \left[ \frac {(-1)^{k-1} z_d^k}{\prod_{i=1}^{d-1} \left(1 - \frac {qz_{i+1}}{z_i} \right)} \prod_{1 \leq i < j \leq d} \zeta \left( \frac {z_i}{z_j} \right) \right] \in \CS
\end{equation}
$$$$

\end{proposition}

\noindent Formula \eqref{eqn:explicit shuffle} is more effectively packaged by computing the generating series of \textbf{all} the $_qW$--currents at the same time. Specifically, define the power series:
\begin{equation}
\label{eqn:big current}
W(x,y) = \sum_{k=0}^\infty \frac {W_k(x)}{(-y)^k} = 1 + \sum_{d\in \BZ} \sum_{k=1}^\infty \frac {W_{d,k}}{x^d (-y)^k} 
\end{equation}
If we let $D_x$ denote the difference operator $f(x) \leadsto f(xq)$, then \eqref{eqn:explicit shuffle} implies:
\begin{equation}
\label{eqn:ldu}
W(x,yD_x) = T \left(x^{-1}, y D_x \right)^\leftarrow \cdot E\left( y D_x \right) \cdot T\left(xq, y D_x \right)^\rightarrow
\end{equation}
which is precisely formula \eqref{eqn:w currents}, where we set:
\begin{equation}
\label{eqn:awesome shuffle 1}
T(x,y) = 1 + \sum_{d=1}^\infty \sym \left[ \frac {x^{-d}}{\left(1 - \frac y{z_d} \right)\prod_{i=1}^{d-1} \left(1 - \frac {qz_{i+1}}{z_i} \right)} \prod_{1 \leq i < j \leq d} \zeta \left( \frac {z_i}{z_j} \right) \right]
\end{equation}
\begin{equation}
\label{eqn:awesome shuffle 2}
E(y) = \sum_{k=0}^\infty \frac {E_{0,k}}{(-y)^k} \ = \ \exp \left[- \sum_{n=1}^\infty \frac {a_n (1-q_1^n)(1-q_2^n)}{ny^n} \right] 
\end{equation}
\text{} \\

\begin{remark}
\label{rem:paragraph}

Note that \eqref{eqn:ldu} is imprecise because $x$ and $D_x$ do not commute. In order for this formula to match \eqref{eqn:explicit shuffle}, one needs to interpret it as follows: in the expansion of $W(x,yD_x)$ and $T(x^{-1}, y D_x)^\leftarrow$ (respectively $T(xq,y D_x)^\rightarrow$), place the powers of $D_x$ after (respectively before) the powers of $x$. We call \eqref{eqn:ldu} the LDU decomposition of $W$--currents, because the three factors in the right-hand side respectively increase, preserve and decrease degree in any good representation. \\

\end{remark}

\noindent In Section \ref{sec:geom}, we will give explicit formulas for the matrix coefficients of the decomposition \eqref{eqn:explicit shuffle} in the level $r$ representation $K$ that is the Hilbert space of the gauge theory side of the AGT--W correspondence. The fact that $W_k(x) = 0$ holds in $K$ for all $k>r$ is not apparent from \eqref{eqn:explicit shuffle}, and will be proved later. \\

\subsection{}
\label{sub:q-W algebra}

Proposition \ref{prop:explicit shuffle} will be proved in the Appendix, together with the following: \\

\begin{proposition}
\label{prop:explicit shuffle 2}

Consider the formal series:
\begin{equation}
\label{eqn:def explicit 2}
W_k(x) * W_{k'}(y) = \eta_k \eta_{k'} \prod^{1\leq i \leq k}_{1 \leq j \leq k'} \zeta \left(\frac {xq^j}{yq^i} \right) \cdot
\end{equation}
$$
\emph{Sym} \left[ \delta \left(\frac {z_1}x \right) ... \delta \left( \frac {z_k}{xq^{1-k}} \right) \delta \left(\frac {z_{k+1}}y \right) ... \delta \left( \frac {z_{k+k'}}{yq^{1-k'}} \right) \right]^\uparrow  
$$
Then $W_k(x) * W_{k'}(y)$ lies in $\wCA^\uparrow[[x^{\pm 1}, y^{\pm 1}]] \otimes (\text{a certain explicit rational function})$:
\begin{equation}
\label{eqn:explicit shuffle 2}
W_k(x) * W_{k'}(y)  =  \frac {\prod^{k-1}_{i = \max(0,k-k')} \zeta \left( \frac {y q^i}x \right)^{-1} \sum_v P_v \cdot a_v(x,y)}{\prod^{k}_{i=\max(0,k-k')+1} (x-yq^{i}) \prod_{i=\max(0,k'-k)+1}^{k'} (y-xq^i)}
\end{equation}
where $a_v(x,y)$ are certain Laurent polynomials that will not matter to us. In the numerator of \eqref{eqn:explicit shuffle 2}, the sum goes over all ordered sequences $v$ as in \eqref{eqn:sequence}. 

\end{proposition}

\tab
Since all representations considered in this paper are good in the sense of Subsection \ref{sub:rep theory}, only finitely many of the shuffle elements $P_v$ will contribute to any given matrix coefficient. Therefore, Proposition \ref{prop:explicit shuffle 2} implies that in a good representation, the operator--valued expression: 
$$
W_k(x) W_{k'}(y) \prod^{k-1}_{i = \max(0,k-k')} \zeta \left( \frac {y q^i}x \right)
$$
is a rational function whose only poles are $x^{\pm 1}$, $y^{\pm 1}$, and $x-yq^*$, $y-xq^{**}$ for:
\begin{equation}
\label{eqn:poles 1}
*\in \{\max(0,k-k')+1,...,k\}
\end{equation}
\begin{equation}
\label{eqn:poles 2}
**\in \{\max(0,k'-k)+1,...,k'\}
\end{equation}
Since the second line of \eqref{eqn:def explicit 2} is symmetric in $(k,x) \leftrightarrow (k',y)$, we will use the information of the poles \eqref{eqn:poles 1}--\eqref{eqn:poles 2} to prove the commutation relations \eqref{eqn:w rel}. \\

\begin{proposition}
\label{prop:w}

The currents $\{ W_k(x)\}_{k \geq 1}  \in \wCA^\uparrow[[x^{\pm 1}]]$ satisfy the relations:
\begin{equation}
\label{eqn:w rel 0 minus}
\left[W_k(x), p_{- n} \right] = - \frac {(1-q_1^n)(1-q_2^n)(1-q^{kn})}{1-q^{n}} \cdot x^{-n} W_k(x)
\end{equation}
\begin{equation}
\label{eqn:w rel 0 plus}
\left[W_k(x), p_{n} \right] = \frac {(1-q_1^n)(1-q_2^n)(q^{-kn} - 1)c^n}{1-q^{n}} \cdot x^{n} W_k(x)
\end{equation}
where $p_n \in \CA^0 \subset \CA$ are the bosons of \eqref{eqn:heisenberg}, and:
\begin{equation}
\label{eqn:w rel}
W_k(x) W_{k'}(y) \cdot f_{kk'}\left(\frac yx \right) - W_{k'}(y) W_k(x) \cdot f_{k'k}\left(\frac xy \right) \ = \
\end{equation}
$$
= \sum_{i=\max(0,k'-k)+1}^{k'} \delta\left(\frac {y}{xq^i} \right) \left[ W_{k'-i}(x) W_{k+i}(y) f_{k'-i,k+i}\left(\frac yx \right) \Big|_{x = \frac y{q^i}} \right] \theta(\min(i,k-k'+i))
$$
$$
- \sum^{k}_{i = \max(0,k-k')+1} \delta\left(\frac {x}{yq^i} \right) \left[ W_{k-i}(y) W_{k'+i}(x) f_{k-i,k'+i}\left(\frac xy \right) \Big|_{y = \frac x{q^i}} \right] \theta(\min(i,k'-k+i))
$$
where we set:
\begin{equation}
\label{eqn:def theta}
\theta(s) = \frac {(1-q_1)(1-q_2)}{1-q}\cdot \zeta(q)...\zeta(q^{s-1})
\end{equation}
and for all $k,k' \geq 0$ define the rational function:
\begin{equation}
\label{eqn:def f}
f_{kk'} (z) =  \exp \left[\sum_{n=1}^\infty \frac {z^n}n \cdot \frac {(1-q_1^n)(1-q_2^n)(q^{\max(0,k-k')n}-q^{kn})}{1-q^n}  \right]
\end{equation}
$$$$


\end{proposition}


\begin{remark}
\label{rem:currents}

Let us give a hands--on explanation of the convention of normal--ordering products of currents. All currents in this paper will be of the form:
$$
W(x) = \sum_{d\in \BZ} \frac {W_d}{x^d}
$$
where $W_d$ is an operator that acts in good representations with degree $-d$. Therefore, the operators $W_d$ for $d \gg 0$ will annihilate any given vector of a good representation, and so the only reasonable way to make sense of compositions such as:
$$
W(x) W'(y) f \left (\frac yx \right) = \sum_{d,d' \in \BZ} \frac {W_d W_{d'}'}{x^d y^{d'}} \cdot \sum_{n \geq 0} f_n \frac {y^n}{x^n}
$$
is to expand the analytic function $f(z)$ in non--negative powers of $z$. This way, the coefficient of any $x^ay^b$ in the right-hand side acts by a finite sum locally in all good representations. This is why the first summand of \eqref{eqn:w rel} must be expanded in the region $|y| \ll |x|$, while the second summand must be expanded in $|x| \ll |y|$. \\

\end{remark}

\begin{remark}
\label{rem:contours}

We may explain relation \eqref{eqn:w rel} in terms of contour integrals (see \cite{O} for more details). Namely, suppose one wanted to compute the coefficient of $x^a y^b$ in the two sides of the equation. Naively, the way to do so is consider:
\begin{equation}
\label{eqn:equality  of integrals}
\oint \oint \frac {\text{LHS of \eqref{eqn:w rel}}}{x^a y^b} \cdot \frac {dx dy}{xy} = \oint \oint \frac {\text{RHS of \eqref{eqn:w rel}}}{x^a y^b} \cdot \frac {dx dy}{xy}
\end{equation}
However, the devil is in the contours: in the LHS, the term $W_k(x) W_{k'}(y) f_{kk'}(\frac yx)$ must be integrated over $|x| \gg |y|$, while the term $W_{k'}(y) W_k(x)  f_{k'k}(\frac xy)$ must be integrated over $|x| \ll |y|$. Meanwhile, every summand in the right-hand side of \eqref{eqn:w rel} reduces to a single contour integral, because for any constant $a$ we have:
$$
\oint \oint \delta\left(\frac y{xa} \right) Z(x,y) \frac {dx dy}{xy} = \oint Z(x,xa) \frac {dx}x
$$
Therefore, formula \eqref{eqn:equality of integrals} states that moving the $y$ contour from $|x| \gg |y|$ to $|x| \ll |y|$ picks up precisely the residues that make up the right-hand side of \eqref{eqn:w rel}. \\

\end{remark}

\subsection{}
\label{sub:verma}

Note that relation \eqref{eqn:w rel} matches the $_qW$--algebra relations of \cite{AKOS}, \cite{FF}, modulo the fact that our currents and theirs are not quite equal. Instead, the precise relation between the two conventions is given in \eqref{eqn:jon}, and this distinction accounts for the fact that our function $f_{kk'}(z)$ is different from that of \cite{AKOS}. \\

\begin{definition} 
\label{def:w}  

Define the algebra $\CA_r$ as the quotient of $ \wCA^{\uparrow}$ by the relations:
\begin{equation}
\label{eqn:impose rels}
W_k(x) = 0
\end{equation}
for all $k>r$. Define the algebra $\CA_r^\emph{ext}$ as the quotient of $\wCA^{\uparrow \emph{ext}}$ by the relations:
\begin{equation}
\label{eqn:impose rels bis}
W_r(x) = u \left[ \sum_{n=0}^\infty \frac {h_{-n}}{x^{-n}} \right] \left[\sum_{n=0}^\infty \frac {h_n}{x^n} \right] \qquad \text{and} \qquad W_k(x) = 0
\end{equation}
for all $k>r$, together with $c = q^r$. We call $\CA_r$ the $_qW$--algebra of type $\fgl_r$. \\

\end{definition}


\noindent The parameter $u$ plays the role of the zero modes of Heisenberg currents from the physics literature (see Section \ref{sec:miura}), and we will abuse notation by ignoring it from our notation. The reader can think of $u$ as a complex number which will be specialized to the product of equivariant parameters in the representations $K$ studied in the next Section. Note that the inclusion:
$$
\wCA^\uparrow \hookrightarrow \wCA^{\uparrow \ext}
$$ 
induces a homomorphism of algebras:
\begin{equation}
\label{eqn:homo}
\CA_r \rightarrow \CA_r^\ext
\end{equation}
Although we will not need this fact, we expect this homomorphism to be injective. Note that it is also ``almost" surjective, or it would be if one were able to express the $_q$Heisenberg generators $p_n$ in terms of the coefficients of the current: 
$$
W_r(x) = u \exp \left[ \sum_{n=1}^\infty \frac {p_{-n}}{n x^{-n}} \right] \exp \left[\sum_{n=1}^\infty \frac {p_n}{nx^n} \right] \in \CA^\ext_r
$$
This would be possible only if we dropped the phrase ``for bounded $k_1+...+k_t \in \BN$" from the definition of the completion in \eqref{eqn:def completion}. Since doing so would be a technical annoyance for us, we accept the fact that the algebras $\CA_r$ and $\CA_r^\ext$ are a little different, although we note that both act by the same formulas on the representations $K$ defined in the next Section. Another feature of this ``equivalence", which we leave as an easy exercise to the interested reader, is that relations \eqref{eqn:w rel 0 minus} and \eqref{eqn:w rel 0 plus} imply \eqref{eqn:w rel} for arbitrary $k$ and $k' = r$ (the fact that $W_l(x) = 0$ for $l>r$ in either $\CA_r$ or $\CA_r^\ext$ implies that the RHS of \eqref{eqn:w rel} vanishes). \\

\begin{proposition} 
\label{prop:w gens and rels}

The algebra $\CA_r$ is isomorphic to the $\BF$--algebra generated by symbols $\{W_{d,k}\}_{k\geq 1}^{d\in \BZ}$ modulo relations \eqref{eqn:w rel} and \eqref{eqn:impose rels}. \\

\end{proposition}

\begin{remark}

Similarly, the algebra $\CA^{\emph{ext}}_r$ is isomorphic to the $\BF$--algebra generated by symbols $\{W_{d,k}\}_{k\geq 1}^{d\in \BZ}$ and $\{p_n\}_{n\in \BZ \backslash 0}$ modulo relations \eqref{eqn:w rel 0 minus}, \eqref{eqn:w rel 0 plus}, \eqref{eqn:w rel}, \eqref{eqn:impose rels bis}. \\

\end{remark}


\noindent We only sketch the main ideas of the proof, and leave the technical details to \cite{W gen}. The reader who does not wish to rely on the proof may replace $\CA_r$ of Definition \ref{def:w} by the algebra $\CA_r'$ generated by $W_{d,k}$ modulo relations \eqref{eqn:w rel} and \eqref{eqn:impose rels}.  The constructions in the present paper give a map $\CA_r' \rightarrow \CA_r$, which the results of \loccit imply is an isomorphism. All modules in the present paper are naturally $\CA_r$ --modules, and everything in the present paper continues to apply to them as such. \\

\begin{proof}  By definition, $\CA_r$ is spanned as an $\BF$--module by all possible products:
\begin{equation}
\label{eqn:gente}
\Big\{ W_{d_1,k_1}...W_{d_t,k_t}, \ d_i \in \BZ, k_i \in \BN \Big\}
\end{equation}
The product \eqref{eqn:gente} may be characterized by the following \textbf{path} (broken line):
$$
(0,0),(d_1,k_1),\dots,(d_1+...+d_t,k_1+...+k_t)
$$
In \cite{W gen}, we will unpackage relation \eqref{eqn:w rel} using the normal-ordering conventions of \cite{O}, and show that it allows us to write the general product \eqref{eqn:gente} as a (possibly infinite) linear combination of specific products which satisfy:
$$
\frac {d_1}{k_1} \leq \frac {d_2}{k_2} \leq ... \leq \frac {d_t}{k_t}
$$
In the language of paths, such products correspond to convex lattice paths in the upper half-plane. In other words, repeated applications of \eqref{eqn:w rel} allows one to ``convexify" any path in the upper half-plane. Finally, relation \eqref{eqn:impose rels} shows that we may restrict to the products \eqref{eqn:gente} for which $k_1,..., k_t \leq r$. To show that there are no relations among $W_{d,k} \in \CA_r$ except for \eqref{eqn:w rel} and \eqref{eqn:impose rels}, it is therefore enough to show that the elements \eqref{eqn:gente} are linearly independent. To this end, it suffices to show that they are linearly independent in $\wCA^{\uparrow}$, in which case the result follows from the fact that the basis \eqref{eqn:gente} is upper-triangular in terms of the basis:
\begin{equation}
\label{eqn:marc}
E_{d_1,k_1}...E_{d_t,k_t}
\end{equation}
(we show this explicitly in \cite{W gen}, although it is already visible from \eqref{eqn:elementary}). The fact that the elements \eqref{eqn:marc} are linearly independent in $\CA$ follows from \cite{BS}, Remark 5.1.


\end{proof}

\subsection{}
\label{sub:verma}

Let us now define the main representation of $_qW$--algebras that we will study. \\

\begin{definition}
\label{def:verma}

Given parameters $u_1,...,u_r$ with product $u$, the \textbf{Verma module} $M_{u_1,...,u_r}$ is the $\CA_r$--module generated by a single vector $|\emptyset\rangle$ modulo relations:
\begin{equation}
\label{eqn:verma}
M_{u_1,...,u_r} = \CA_r|\emptyset \rangle \Big /_{W_{d,k} | \emptyset \rangle = 0, \ W_{0,k}| \emptyset \rangle = e_k(u_1,...,u_r)| \emptyset \rangle, \ \ \forall k \in \{1,...,r\}, \ d>0}
\end{equation}
where $e_k$ denotes the $k$--th elementary symmetric function. \\

\end{definition}

\noindent A more rigorous way to define $M_{u_1,...,u_r}$ is to say that it is the quotient of $\CA_r$ by the right ideal $(W_{d,k}, W_{0,k} - e_k(u_1,...,u_r))_{1 \leq k \leq r, d>0}$, endowed with the induced left action of $\CA_r$. Therefore, we could also describe the Verma module as:
$$
M_{u_1,...,u_r} \cong \frac {\wCA^\uparrow}{\text{two-sided ideal }(W_{d,k})^{d\in \BZ}_{k>r}, \text{right ideal }(W_{d,k}, W_{0,k} - e_k(u_1,...,u_r))^{d>0}_{k>0}} 
$$
$$
\rightarrow  \frac {\wCA^{\uparrow \ext}}{\text{two-sided ideal of relations \eqref{eqn:impose rels bis}}, \text{right ideal }(p_k, W_{d,k}, W_{0,k} - e_k(u_1,...,u_r))^{d>0}_{k>0}}
$$
as linear maps of $\BF$--vector spaces. \\

\begin{proposition}

When $r=1$, we have $M_u \cong F_u$ of Subsection \ref{sub:fock}. \\

\end{proposition}

\begin{proof} When $k'=1$, relation \eqref{eqn:w rel} states that:
$$
W_k(x) W_1(y) \zeta \left(\frac x{yq^k} \right) - W_1(y) W_k(x) \zeta \left(\frac y{xq} \right) = 
$$
\begin{equation}
\label{eqn:luna}
= \frac {(1-q_1)(1-q_2)}{1-q} \left[\delta \left(\frac y{xq} \right) W_{k+1}(y) - \delta \left( \frac x{yq^k} \right)W_{k+1}(x) \right]
\end{equation}
Taking the series coefficient of $x^{-a}y^{-b}$ allows one to express $W_{k+1,a+b}$ as a $\BF$--linear combination of $W_{k,d}$'s and $W_{1,e}$'s. In order to show that the action $\CA \curvearrowright F_u$ of Proposition \ref{prop:fock act} factors through its quotient $\CA_1$, we need to show that $W_k(x)$ acts by 0 for all $k>1$. By iterating \eqref{eqn:luna}, it is enough to prove that $W_2(x)$ acts by 0. Again in virtue of \eqref{eqn:luna}, this boils down to showing that:
\begin{equation}
\label{eqn:swal}
W_1(x) W_1(y) \zeta \left(\frac x{yq} \right) = W_1(y) W_1(x) \zeta \left(\frac y{xq} \right)
\end{equation}
holds in the module $F_u$. By \eqref{eqn:vertex 1}, this identity reduces to the expression:
$$
u^2 \exp \left[ \sum_{n=1}^\infty \frac {p_{-n}}{n x^{-n}} \right] \exp \left [ \sum_{n=1}^\infty \frac {p_n}{nx^n} \right] \exp \left[ \sum_{n=1}^\infty \frac {p_{-n}}{n y^{-n}} \right] \exp \left [ \sum_{n=1}^\infty \frac {p_n}{ny^n} \right] \zeta \left(\frac x{yq} \right)
$$
being symmetric in $x$ and $y$, as an endomorphism of the module $F_u$. This is a straightforward computation, that one can derive from $[p_{-n},p_n] = n(1-q_1^n)(1-q_2^n)$, in the same way that the second line of \eqref{eqn:equality} implies the third line of \eqref{eqn:equality}. \\

\noindent Having shown that $\CA_1$ acts on $F_u$, let us show that $M_u \cong F_u$ as $\CA_1$--modules. The equation right before the statement of the Proposition gives a $\BF$--linear map:
\begin{equation}
\label{eqn:thalia}
M_u \stackrel{\phi}\rightarrow \frac {\wCA^{\uparrow\ext}}{(I,J)}
\end{equation}
where $I$ is the two-sided ideal generated by $c-q$ and the coefficients of relations:
\begin{equation}
\label{eqn:natalia}
W_k(x) = \begin{cases} u \exp \left[ \sum_{n=1}^\infty \frac {p_{-n}}{n x^{-n}} \right] \exp \left[\sum_{n=1}^\infty \frac {p_n}{nx^n} \right] & \text{ if }k=1 \\ 0 & \text{ if }k>1 \end{cases}
\end{equation}and $J$ is the right ideal generated by $p_k$ for all $k>0$. Note that the right-hand side of \eqref{eqn:thalia} is the Fock space $F_u$, because the ideal of relations allows one to reduce any element to a product of $p_{-1},p_{-2},...$. The fact that the action matches that of \eqref{eqn:vertex 0} and \eqref{eqn:vertex 1} is forced upon us by relations \eqref{eqn:heisenberg}, $c=q$ and \eqref{eqn:natalia}. \\

\noindent Thus we obtain a map of $\BF$--vector spaces $M_u \stackrel{\phi}\rightarrow F_u$. However, both of these vector spaces are non-negatively graded with the vacuum vector in degree zero. The dimension of $F_u$ in degree $k$ is equal to the number of partitions of size $k$. Meanwhile, the dimension of $M_u$ in degree $k$ is less than or equal to the number of partitions of size $k$, because the collection $\{W_{-d_1,1}...W_{-d_s,1}|\emptyset\rangle \}_{d_1 \geq ... \geq d_s>0}$ is a spanning set of $M_u$ over $\BF$ (this follows from the proof of Proposition \ref{prop:w gens and rels}, or by \eqref{eqn:swal}). Thus, to prove that the map $\phi$ is an isomorphism, it is enough to show that it is surjective. \\

\noindent Recall that $h_{\pm n}$ are in relation to $p_{\pm n}$ as complete symmetric functions are in relation to power sum functions. As $d_1 \geq ... \geq d_s$ range over the natural numbers, the monomials $p_{-d_1} ... p_{-d_s}|\emptyset\rangle$ form a $\BF$--basis of $F_u$, which is upper triangular in the basis consisting of monomials $h_{-d_1} ... h_{-d_s}|\emptyset\rangle$. Meanwhile, \eqref{eqn:natalia} implies that:
$$
W_{-d,1} = u h_{-d} + u\sum_{i=1}^\infty h_{-d-i}h_i
$$
If $d_1 \geq ... \geq d_s >0$, we may iterate the equality above to obtain:
\begin{equation}
\label{eqn:mihai}
W_{-d_1,1}...W_{-d_s,1} = u^s h_{-d_1}...h_{-d_s} + u^s\sum_{i_1,...,i_s \in \BN \sqcup \{0\}}^{i_1+...+i_s > 0} h_{-d_1-i_1}h_{i_1}...h_{-d_s-i_s}h_{i_s}
\end{equation}
The commutation relation \eqref{eqn:heisenberg} implies that for all $a,b>0$, the product $h_a h_{-b}$ is a linear combination of $h_{-b+j}h_{a-j}$ as $j \geq 0$. Therefore, equation \eqref{eqn:mihai} implies:
$$
W_{-d_1,1}...W_{-d_s,1} =  u^s h_{-d_1}...h_{-d_s} + \sum_{e_1 \geq ... \geq e_s \geq 0}^{f_1,...,f_s \geq 0} \text{ constant} \cdot h_{-e_1}...h_{-e_s} h_{f_1}...h_{f_s} 
$$
where $e_1 \geq d_1$, if $e_1 = d_1$ then $e_2\geq d_2$, if $e_1 = d_1$ and $e_2 = d_2$ then $e_3 \geq d_3$ etc. We conclude that if $d_1 \geq d_2 \geq ... \geq d_s >0$, we have the following equality in $M_u$:
$$
W_{-d_1,1}...W_{-d_s,1} |\emptyset\rangle = h_{-d_1}...h_{-d_s} |\emptyset\rangle + \text{higher order terms}
$$
where we say that the monomial $h_{-e_1}...h_{-e_t}$ (with $e_1 \geq ... \geq e_t > 0$) has higher order than $h_{-d_1}...h_{-d_s}$ if $e_1 > d_1$, or if $e_1 = d_1$ and $e_2 > d_2$, or if $e_1 = d_1$ and $e_2 = d_2$ and $e_3 > d_3$ etc. Therefore, the basis $W_{-d_1,1}...W_{-d_s,1} |\emptyset\rangle$ is upper triangular in the basis  $h_{-d_1} ... h_{-d_s}|\emptyset\rangle$, and this concludes the surjectivity of the map $M_u \rightarrow F_u$.

\end{proof}

\section{The Moduli Space of Sheaves on $\BP^2$}
\label{sec:geom}

\subsection{}
\label{sub:sheaf}

In this Section, we construct level $r$ good representations of $\CA$, in the guise of the $K$--theory of the moduli space of rank $r$ sheaves. To introduce this algebraic variety, consider the projective plane and fix a line $\infty \subset \BP^2$. We let $\CM$ denote the moduli space of rank $r$ torsion free sheaves $\CF$ on $\BP^2$, together with an isomorphism:
\begin{equation}
\label{eqn:framing}
\CF|_\infty \cong \CO_\infty^{\oplus r}
\end{equation}
This latter condition forces $c_1(\CF)=0$, but $c_2(\CF)$ is still free to range over the non-negative integers. For $n \geq 0$, we denote by $\CM_n \subset \CM$ the connected component of rank $r$ sheaves of second Chern class $n \cdot [\text{pt}]$. Its tangent spaces are given by:
\begin{equation}
\label{eqn:ks}
\Tan_{\CF} \CM_n = \textrm{Ext}^1(\CF, \CF(-\infty))
\end{equation}
by the Kodaira-Spencer isomorphism. Using the Riemann-Roch theorem, one can easily prove that $\CM_n$ is smooth of dimension $2rn$. We have a universal sheaf $\CU_n$ on $\CM_n \times \BP^2$, and its first derived direct image under the standard projection $\text{pr}_1: \CM_n \times \BP^2 \rightarrow \CM_n$ is called the \textbf{tautological vector bundle}:
\begin{equation}
\label{eqn:taut}
\CV_n = R^1\textrm{pr}_{1*} (\CU_n(-\infty))
\end{equation}
on $\CM_n$. The twist in formula \eqref{eqn:taut} is by the pull-back of the divisor $\infty \subset \BP^2$ to $\CM_n \times \BP^2$, and it forces $R^0$ and $R^2$ to vanish. Therefore, $\CV_n$ is a vector bundle, and a standard application of the Riemann-Roch Theorem shows that it has rank $n$. \\

\subsection{} 
\label{sub:tor act} 

Consider the torus $T = \BC^* \times \BC^* \times (\BC^*)^r$, which acts on the moduli space $\CM$ in the following way: the first two factors rescale the coordinate directions of $\BP^2$ and keep the line $\infty$ invariant, and the torus $(\BC^*)^r$ acts by left multiplication on the framing isomorphism \eqref{eqn:framing}. We will consider the equivariant $K-$theory groups $K_T(\CM_n)$, which are all modules over: 
$$
K_T(\pt) = \BZ[q_1^{\pm 1},q_2^{\pm 1},u_1^{\pm 1},...,u_r^{\pm 1}]
$$
Here, $q_1$ and $q_2$ are equivariant parameters in the factors of $\BC^* \times \BC^*$ (which will be identified with the homonymous parameters of Subsection \ref{sub:def shuf}), and $u_i$ are equivariant parameters of the maximal torus of $GL_r$ (which will be identified with the parameters of Subsection \ref{sub:verma}). It will be convenient to localize these groups:
$$
K_n = K_T(\CM_n) \bigotimes_{K_T(\pt)} \text{Frac} \left( K_T(\pt) \right)
$$
and work with them all together:
$$
K = \bigoplus_{n \geq 0} K_n
$$
In the remainder of this Section, we will recall the action $\CA \curvearrowright K$ which makes $K$ into a good level $r$ representation. Before doing so, let us describe the spaces $K_n$ using tautological classes. Consider the assignment:
\begin{equation}
\label{eqn:def taut}
f(x_1,x_2,...) \text{ symmetric} \quad \stackrel{\Upsilon_n}\mapsto \quad \bar{f}_n := f \left(e^{\text{Chern roots of }\CV_n} \right) \in K_n
\end{equation}
When $f = e_k(x_1,x_2,...)$ is an elementary symmetric function, we have $\bar{f}_n = [\wedge^k \CV_n]$. Since the assignment $\Upsilon_n$ is additive and multiplicative, any class in its image is a linear combination of products of exterior powers of $\CV_n$. As $\CV_n$ is the tautological bundle, elements in the image of $\Upsilon_n$ will be called \textbf{tautological classes}. \\

\begin{proposition}
\label{prop:gen}

\emph{(see, for example, \cite{Mod}):} For all $n \geq 0$, the map $\Upsilon_n$ is surjective, i.e. the $\BQ(q_1,q_2,u_1,...,u_r)$--algebra $K_n$ is generated by tautological classes. \\

\end{proposition}

\subsection{} 
\label{sub:part}

In this Subsection, we will use the language of partitions $\lambda = (\lambda_1 \geq \lambda_2 \geq ...)$. To any such partition, we can associate its Young diagram, which is a collection of lattice squares in the first quadrant. For example, the following is the Young diagram of the partition $\la = (4,3,1)$:

\begin{picture}(100,160)(-90,-15)
\label{fig}

\put(0,0){\line(1,0){160}}
\put(0,40){\line(1,0){160}}
\put(0,80){\line(1,0){120}}
\put(0,120){\line(1,0){40}}

\put(0,0){\line(0,1){120}}
\put(40,0){\line(0,1){120}}
\put(80,0){\line(0,1){80}}
\put(120,0){\line(0,1){80}}
\put(160,0){\line(0,1){40}}




\put(65,-20){\mbox{Figure \ref{fig}}}

\end{picture}


\tab 
Given two partitions, we will write $\lambda \geq \mu$ if the Young diagram of $\lambda$ completely contains that of $\mu$. In this case, the collection of boxes $\lamu$ will be called a \textbf{skew partition}. If such a skew partition $\lamu$ has $n$ boxes, then a labeling of these boxes with the numbers $1,...,n$ is called a \textbf{standard Young tableau} if the numbers decrease as we go up and to the right in the partition.

\tab 
For a fixed natural number $r$, we will use the term $r$--\textbf{partition}:
$$
\bla = (\la^1,...,\la^r)
$$
for an ordered collection of ordinary partitions. The size of $\bla$ is defined to be $|\bla| = |\lambda^1|+...+|\lambda^r|$, and we mimic the notation of usual partitions by writing $\bla \vdash n$ for ``$\bla$ is an $r$--partition of size $n$". A box in an $r$--partition is determined not only by its position in the plane, but also by which of $\lambda^1,...,\lambda^r$ it lies in. We collect this information in the \textbf{weight} of the box, namely the monomial:
\begin{equation}
\label{eqn:weight}
\chi_\sq = u_k q_1^i q_2^j
\end{equation}
for the box $\sq \in \la^k$ whose bottom left corner lies at coordinates $(i,j)$. We will say that $\bla \geq \bmu$ if we have $\la^i \geq \mu^i$ for all $i\in \{1,...,r\}$. If this is the case, the $r$--tuple $\blamu = \{\la^1 \backslash \mu^1,...,\la^r \backslash \mu^r\}$ will be called a \textbf{skew $r$--partition}. If such a skew $r$--partition consists of $n$ boxes, then a labeling of these boxes with the numbers $1,...,n$ is called a \textbf{standard Young tableau} if the numbers decrease as we go up and to the right in each constituent partition $\la^i \backslash \mu^i$, $i\in \{1,...,r\}$. Note that there is no restriction between the numbers placed in $\sq \in \la^i \backslash \mu^i$ and $\sq' \in \la^j \backslash \mu^j$ if $i\neq j$. \\

\subsection{} 
\label{sub:fixed points}

The reason why we introduced the language of the previous Subsection is that the torus fixed points of the action $T \curvearrowright \CM$ are naturally indexed by $r$--partitions $\bla = (\la^1,...,\la^r)$. Specifically, fixed points are of the form:
$$
\CI_\bla := \CI_{\lambda^1} \oplus ... \oplus \CI_{\lambda^r}
$$ 
where for any usual partition $\mu=(\mu_1 \geq \mu_2 \geq ...)$, we write $\CI_\mu \subset \CO_{\BP^2}$ for the sheaf corresponding to the ideal $I_\mu = (x^{\mu_1}, x^{\mu_2}y,x^{\mu_3}y^2,...) \subset \BC[x,y]$. The $K$--theory class of the skyscraper sheaf at $\CI_\bla$ will be denoted by $[\bla] \in K$, and for convenience we will renormalize it to:
\begin{equation}
\label{eqn:renormalize}
|\bla \rangle := \frac {[\bla]}{[ \wedge^\bullet \Tan_\bla \CM_n ]} \in K_{n}
\end{equation}
if $n = |\bla| \Leftrightarrow \bla \vdash n$. The symbol $\wedge^\bullet$ is defined in \eqref{eqn:exterior power} below. The reason for the normalization \eqref{eqn:renormalize} is a consequence of the Thomason localization theorem, also known as the $K$--theoretic version of Atiyah-Bott localization, which reads:
\begin{equation}
\label{eqn:localization}
c = \sum_{\bla \vdash n} c|_\bla \cdot |\bla \rangle \quad \forall \ c \in K_n
\end{equation}
The situation we will mostly be concerned with is when $c = \bar{f}_n$ is a tautological class, for some symmetric polynomial $f$. In this case, we have:
$$
\bar{f}_n|_\bla = f(\bla) := f(...,\chi_\sq,...)_{\sq \in \bla}
$$
The notation above means that we plug the weights of the boxes of the $r$--partition $\bla$, which are defined in \eqref{eqn:weight}, into the arguments of the symmetric polynomial $f$. Combining this with \eqref{eqn:localization}, we conclude that:
$$
\bar{f}_n = \sum_{\bla \vdash n} f(\bla) \cdot |\bla\rangle
$$



\subsection{}
\label{sub:action}

We are now ready to define the action $\CA \curvearrowright K$, following \cite{Mod} (note that this action also appears in \cite{FT} and \cite{SV}, in either case presented in a different language). Let us write $Z = z_1+...+z_k$ for the finite alphabet of variables that appears in the definition of shuffle elements, and $X = x_1+x_2+...$ for the infinite alphabet of variables that appears in the definition of tautological classes. Also write:
\begin{equation}
\label{eqn:def tau}
\tau(z) = \prod_{i=1}^r \left(1 - \frac {z}{u_i} \right)
\end{equation}
We will often use the language of plethysms when it comes to symmetric polynomials. Specifically, given a symmetric polynomial in infinitely many variables $f(x_1,x_2,...)$ and a finite collection of variables $Z = z_1+...+z_k$, we will write:
$$
f(X+Z) = f(x_1,x_2,...,z_1,...,z_k)
$$
and think of it as a symmetric polynomial in the $x_i$'s with coefficients depending on the $z_j$'s.
The notation $f(X-Z)$ does not have as straightforward a presentation as above, but we may rigorously define the assignments $f(X) \leadsto f(X \pm Z)$ as the homomorphisms defined on the ring of symmetric polynomials in $X$ by setting:
$$
\sum_{i=1}^\infty x_i^n \leadsto \sum_{i=1}^\infty x_i^n \pm (z_1^n+...+z_k^n)
$$
for all $n \in \BN$. Since power sum functions generate the ring of symmetric polynomials, this completely defines the notation $f(X \pm Z)$ for any symmetric $f$. \\

\begin{theorem}
\label{thm:action}

\emph{(\cite{Mod}):} For any $k,d>0, n \geq 0$, any symmetric polynomial $f(X) = f(x_1,x_2,...)$ and any shuffle element $R(Z) = R(z_1,...,z_k)$, the assignments $c=q^r$,
\begin{equation}
\label{eqn:act 0}
P_{0,\pm d} \cdot \bar{f}_n = \pm \overline{f \cdot \left[ \sum_{i=1}^r u_i^{\pm d}q^{\delta_\pm^- d} - (1-q_1^d)(1-q_2^d) \sum_{j=1}^\infty x_j^{\pm d} \right]}_n
\end{equation}
\begin{equation}
\label{eqn:act minus}
R^\leftarrow \cdot \bar{f}_n  = \int_+ \overline{f(X-Z)\zeta \left(\frac ZX \right)}_{n+k} \frac {R(Z) \tau(qZ)}{\zeta \left(\frac {Z}{Z} \right)}
\end{equation}
\begin{equation}
\label{eqn:act plus}
R^\rightarrow \cdot \bar{f}_n = q^{k(1-r)} \int_- \overline{f(X+Z)\zeta \left(\frac XZ \right)^{-1}}_{n-k} \frac {R(Z) \tau\left( Z \right)^{-1} }{\zeta \left(\frac {Z}{Z} \right)}
\end{equation}
give rise to an action $\CA \curvearrowright K$. The integrals $\int_+, \int_-$ will be defined in Remark \ref{rem:normal} below. The right-hand sides of \eqref{eqn:act minus} and \eqref{eqn:act plus} use the multiplicative notation:
\begin{equation}
\label{eqn:multiplicative}
\zeta\left(\frac XZ \right) := \prod_{i=1}^\infty \prod_{j=1}^k \zeta \left( \frac {x_i}{z_j} \right), \qquad \zeta \left( \frac ZZ \right) := \prod_{1 \leq i \neq j \leq k} \zeta \left( \frac {z_i}{z_j} \right)
\end{equation}
If $n<k$, then the right-hand side of \eqref{eqn:act plus} is defined to be 0. \\

\end{theorem}


\begin{remark}
\label{rem:normal}

The normal ordered integrals $\int_\pm$ are defined in \cite{Mod}, but we will only need to invoke them when $R$ is a shuffle element of the form:
\begin{equation}
\label{eqn:our r}
R(z_1,...,z_k) = \esym \left[ \frac {\rho(z_1,...,z_k)}{\prod_{i=1}^{k-1} \left(1 - \frac {qz_{i+1}}{z_i} \right)} \prod_{1 \leq i < j \leq k} \zeta \left( \frac {z_i}{z_j} \right) \right]
\end{equation}
for a Laurent polynomial $\rho$. For this choice of $R$, the right-hand sides of \eqref{eqn:act minus} and \eqref{eqn:act plus} are \textbf{defined} as the following contour integrals (write $Dz = \frac {dz}{2\pi i z}$):
$$
\int_+ = \int_{z_k \succ ... \succ z_1 \succ X} \overline{ f \left(X-\sum_{i=1}^k z_i \right)\prod_{i=1}^k \left[ \zeta \left(\frac {z_i}X \right) \tau(qz_i)\right]}  \frac {\rho(z_1,...,z_k) Dz_1 ... Dz_k}{\prod_{i=1}^{k-1} \left(1 - \frac {qz_{i+1}}{z_i} \right) \prod_{i < j} \zeta \left( \frac {z_j}{z_i} \right)} 
$$
\begin{equation}
\label{eqn:dank}
\int_- = \int_{z_1 \succ ... \succ z_k \succ X} \overline {\frac {f (X+\sum_{i=1}^k z_i)}{\prod_{i=1}^k \left[ \zeta \left(\frac X{z_i} \right) \tau(z_i) \right]}} \cdot \frac {\rho(z_1,...,z_k) Dz_1 ... Dz_k}{\prod_{i=1}^{k-1} \left(1 - \frac {qz_{i+1}}{z_i} \right) \prod_{i < j} \zeta \left( \frac {z_j}{z_i} \right)} 
\end{equation}
In the right-hand side, the notation $z \succ X$ means that the variable $z$ runs over a contour that surrounds all of the $X$ variables. More generally, in the integral $\int_-$, the contour of $z_1$ surrounds the contour of $z_2$, which surrounds the contour of $z_3$,..., which surrounds the contour of $z_k$, which in turn surrounds all the $X$ variables. 0 and $\infty$ \textbf{must not} be contained in any of the contours. \\

\end{remark}

\begin{remark}
\label{rem:normal 2}

Formulas \eqref{eqn:act minus} and \eqref{eqn:act plus} continue to hold if we multiply $f(X)$ by a product of the form $\prod_j \prod_{i=1}^\infty (x_i - t_j)^{\pm 1}$ for some formal symbols $t_j$. Then the notation $f(X\pm Z)$ may pick up poles from various factors $\prod_j \prod_{i=1}^k (z_i - t_j)^{\pm 1}$, and in this case, the normal-ordered integrals $\int_\pm$ must be defined such that the complex numbers $\{t_j\}$, together with $0$ and $\infty$, lie outside the $z_1,...,z_k$ contours. \\

\end{remark}

\subsection{}
\label{sub:fixed points}

In \cite{Mod}, we proved that formulas \eqref{eqn:act 0}--\eqref{eqn:act plus} imply the following relations in the basis of fixed points $| \bla \rangle$, and are in fact equivalent to them upon computing the integrals \eqref{eqn:dank} by a residue computation:
\begin{equation}
\label{eqn:act 0 fixed}
\langle \bmu |P_{0,\pm d} | \bla \rangle = \pm \delta_\bla^\bmu \left[ \sum_{i=1}^r u_i^{\pm d}q^{\delta_\pm^- d} - (1-q_1^d)(1-q_2^d) \sum_{\sq \in \bla} \chi_\sq^{\pm d} \right]
\end{equation}
\begin{equation}
\label{eqn:act minus fixed}
\langle \bla | R^\leftarrow |\bmu \rangle = R(\blamu) \prod_{\bsq \in \blamu} \left[\frac {(1-q_1)(1-q_2)}{1-q} \zeta\left(\frac {\chi_\bsq}{\chi_\bmu} \right) \tau \left( q \chi_\bsq \right) \right]
\end{equation}
\begin{equation}
\label{eqn:act plus fixed}
\langle \bmu | R^\rightarrow |\bla \rangle = R(\blamu) \prod_{\bsq \in \blamu} \left[ \frac {(1-q_1)(1-q_2)}{q^{r-1}(1-q)} \zeta\left(\frac {\chi_\bla}{\chi_\bsq} \right)^{-1} \tau \left( \chi_\bsq  \right)^{-1} \right]
\end{equation}
In formulas \eqref{eqn:act 0 fixed}--\eqref{eqn:act plus fixed}, $\zeta\left(\frac {\chi_\bla}z \right)$ is multiplicative notation for $\prod_{\sq \in \bla} \zeta\left(\frac {\chi_\sq}z \right)$. If $\bla \not \geq \bmu$, then the right-hand sides of \eqref{eqn:act minus fixed} and \eqref{eqn:act plus fixed} are defined to be zero. In particular, if $R$ is a shuffle element of the form \eqref{eqn:our r}, then it is easy to see that:
\begin{equation}
\label{eqn:act minus fixed 2}
\langle \bla | R^\leftarrow |\bmu \rangle = \sum^{T \text{ a standard Young}}_{\text{tableau of shape } \blamu} \ \frac {\rho(\chi_1,...,\chi_k)}{\prod_{i=1}^{k-1} \left(1 - \frac {q\chi_{i+1}}{\chi_i} \right)} 
\end{equation}
$$
\prod_{1 \leq i < j \leq k} \zeta \left( \frac {\chi_i}{\chi_j} \right)  \prod_{i=1}^k \left[\frac {(1-q_1)(1-q_2)}{1-q} \zeta\left(\frac {\chi_i}{\chi_\bmu} \right) \tau \left( q \chi_i \right) \right]
$$
\begin{equation}
\label{eqn:act plus fixed 2}
\langle \bmu | R^\rightarrow |\bla \rangle = \sum^{T \text{ a standard Young}}_{\text{tableau of shape } \blamu} \ \frac {\rho(\chi_1,...,\chi_k)}{\prod_{i=1}^{k-1} \left(1 - \frac {q\chi_{i+1}}{\chi_i} \right)} 
\end{equation}
$$
\prod_{1 \leq i < j \leq k} \zeta \left( \frac {\chi_i}{\chi_j} \right)  \prod_{i=1}^k \left[\frac {(1-q_1)(1-q_2)}{q^{r-1}(1-q)} \zeta\left(\frac {\chi_\bla}{\chi_i} \right)^{-1} \tau \left( \chi_i  \right)^{-1} \right]
$$
where $\chi_1,...,\chi_k$ denote the weights of the boxes labeled $1,...,k$ in the standard Young tableau $T$. The argument for why formulas \eqref{eqn:act minus fixed}--\eqref{eqn:act plus fixed} imply \eqref{eqn:act minus fixed 2}--\eqref{eqn:act plus fixed 2} is explained in detail in \loccitt: the main idea is that, up to a product of linear factors, \eqref{eqn:act minus fixed} and \eqref{eqn:act plus fixed} are given by evaluating a shuffle element $R$ at the weights of a skew $r$--partition $\blamu$. When $R$ is presented as a symmetrization $\sym [...]$ as in \eqref{eqn:our r}, this evaluation may be computed by adding together the specializations of $[...]$ at all labelings of the boxes of $\blamu$ with the numbers $1,...,k$. However, because $\zeta(q_1^{-1}) = \zeta(q_2^{-1}) = 0$, such a labeling produces a non-zero contribution if and only if the labels decrease as we go up and to the right in the $r$--partition. This precisely says that the labeling must be a standard Young tableau. \\

\subsection{}
\label{sub:w action}

Formulas \eqref{eqn:act minus fixed 2} and \eqref{eqn:act plus fixed 2} allow us to write down the operators $W_{d,k} \in \wCA^\uparrow$ in the basis $\{|\bla \rangle, \bla \ r\text{--partition} \}$ of the level $r$ representation $K$. We will do so by invoking formula \eqref{eqn:explicit shuffle}:
\begin{equation}
\label{eqn:ldu reloaded}
W_{d,k} =  \sum^{d_\rightarrow - d_\leftarrow = d}_{k_\leftarrow + k_0 + k_\rightarrow = k} T_{d_\leftarrow,k_\leftarrow}^\leftarrow E_{0,k_0} T_{d_\rightarrow, k_\rightarrow}^\rightarrow \cdot q^{(k-1)d_\rightarrow}
\end{equation}
where the sum goes over $d_\leftarrow, d_\rightarrow \in \BN$ and $k_\leftarrow, k_0, k_\rightarrow \in \BN \sqcup \{0\}$. Since the shuffle elements $T_{d,k}$ are given by \eqref{eqn:awesome shuffle}, then \eqref{eqn:act minus fixed 2}--\eqref{eqn:act plus fixed 2} imply:
\begin{equation}
\label{eqn:lower}
\langle \bmu | T_{d_\leftarrow,k_\leftarrow}^\leftarrow |\bnu \rangle = \sum^{T \text{ a standard Young}}_{\text{tableau of shape } \bmunu} \ \frac {(-1)^{k_\leftarrow-1} \chi_{d_\leftarrow}^{k_\leftarrow}}{\prod_{i=1}^{d_\leftarrow-1} \left(1 - \frac {q\chi_{i+1}}{\chi_i} \right)} 
\end{equation}
$$
\prod_{1 \leq i < j \leq d_\leftarrow} \zeta \left( \frac {\chi_i}{\chi_j} \right)  \prod_{i=1}^{d_\leftarrow} \left[\frac {(1-q_1)(1-q_2)}{1-q} \zeta\left(\frac {\chi_i}{\chi_\bnu} \right) \tau \left( q \chi_i \right) \right]
$$
\begin{equation}
\label{eqn:upper}
\langle \bnu | T_{d_\rightarrow, k_\rightarrow}^\rightarrow |\bla \rangle = \sum^{T \text{ a standard Young}}_{\text{tableau of shape } \blanu} \ \frac {(-1)^{k_\rightarrow
-1} \chi_{d_\rightarrow}^{k_\rightarrow}}{\prod_{i=1}^{d_\rightarrow-1} \left(1 - \frac {q\chi_{i+1}}{\chi_i} \right)} 
\end{equation}
$$
\prod_{1 \leq i < j \leq d_\rightarrow} \zeta \left( \frac {\chi_i}{\chi_j} \right)  \prod_{i=1}^{d_\rightarrow} \left[\frac {(1-q_1)(1-q_2)}{q^{r-1}(1-q)} \zeta\left(\frac {\chi_\bla}{\chi_i} \right)^{-1} \tau \left( \chi_i  \right)^{-1} \right]
$$
for any skew partitions $\bmunu$ and $\blanu$ of size $d_\leftarrow$ and $d_\rightarrow$, respectively. Finally, because of \eqref{eqn:act 0 fixed}, we have:
\begin{equation}
\label{eqn:diagonal}
\left \langle \bnu' | E_{0,k_0} | \bnu \right \rangle = \delta_{\bnu'}^\bnu \resy \left[ (-1)^{k_0} y^{k_0} \prod_{i=1}^r \left(1 - \frac {u_i}y\right) \prod_{\sq \in \bnu} \zeta \left(\frac {\chi_\sq}{y} \right) \frac {dy}y \right]
\end{equation}
Therefore, we may restate \eqref{eqn:ldu reloaded} by saying that the matrix coefficient $\langle \bmu | W_{d,k} |\bla \rangle$ is a sum over all $r$--partitions $\bnu \subset \bla \cap \bmu$ of products of the matrix coefficients as in \eqref{eqn:lower}, \eqref{eqn:upper}, \eqref{eqn:diagonal}, thus proving the second formula in Theorem \ref{thm:ldu}. This formula will allow us to show that the action $\CA \curvearrowright K$ factors through an action of the $_qW$--algebra $\CA_r$, as introduced in Definition \ref{def:w}: \\

\begin{theorem}
\label{thm:geom}

There is an action $\CA_r \curvearrowright K$, with respect to which the latter is isomorphic to the Verma module $M_{u_1,...,u_r}$ of \eqref{eqn:verma}. The parameters $u_1,...,u_r$ are defined as the equivariant parameters of $K$, hence there is no abuse of notation. \\

\end{theorem}

\begin{proof} We will actually show that $\CA_r^\ext$ acts on $K$, hence the desired action will be a consequence of the homomorphism \eqref{eqn:homo}. In fact, we have actions of:
$$
\wCA^\uparrow \subset \wCA^{\uparrow \ext} \subset \text{a completion of }\CA
$$
on $K$, all induced by the fact that Theorem 3.7 provides an action $\CA \curvearrowright K$, for which the latter is a good representation. To show that the action factors through the subquotient $\CA_{r}^\ext$, we need to show that this action respects the relations:
\begin{equation}
\label{eqn:want 0}
W_r(x) = u \left[ \sum_{n=0}^\infty \frac {h_{-n}}{x^{-n}}\right] \left[ \sum_{n=0}^\infty \frac {h_n}{x^n}\right] 
\end{equation}
\begin{equation}
\label{eqn:want 1}
W_k(x) = 0 \qquad \forall \ k>r
\end{equation}
where $u = u_1...u_r$. First of all, let us prove that $K$ is generated by the coefficients of the $_qW$--currents acting on the vacuum vector $|\emptyset \rangle$ (here, $\emptyset$ denotes the empty $r$--partition). The reason for this is that the operators $P_{n,1} \in \CA$ lie in the $_qW$--algebra, according to the case $k=1$ of \eqref{eqn:w k}. From formula \eqref{eqn:relation 2}, we see that the operators $P_{n,1}$ generate the entire upper shuffle algebra $\CA^\uparrow$ (alternatively, this follows from Theorem \ref{thm:shuf} and Proposition \ref{prop:half is shuffle}), and the upper shuffle algebra acting on $|\emptyset \rangle$ generates $K$ (a quick argument for this is to observe that $P_{0,d} \in \CA^\uparrow$, and $\{P_{0,d}\}_{d>0}$ form a commuting family of operators which are diagonal in the basis $|\bla \rangle$ with distinct spectra). Thus, proving \eqref{eqn:want 0} and \eqref{eqn:want 1} reduces to proving that:
\begin{equation}
\label{eqn:want 2}
W_r(x) \prod^{\text{various}}_{(d,k')} W_{d,k'} |\emptyset \rangle = u \left[ \sum_{n=0}^\infty \frac {h_{-n}}{x^{-n}}\right] \left[ \sum_{n=0}^\infty \frac {h_n}{x^n}\right] \prod^{\text{various}}_{(d,k')} W_{d,k'} |\emptyset \rangle
\end{equation}
\begin{equation}
\label{eqn:want 3}
W_k(x) \prod^{\text{various}}_{(d,k')} W_{d,k'} |\emptyset \rangle = 0 \qquad \forall \ k>r
\end{equation}

\begin{claim}
\label{claim:claim}

Formulas \eqref{eqn:want 2} and \eqref{eqn:want 3} reduce to the fact that:
\begin{equation}
\label{eqn:want 4}
W_r(x)|\emptyset \rangle = u  \left[ \sum_{n=0}^\infty \frac {h_{-n}}{x^{-n}}\right] |\emptyset \rangle \quad \text{and} \quad W_k(x)|\emptyset \rangle = 0  \qquad \forall \ k>r
\end{equation}

\end{claim}

\noindent \textbf{Proof of Claim \ref{claim:claim}:} Suppose that \eqref{eqn:want 4} holds. Then one may prove that the left-hand side of \eqref{eqn:want 3} vanishes by commuting the current $W_k(x)$ past the various $W_{d,k'}$'s. This uses the commutation relations \eqref{eqn:w rel}, which express the commutator of $W_k(x)$ and $W_{k'}(y)$ in terms of currents $W_{k''}(x)W_{k'''}(y)$ with $k'''>\max(k,k')$. The same argument shows that \eqref{eqn:want 4} $\Rightarrow$ \eqref{eqn:want 2}, once one shows that $W_r(x)$ and: 
$$
u \left[ \sum_{n=0}^\infty \frac {h_{-n}}{x^{-n}}\right] \left[ \sum_{n=0}^\infty \frac {h_n}{x^n}\right]
$$
have the same commutation relations with $W_1(y),...,W_r(y)$. This is an immediate consequence, which we leave as an exercise to the interested reader, of \eqref{eqn:w rel 0 minus}--\eqref{eqn:w rel} and the already proved fact that $W_k(x) = 0$ for $k>r$. \qed \\

\noindent It therefore remains to prove \eqref{eqn:want 4}, so let us recall the LDU decomposition \eqref{eqn:ldu reloaded}. Since $T_{d,k}^\rightarrow$ annihilates the vacuum vector unless $d=0$, we have:
$$
W_{d,k} |\emptyset \rangle =  \sum^{k_\leftarrow, k_0 \geq 0}_{k_\leftarrow + k_0 = k} T_{d,k_\leftarrow}^\leftarrow E_{0,k_0} |\emptyset\rangle 
$$
We may sum the above expression over all $d\in \BZ$, and obtain the formula:
\begin{equation}
\label{eqn:big factor}
W_k(x) = \resy \left[ T(x,y)^\leftarrow E(y) |\emptyset\rangle \cdot (-y)^k \frac {dy}y \right]
\end{equation}
where $E(y) = \sum_{k=0}^\infty (-1)^k \frac {E_{0,k}}{y^k}$ and:
\begin{equation}
\label{eqn:pam pam}
T(x,y) = \sum_{d=0}^\infty \sym \left[ \frac {x^d}{\left(1 - \frac y{z_d} \right)\prod_{i=1}^{d-1} \left(1 - \frac {qz_{i+1}}{z_i} \right)} \prod_{1 \leq i < j \leq d} \zeta \left( \frac {z_i}{z_j} \right) \right]
\end{equation}
Recall that \eqref{eqn:diagonal} implies that $E(y) |\emptyset\rangle = \prod_{i=1}^r \left(1 - \frac {u_i}y \right)|\emptyset\rangle$. Thus \eqref{eqn:big factor} becomes:
$$
W_k(x)|\emptyset \rangle = \resy \left[ T(x,y)^\leftarrow |\emptyset \rangle \cdot (-y)^k \prod_{i=1}^r \left(1 - \frac {u_i}y \right) \frac {dy}y \right]
$$
Now we will use \eqref{eqn:act minus fixed 2}, or equivalently \eqref{eqn:lower}, to compute the matrix coefficients of the above vector in terms of the basis $|\bla \rangle$. Letting $|\bla| = d$, we have:
\begin{equation}
\label{eqn:stephen}
\langle \bla | W_k(x)|\emptyset \rangle = \resy \left[ (-y)^k \prod_{i=1}^r \left(1 - \frac {u_i}y \right)  \sum^{T \text{ a standard Young}}_{\text{tableau of shape } \bla} \right.  
\end{equation}
$$
\left. \frac {x^d}{\left(1 - \frac y{\chi_d} \right)\prod_{i=1}^{d-1} \left(1 - \frac {q\chi_{i+1}}{\chi_i} \right)}  \prod_{1 \leq i < j \leq d} \zeta \left( \frac {\chi_i}{\chi_j} \right)  \prod_{i=1}^d \left[\frac {(1-q_1)(1-q_2)}{1-q} \tau \left( q \chi_i \right) \right] \frac {dy}y \right]
$$
The integral of \eqref{eqn:stephen} is equal to the sum of its residues at the poles $y = 0$ and $y = \chi_d$. However, note that the integrand does not actually have a pole at $y = \chi_d$, because in any standard Young tableau of shape $\bla$, we have $\chi_d \in \{u_1,...,u_r\}$, and thus the pole at $y = \chi_d$ is canceled by one of the factors on the first line of \eqref{eqn:stephen}. Moreover, when $k>r$ there is also no pole at $y=0$, and hence the integral is 0. This implies that $W_k(x) |\emptyset \rangle = 0$ for all $k>r$, as in \eqref{eqn:want 4}. By the same reason, when $k=r$ the integral is equal to the residue at the simple pole at $y = 0$:
$$
\langle \bla | W_r(x)|\emptyset \rangle = u  \sum^{T \text{ a standard Young}}_{\text{tableau of shape } \bla} \frac {x^d \prod_{i < j} \zeta \left( \frac {\chi_i}{\chi_j} \right) }{\prod_{i=1}^{d-1} \left(1 - \frac {q\chi_{i+1}}{\chi_i} \right)}   \prod_{i=1}^d \left[\frac {(1-q_1)(1-q_2)}{1-q} \tau \left( q \chi_i \right) \right]
$$
Comparing this with formula \eqref{eqn:hkd} for $H_{-d,0}$ and \eqref{eqn:act minus fixed 2}, we infer that:
$$
\langle \bla | W_r(x)|\emptyset \rangle = \langle \bla | u \sum_{n=0}^\infty \frac {h_{-n}}{x^{-n}} |\emptyset \rangle \quad \forall \bla \qquad \Rightarrow \qquad  W_r(x)|\emptyset \rangle = u \sum_{n=0}^\infty \frac {h_{-n}}{x^{-n}} |\emptyset \rangle
$$
and this implies \eqref{eqn:want 4}. Now that we have shown that $\CA_r \rightarrow \CA_r^\ext \curvearrowright K$, let us show that the representation $K$ is isomorphic to the Verma module $M_{u_1,...,u_r}$. The map:
\begin{equation}
\label{eqn:fock to k}
M_{u_1,...,u_r} \rightarrow K, \qquad | \emptyset \rangle \mapsto | \emptyset \rangle 
\end{equation}
is well-defined because the currents $W_{d,k}$ for $d>0$ annihilate the vacuum vector for degree reasons. Meanwhile, the LDU decomposition \eqref{eqn:big factor} implies:
$$
W_{0,k} | \emptyset \rangle = \resy \left[ (-y)^k E(y) |\emptyset \rangle  \frac {dy}y \right] = e_k(u_1,...,u_r) |\emptyset\rangle
$$
Moreover, the map \eqref{eqn:fock to k} is surjective, because as we have shown in the beginning of the proof, the vector space $K$ is generated by the $_qW$--algebra acting on the vacuum. Therefore, it is enough to show that the graded dimension of $M_{u_1,...,u_r}$ is less than or equal to that of $K$. As a consequence of the fact that the monomials \eqref{eqn:gente} linearly span the $_qW$--algebra, the monomials:
\begin{equation}
\label{eqn:mon mon}
W_{d_1,k_1}... W_{d_t,k_t} |\emptyset\rangle
\end{equation}
with $\frac {d_1}{k_1} \leq ... \leq \frac {d_t}{k_t}$ and any $d_i<0$, $k_i \in \{1,...,r\}$, linearly generate $M_{u_1,...,u_r}$. Monomials \eqref{eqn:mon mon} are therefore in 1-to-1 correspondence with unordered collections:
$$
\underbrace{(d_{a_1},1), \dots ,(d_{a_e},1)}_{a_1 \geq ... \geq a_e >0}, \underbrace{(d_{b_1},2), \dots, (d_{b_f},2)}_{b_1 \geq ... \geq b_f > 0}, \dots, \underbrace{(d_{c_1},r), \dots, (d_{c_g},r)}_{c_1 \geq ... \geq c_g > 0}
$$
The number of such collections of any total degree $n = \sum d_i$ is equal to the number of $r$--partitions of size $r$, which is also equal to the dimension of the degree $n$ piece of $K$. This proves that the surjective map \eqref{eqn:fock to k} is bijective, hence an isomorphism.



\end{proof}

\subsection{}
\label{sub:adjoint}

Given an equivariant vector bundle $\CV$ on a space $X$ with an action of a torus $T$, we will define two operations on it. The first is the total exterior power:
\begin{equation}
\label{eqn:exterior power}
\wedge^\bullet \CV = \sum_{i=0}^{\text{rank }\CV} (-1)^i [ \wedge^i \CV^\vee] \in K_T(X)
\end{equation}
Note the dual sign in the right-hand side, which is placed there to ensure that $\wedge^\bullet$ naturally deforms the Euler class of the vector bundle $\CV$, in the same way $1-e^{-x}$ deforms $x$. The total exterior power is multiplicative in $\CV$. The second operation on vector bundles is $T$--equivariant Euler characteristic, which is additive in $\CV$:
$$
\chi_T(X,\CV) = \sum_{i=0}^{\infty} (-1)^i \left[ T\text{--character of } H^i(X,\CV) \right] \in \text{Rep}(T) 
$$
We will now show that for any shuffle element $R \in \CS$, the operators $R^\leftarrow, R^\rightarrow$ of \eqref{eqn:act minus}, \eqref{eqn:act plus} are (up to a constant) adjoint with respect to the inner product:
\begin{equation}
\label{eqn:inner 1}
(\cdot, \cdot) : K_n \otimes K_n \rightarrow \BQ(q_1, q_2, u_1,...,u_r)
\end{equation}
$$
(\alpha, \beta) = \chi_T\left(\CM_n, \frac {\alpha \otimes \beta}{(\det \CV_n)^{\otimes r}} \right) \prod_{i=1}^r (-u_i)^n
$$
The basis $|\bla \rangle$ is orthogonal with respect to this pairing:
\begin{equation}
\label{eqn:euler}
(| \bla \rangle, | \bmu \rangle ) = \frac {\delta_\bla^\bmu}{\left[ \wwedge^\bullet \Tan_\bla \CM_n \right]}
\end{equation}
where $n=|\bla|$ and:
$$
\wwedge^\bullet \Tan_\bla \CM_n = \wedge^\bullet \Tan_\bla \CM_n \otimes \frac {(\det \CV_{|\bla|})^{\otimes r}}{\prod_{i=1}^r (-u_i)^{|\bla|}}
$$
We will use the following expression (see, for example, \cite{Mod}) for the $K$--theory class of the tangent space to $\CM_n$:
\begin{equation}
\label{eqn:tan}
[\Tan \ \CM_n] = \sum_{i=1}^r \left(\frac {\CV_n}{u_i} + \frac {u_i}{q\CV_n} \right) - \left(1 - \frac 1{q_1} \right)\left(1 - \frac 1{q_2} \right) \frac {\CV_n}{\CV_n}
\end{equation}
In the right-hand side of \eqref{eqn:tan} we employ the shorthand notation:
$$
\frac {\CV}{\CW} := [\CV] \otimes [\CW^\vee]
$$
for any vector bundles (or $K$--theory classes) $\CV$ and $\CW$. If we recall that the restriction of $\CV_n$ to the fixed point $\bla$ is given by $\CV_n|_\bla = \sum_{\sq \in \bla} \chi_\sq$, we obtain:
$$
\left[ \wedge^\bullet \Tan_\bla \CM_n \right] = \prod^{1 \leq i \leq r}_{\sq \in \bla} \left(1- \frac {u_i}{\chi_\sq}\right)\left(1- \frac {q\chi_\sq}{u_i} \right) \prod_{\sq,\sq' \in \bla} \zeta \left(\frac {\chi_\sq}{\chi_{\sq'}} \right) \quad \Rightarrow
$$
$$
\Rightarrow \quad \left[ \wwedge^\bullet \Tan_\bla \CM_n \right] = \prod^{1 \leq i \leq r}_{\sq \in \bla} \left(1- \frac {\chi_\sq}{u_i}\right)\left(1- \frac {q\chi_\sq}{u_i} \right) \prod_{\sq,\sq' \in \bla} \zeta \left(\frac {\chi_\sq}{\chi_{\sq'}} \right)
$$
We conclude that \eqref{eqn:euler} reads:
\begin{equation}
\label{eqn:orthonorm}
(| \bla \rangle, | \bmu \rangle ) = \frac {\delta_\bla^\bmu}{\prod_{\sq \in \bla} \tau(\chi_\sq) \tau(q\chi_\sq) \prod_{\sq,\sq' \in \bla} \zeta \left( \frac {\chi_\sq}{\chi_{\sq'}} \right)} 
\end{equation}
Together with \eqref{eqn:act minus fixed} and \eqref{eqn:act plus fixed}, this implies that:
\begin{equation}
\label{eqn:adjoint}
\left(R^\rightarrow \right)^\dagger = q^{(1-r) \deg R} \cdot R^\leftarrow \qquad \forall \ R \in \CS 
\end{equation}
where adjoint is defined with respect to the inner product \eqref{eqn:inner 1}. \\

\subsection{}
\label{sub:correspondences}

Although not apparently clear from formulas \eqref{eqn:act minus} and \eqref{eqn:act plus}, the operators $R^\leftarrow$ and $R^\rightarrow$ are given by geometric correspondences. To be precise, this is only true for shuffle elements $R$ of the form \eqref{eqn:our r}, which we will assume for the remainder of this Section. Consider the so-called \textbf{fine correspondence} defined in \cite{Mod}:
\begin{equation}
\label{eqn:diag fine}
\xymatrix{& \fZ_{n+k,n} \ar[ld]_{\pi_+} \ar[rd]^{\pi_-} \\ 
\CM_{n+k} & & \CM_{n}}
\end{equation}
for any $k>0$ and any natural number $n$, which is defined as the locus:
\begin{align}
\fZ_{n+k,n} = \Big\{ &\text{flags of sheaves } \CF_{n+k} \subset ... \subset \CF_n \text{ s.t. } \CF_{n+i-1}/\CF_{n+i} \text{ for } 1 \leq i \leq k \nonumber \\
&\text{are all length 1 skyscraper sheaves supported at the same point} \Big\} \label{eqn:fine}
\end{align}
We will write $\fZ_k = \sqcup_{n \in \BN} \fZ_{n+k,n}$. Note that $\fZ_k$ is endowed with line bundles:
\begin{equation}
\label{eqn:picard}
\CL_1,...,\CL_k \in \text{Pic}_T(\fZ_k)
\end{equation}
with fibers given by:
$$
\CL_i|_{\CF_{n+k} \subset ... \subset \CF_n} = \Gamma \left(\BP^2, \CF_{n+k-i}/\CF_{n+k-i+1} \right)
$$
In \eqref{eqn:diag fine}, the maps $\pi_+$ and $\pi_-$ send a flag $\CF_{n+k} \subset ... \subset \CF_n$ to the sheaves $\CF_{n+k}$ and $\CF_n$, respectively. In \emph{loc. cit.}, we defined a virtual class on $\fZ_k$ which allows us to make sense of push--forward maps $\pi_{+*}$ and $\pi_{-*}$ on $K$--theory. For convenience, let us renormalize these maps as:
$$
\wpi_{+*}(c) = \pi_{+*}(c), \qquad \wpi_{-*}(c) = \pi_{-*} \left(c \cdot \CL_1^{-r}...\CL_k^{-r} \right) \prod_{i=1}^r (-u_i)^k
$$
This normalization was done precisely to ensure that the pairs $(\wpi_{\pm *}, \pi_\pm^*)$ are adjoint with respect to the inner product \eqref{eqn:inner 1} on $K_T(\CM_n)$ and the inner product:
\begin{equation}
\label{eqn:inner 2}
(\alpha, \beta) = \chi_T\left(\fZ_{n+k,n}, \frac {\alpha \otimes \beta}{(\det \CV_{n+k})^{\otimes r}} \right) \prod_{i=1}^r (-u_i)^{n+k}
\end{equation}
on the $K$--theory of $\fZ_k$. We have the following: \\

\begin{theorem}
\label{thm:corr}
\emph{(\cite{Mod}):} For a shuffle element $R$ as in \eqref{eqn:our r}, we have:
$$
R^\leftarrow \ = \ \wpi_{+*}\left(\rho(\CL_1,...,\CL_k) \cdot \pi_-^* \right)
$$
$$
R^\rightarrow = \wpi_{-*}\left(\frac {\rho(\CL_1,...,\CL_k)}{q^{k(r-1)}} \cdot \pi_+^* \right)
$$

\end{theorem}
 
\noindent This Theorem establishes two things: first of all, the operators $R^\leftarrow$ and $R^\rightarrow$ of Theorem \ref{thm:action} are geometric in nature. Secondly, relation \eqref{eqn:adjoint} on the adjointness of these operators follows from the adjointness of the correspondences $\wpi_{+*}\pi_-^*$ and $\wpi_{-*}\pi_+^*$. In the next Section, we will make use of the following Proposition, whose proof follows that of Proposition 3.52 of \cite{Ext} almost word-by-word, so we omit it. \\

\begin{proposition}
\label{prop:corr}

Assume $\rho(z_1,...,z_k)$ is a Laurent polynomial, divided by linear factors of the form $z_i-y$, for constants $y$ in a certain finite set $Y$. Then:
\begin{align*}
&\wpi_{+*} \left( \rho(\CL_1,...,\CL_k) \right) = \int_{ X^+  \succ z_1 \succ ... \succ z_k \succ Y \cup \{0,\infty\}} \frac {\overline{\rho\left(z_1,...,z_k\right) \prod_{i=1}^k \left[ \zeta \left(\frac {z_i}{X^+} \right) \tau(qz_i) Dz_i \right]}}{\prod_{i=1}^{k-1} \left(1 - \frac {qz_{i+1}}{z_i} \right) \prod_{i < j} \zeta \left( \frac {z_j}{z_i} \right)} \\
&\wpi_{-*} \left( \rho(\CL_1,...,\CL_k) \right) = \int_{Y \cup \{0,\infty\} \succ z_1 \succ ... \succ z_k \succ X^-} \overline {\frac {\rho\left(z_1,...,z_k\right)\prod_{i=1}^k \left[ \zeta \left(\frac {X^-}{z_i} \right)^{-1} \frac {Dz_i}{\tau(z_i)} \right]}{\prod_{i=1}^{k-1} \left(1 - \frac {qz_{i+1}}{z_i} \right) \prod_{i < j} \zeta \left( \frac {z_j}{z_i} \right)}} 
\end{align*}
where the notation $z \succ z'$ is defined in Remark \ref{rem:normal}. \\

\end{proposition}

\noindent Because of the normal weights of $\pi_-$, the analogous formula for $\pi_{-*}$ instead of $\widetilde{\pi}_{-*}$ involves replacing the factors:
$$
\prod_{i=1}^k \frac 1{\tau(z_i)} = \prod_{i=1}^k \prod_{j=1}^r \left(1 - \frac {z_i}{u_j} \right)^{-1} \qquad \text{by} \qquad \prod_{i=1}^k \prod_{j=1}^r \left(1 - \frac {u_j}{z_i} \right)^{-1}
$$
The discrepancy between the two expressions is precisely the factor $\prod_{i=1}^k \prod_{j=1}^r \frac {-u_j}{z_i}$, which accounts for the discrepancy between the operators $\widetilde{\pi}_{-*}$ and $\pi_{-*}$. \\

\begin{remark}
\label{rem:notation}

Let us explain the notation in Proposition \ref{prop:corr}, since it will feature in the following Section. Let $X^+$ and $X^-$ denote tautological classes on the spaces $\CM_{n+k}$ and $\CM_{n}$ in \eqref{eqn:diag fine}, respectively, as well as their pull-backs to $\fZ_{n+k,n}$. For example, the notation $\overline{e_a(X^+) e_b(X^-)}$ will refer to the class:
$$
\wedge^a\left(\pi_+^*\CV_{n+k}\right) \cdot \wedge^b\left(\pi_-^*\CV_n\right) \in K_T \left( \fZ_{n + k, n} \right)
$$
There will be a slight ambiguity as to whether the notation $\overline{e_a(X^+)}$ refers to a $K$--theory class on $\CM_{n+k}$ or its pull-back to $\fZ_{n+k, n}$, but the situation will always be clear from context. For example, because the right-hand sides of the formulas in Proposition \ref{prop:corr} are $K$--theory classes in the targets of $\wpi_{+*}$ and $\wpi_{-*}$, respectively, they have no choice but to live on the spaces $\CM_{n+k}$ and $\CM_{n}$, respectively. \\

\end{remark}

\begin{remark}
\label{rem:change contours}

The contour that gives the integral for $\widetilde{\pi}_{-*}$ in Proposition \ref{prop:corr} precisely matches that of $\int_-$ in Remark \ref{rem:normal}. However, the contour in the integral for $\wpi_{+*}$ is ordered as $ X^+  \succ z_1 \succ ... \succ z_k \succ Y \cup \{0,\infty\}$, while the corresponding contour that defines $\int_+$ in Remark \ref{rem:normal} would be $Y \cup \{0,\infty\} \succ z_k \succ ... \succ z_1 \succ X_+$. The discrepancy between the two contours is merely cosmetic, since it just involves looking at the Riemann sphere from ``behind", i.e. changing the variable $z \leftrightarrow \frac 1z$. \\

\end{remark}

\section{The $\Ext$ operator}
\label{sec:ext}

\subsection{}
\label{sub:def e}

Keep the natural number $r$ fixed. We will now consider two different moduli spaces of rank $r$ sheaves, of degrees $n$ and $n'$, respectively:
$$
\CM_{n,\bu} \quad \text{and} \quad \CM_{n',\bu'}
$$
We assume that these moduli spaces are acted on by two different rank $r$ tori, whose equivariant parameters will be denoted by $\bu = (u_1,...,u_r)$ and $\bu' = (u_1',...,u_r')$, respectively. Therefore, we will write $K_{n,\bu} = K_{\BC^* \times \BC^* \times (\BC^*)^r}(\CM_{n,\bu})$, with the understanding that the rank $r$ torus $(\BC^*)^r$ has equivariant parameters $\bu$. With this in mind, consider the vector bundle $\CE_{n,n'}$ on $\CM_{n,\bu} \times \CM_{n',\bu'}$ with fibers given by:
$$
\CE_{n,n'}|_{\CF,\CF'} = \Ext^1(\CF',\CF(-\infty))
$$
The fact that $\CE_{n,n'}$ is a vector bundle follows from the vanishing of the corresponding $\Hom$ and $\Ext^2$ groups. In fact, an easy application of the Riemann-Roch theorem allows one to see that the rank of $\CE_{n,n'}$ is $r(n+n')$. Its $K$--theory class is given by: 
\begin{equation}
\label{eqn:ext}
[\CE_{n,n'}] = \sum_{i=1}^r \left(\frac {\CV_n}{u'_i} + \frac {u_i}{q\CV_{n'}} \right) - \left(1 - \frac 1{q_1} \right)\left(1 - \frac 1{q_2} \right) \frac {\CV_n}{\CV_{n'}}
\end{equation}
where the vector bundles $\CV_n$ and $\CV_{n'}$ are the pull-backs of the tautological bundles from $\CM_{n,\bu}$ and $\CM_{n',\bu'}$, respectively, to the product $\CM_{n,\bu} \times \CM_{n',\bu'}$ (see formula 6.16 of \cite{Mod}). We use the $K$--theory class of $\CE_{n,n'}$ as a correspondence:
$$
\xymatrix{ & \CE_{n,n'} \ar@{-->}[d] \\
& \CM_{n,\bu} \times \CM_{n',\bu'} \ar[ld]_{p_1} \ar[rd]^{p_2} \\ 
\CM_{n,\bu} & & \CM_{n',\bu'}}
$$

\begin{definition}
\label{def:am}

For a formal parameter $m$, consider the operators:
\begin{equation}
\label{eqn:a correspondence}
A_m|_n^{n'} \ : \ K_{n',\emph{\bu}'} \longrightarrow K_{n,\emph{\bu}}
\end{equation}
$$
A_m|_n^{n'} = \tp_{1*} \left(\wwedge^\bullet (\CE_{n, n'} \otimes m ) \cdot p_2^{*} \right)
$$
where we define the ``tilde" quantities as small modifications of the usual ones:
$$
\tp_{1*}(c) = p_{1*}\left(c \otimes \frac {\prod_{i=1}^r (-u_i')^{n'}}{(\det \CV'_n)^{\otimes r}} \right) 
$$
$$
\wwedge^\bullet (\CE_{n, n'} \otimes m ) = \wedge^\bullet(\CE_{n, n'} \otimes m ) \otimes \left(\det  \CV_n \right)^{\otimes r} \prod_{i=1}^r \left( - \frac m{u_i} \right)^n 
$$

\end{definition}

\noindent Note that the latter formula may be recast using \eqref{eqn:ext} as:
\begin{equation}
\label{eqn:ext 2}
\left[ \wwedge^\bullet (\CE_{n, n'} \otimes m ) \right] = \tau'\left( mX \right) \tau\left(\frac {qX'}m \right) \zeta \left( \frac {X'}{mX} \right)
\end{equation}
where:
$$
\tau(z) = \prod_{i=1}^r \left(1 - \frac z{u_i} \right) \quad \text{ and } \quad \tau'(z) = \prod_{i=1}^r \left(1 - \frac z{u_i'} \right)
$$
In formula \eqref{eqn:ext 2} and throughout the remainder of this paper, $X$ and $X'$ denote tautological classes on the two factors of $\CM_{n,\bu} \times \CM_{n',\bu'}$, on which $\CE_{n,n'}$ is defined. \\


\subsection{}
\label{sub:def a}

Formula \eqref{eqn:ext} implies the following expression in the basis of fixed points:
$$
\CE|_{\bla,\bla'} := \CE_{|\bla|,|\bla'|}|_{\CI_\bla,\CI_{\bla'}} = \sum^{1 \leq i \leq r}_{\sq \in \bla} \frac {\chi_\sq}{u'_i} + \sum^{1 \leq i \leq r}_{\sq' \in \bla'} \frac {u_i}{q\chi_{\sq'}} - \sum^{\sq \in \bla}_{\sq' \in \bla'} \left(1 - \frac 1{q_1} \right)\left(1 - \frac 1{q_2} \right) \frac {\chi_\bla}{\chi_{\bla'}}
$$
One understands the above formula by interpreting the right-hand side as the character of $\BC^* \times \BC^* \times (\BC^*)^r \times (\BC^*)^r$ in the fiber of $\CE$ above the fixed point indexed by $r$--partitions $\bla, \bla'$. The two factors $(\BC^*)^r$ act with equivariant parameters $\bu$ and $\bu'$, respectively. If we let $n=|\bla|$ and $n'=|\bla'|$, then the definition of $A_m|_n^{n'}$ as a correspondence in \eqref{eqn:a correspondence} allows us to compute its matrix coefficients in the basis of renormalized fixed points:
\begin{equation}
\label{eqn:matrix coefficients}
\langle \bla | A_m|_n^{n'} |\bla' \rangle = \frac {\wwedge^\bullet (\CE|_{\bla,\bla'} \otimes m )}{\wwedge^\bullet ( \Tan_{\bla'} \CM_{n',\bu'} )} =
\end{equation}
$$
= \frac {\prod^{1 \leq i \leq r}_{\sq \in \bla} \left(1 - \frac {m \chi_\sq}{u'_i} \right) \prod^{1 \leq i \leq r}_{\sq' \in \bla'} \left(1 - \frac {q\chi_{\sq'}}{mu_i} \right) \prod^{\sq \in \bla}_{\sq' \in \bla'} \zeta \left( \frac {\chi_\sq'}{m \chi_\sq} \right)}{\prod^{1 \leq i \leq r}_{\sq \in \bla'} \left(1 - \frac {\chi_\sq}{u_i'} \right) \prod^{1 \leq i \leq r}_{\sq' \in \bla'} \left(1 - \frac {q\chi_{\sq'}}{u_i'} \right) \prod^{\sq \in \bla'}_{\sq' \in \bla'} \zeta \left( \frac {\chi_\sq'}{\chi_\sq} \right)}
$$
To group all of the operators $A_m|_n^{n'}$ together, we introduce the generating current:
\begin{equation}
\label{eqn:gen current}
A_m(x) = \sum_{n,n' \geq 0} A_m|_n^{n'} \cdot x^{n-n'}
\end{equation}
and note that this was the main actor in the Introduction. Then \eqref{eqn:matrix coefficients} implies the following formula for the Nekrasov partition function \eqref{eqn:nek part}:
$$
Z_{m_1,...,m_k}(x_1,...,x_k) = \Tr \left(A_{m_1}(x_1)... A_{m_k}(x_k) \right) = \sum^{r\text{--partitions}}_{\bla_1,...,\bla_k} \prod_{a=1}^k x_a^{|\bla_a| - |\bla_{a+1}|} 
$$
\begin{equation}
\label{eqn:partition function}
\frac {\prod^{1 \leq i \leq r}_{\sq \in \bla_a} \left(1 - \frac {m \chi_\sq}{u^{a+1}_i} \right) \prod^{1 \leq i \leq r}_{\sq' \in \bla_{a+1}} \left(1 - \frac {q\chi_{\sq'}}{mu^a_i} \right) \prod^{\sq \in \bla_a}_{\sq' \in \bla_{a+1}} \zeta \left( \frac {\chi_\sq'}{m \chi_\sq} \right)}{\prod^{1 \leq i \leq r}_{\sq \in \bla_{a+1}} \left(1 - \frac {\chi_\sq}{u^{a+1}_i} \right) \prod^{1 \leq i \leq r}_{\sq' \in \bla_{a+1}} \left(1 - \frac {q\chi_{\sq'}}{u^{a+1}_i} \right) \prod^{\sq \in \bla_{a+1}}_{\sq' \in \bla_{a+1}} \zeta \left( \frac {\chi_\sq'}{\chi_\sq} \right)}
\end{equation}
where $\bu^{a} = (u_1^{a},...,u_r^{a})$ is the collection of equivariant parameters corresponding to the moduli space of sheaves whose fixed points are indexed by $r$--partitions $\bla_a$, as in the Introduction, and we identify $\bu^{k+1} = \bu^1$, $\bla_{k+1} = \bla_1$. The right-hand side of \eqref{eqn:partition function} is the partition function of gauge theory with matter on a length $k$ cyclic quiver, and justifies physical interest in the operators $A_m(x)$. Note that in our definition, there is a single rank $r$ involved in the definition of the operators $A_m(x)$, which amounts to the fact that all the vertices in the cyclic quiver are given gauge group $U(r)$. If one wanted to study quiver gauge theory for various groups $U(r_1),...,U(r_k)$, one would have to consider the operators $A_m(x)$ for some integer $r \geq r_1,...,r_k$ and then send some of the equivariant parameters in \eqref{eqn:partition function} to $\infty$. \\

\subsection{}
\label{sub:comm rels}

For brevity, when we will work with a single operator $A_m(x)$, we will set $x=1$ without losing any information (the variable $x$ comes in handy when computing compositions of vertex operators, as we will see in Subsection \ref{sub:work out}). We will now turn to the main purpose of this Section, which is to compute the commutation relations of $A_m:=A_m(1)$ with the operators $P_{k,0}$ and $P_{k,1}$ (let us remark that in principle, our method allows us to compute the commutation relations with any $P_{k,d}$, but the answer becomes more complicated as $d$ grows larger). \\

\begin{theorem}
\label{thm:commute}

For any rank $r$, we have the following relations between the \emph{Ext} operator $A_m$ and the generators of the algebra $\CA$:
\begin{equation}
\label{eqn:Comm 1}
[A_m, p_{-k}] = A_m  \left[ \left(\frac {q^r u'}{m^r u} \right)^k - 1\right], \quad 
[A_m, p_{k}] = A_m \left[ 1 - \left( \frac {m^r u}{u'} \right)^k \right]
\end{equation}
for all $k>0$, and:
\begin{equation}
\label{eqn:Comm 2}
A_m P_{k,1} - \frac {m^r u}{q^r u'} A_m P_{k-1,1}  = \frac mq P_{k,1} A_m  - \frac {m^{r+1} u}{q^r u'}  P_{k-1,1} A_m 
\end{equation}
for all $k \in \BZ$. Here we use the notation $u = u_1...u_r$ and $u' = u_1'...u_r'$. \\

\end{theorem}

\begin{proof} We will actually prove slightly different, but equivalent, formulas than \eqref{eqn:Comm 1} and \eqref{eqn:Comm 2}. For example, since the $p_{\pm k} = P_{\pm k,0}$ generate a copy $\Lambda_\pm$ of the ring of symmetric polynomials, then we may recast formula \eqref{eqn:Comm 1} in terms of the following commutation relations involving the corresponding complete symmetric functions:
\begin{equation}
\label{eqn:comm 1}
A_m h_{-k} - h_{-k} A_m \ = \ \frac {q^r u'}{m^r u}  A_m h_{-k+1} - h_{-k+1} A_m
\end{equation}
\begin{equation}
\label{eqn:comm 2}
A_m h_k - h_k  A_m = A_m h_{k-1}  - \frac {m^r u}{u'} h_{k-1}  A_m 
\end{equation}
The way we prove that \eqref{eqn:Comm 1} is equivalent to \eqref{eqn:comm 1}--\eqref{eqn:comm 2} is to observe that both are particular cases of the following formulas:
\begin{equation}
\label{eqn:zz 1}
A_m x^{(1)}_- \cdot \ph_{\frac {q^r u'}{m^r u}} \left( x^{(2)}_- \right) = x^{(1)}_- A_m  \cdot \ph_1 \left(x^{(2)}_- \right)
\end{equation}
\begin{equation}
\label{eqn:zz 2}
A_m x^{(1)}_+ \cdot \ph_1 \left( x^{(2)}_+ \right) = x^{(1)}_+ A_m  \cdot \ph_{\frac {m^r u}{u'}} \left(x^{(2)}_+ \right)
\end{equation}
for all $x_\pm \in \Lambda_\pm$, where $\Delta(x_\pm) = x^{(1)}_\pm \otimes x^{(2)}_\pm$ denotes the usual coproduct on $\Lambda_\pm$ in Sweedler notation (recall that this coproduct has $p_{\pm k}$ as primitive elements and $\sum_{k=0}^\infty h_{\pm k}z^k$ as a group-like element). The functionals $\ph_s:\Lambda_\pm \rightarrow \BQ(q,m,u,u')$ are ring homomorphisms defined by either of the equivalent sets of assignments:
$$
\ph_s(h_{\pm k}) = \begin{cases} 1 & \text{ if } k = 0 \\ - s & \text{ if } k = 1 \\ 0 & \text{ otherwise} \end{cases} \qquad \Leftrightarrow \qquad \ph_s(p_{\pm k}) = - s^k
$$
Therefore, formulas \eqref{eqn:zz 1}--\eqref{eqn:zz 2} are multiplicative in $x_\pm$: if they hold for $x_\pm$ and $y_\pm$, then they also hold for $(xy)_\pm$. Therefore, proving these formulas for $x_\pm = p_{\pm k}$ is equivalent to proving them for $x_\pm = h_{\pm k}$, hence \eqref{eqn:Comm 1} is equivalent to \eqref{eqn:comm 1}--\eqref{eqn:comm 2}. \\

\noindent Meanwhile, by changing the index $k \leftrightarrow -k+1$, one can easily see that formula \eqref{eqn:Comm 2} is equivalent to the two following relations for all $k>0$:
\begin{equation}
\label{eqn:comm 3}
A_m P_{-k,1} - m P_{-k,1} A_m = \frac {q^r u'}{m^r u} \left( A_m P_{-k+1,1} - \frac mq P_{-k+1,1} A_m \right) 
\end{equation}
\begin{equation}
\label{eqn:comm 4}
A_m P_{k,1} - \frac mq P_{k,1} A_m  = \frac {m^r u}{q^r u'} \Big( A_m P_{k-1,1} - m P_{k-1,1} A_m\Big) 
\end{equation}
To prove \eqref{eqn:comm 1} and \eqref{eqn:comm 3}, consider the following diagrams of spaces and arrows:
\begin{equation}
\label{eqn:big diagram 1}
\xymatrix{ & & \CM_{n,\bu} \times \CM_{n',\bu'}  \ar@/_3pc/[llddd]_{p_1} \ar@/^3pc/[rrddd]^{p_2} & & \\
& & \CM_{n,\bu} \times \fZ_{n'+k,n'} \ar[ld]_{\Id \times \pi_+} \ar[rd]^{\nu}   \ar[u]^{\Id \times \pi_-} & & \\
& \CM_{n,\bu} \times \CM_{n'+k,\bu'} \ar[ld] \ar[rd] & & \fZ_{n'+k,n'} \ar[ld]_{\pi_+} \ar[rd]^{\pi_-} & \\ 
\CM_{n,\bu} & & \CM_{n'+k,\bu'} & & \CM_{n',\bu'}}
\end{equation}


\begin{equation}
\label{eqn:big diagram 2}
\xymatrix{ & & \CM_{n,\bu} \times \CM_{n',\bu'}  \ar@/_3pc/[llddd]_{p_1'} \ar@/^3pc/[rrddd]^{p'_2} & & \\
& & \fZ_{n,n-k} \times \CM_{n',\bu'} \ar[ld]_{\nu'} \ar[rd]^{\pi'_- \times \Id}   \ar[u]^{\pi_+' \times \Id} & & \\
& \fZ_{n,n-k} \ar[ld]_{\pi'_+} \ar[rd]^{\pi'_-} & & \CM_{n-k,\bu} \times \CM_{n',\bu'} \ar[ld] \ar[rd] & \\ 
\CM_{n,\bu} & & \CM_{n-k,\bu} & & \CM_{n',\bu'}}
\end{equation}


\noindent As a consequence of \eqref{eqn:hkd} and Theorem \ref{thm:corr}, we have:
$$
H_{-k,0} = \wpi_{+*} \left( \pi_-^* \right) = \wpi'_{+*}  \left( \pi'^{*}_- \right) 
$$
Then the usual composition of operators gives us the following formulas:
\begin{equation}
\label{eqn:comp}
A_m H_{-k,0} = \tp_{1*}(\Gamma_k \cdot p_2^{*}) \qquad \text{and} \qquad H_{-k,0} A_m = \tp'_{1*}(\Gamma'_k \cdot {p_2'}^{*})
\end{equation}
where $\tp_{1*}$, $\tp_{1*}'$ are the renormalizations of $p_{1*}$, $p_{1*}'$ defined in Subsection \ref{sub:def e}, and:
$$
\Gamma_k = \left( \Id \times \tpi_- \right)_* \left[\wwedge^\bullet( (\Id \times \pi_+)^* \CE_{n,n'+k} \otimes m) \right]
$$
$$
\Gamma'_k = \left( \tpi_+' \times \Id \right)_* \left[\wwedge^\bullet( (\pi'_- \times \Id)^* \CE_{n-k,n'} \otimes m) \right]
$$
Using \eqref{eqn:ext 2}, we may rewrite $\Gamma_k$ and $\Gamma'_k$ in terms of tautological classes:
$$
\Gamma_k \ = \ \left( \Id \times \tpi_- \right)_* \left[ \tau'\left( m{X} \right) \tau\left(\frac {qU}m \right) \zeta \left( \frac {U}{mX} \right)  \right]
$$
$$
\Gamma'_k = \left( \tpi_+' \times \Id \right)_* \left[ \tau'\left( mU'\right) \tau\left(\frac {qX'}m \right) \zeta \left( \frac {X'}{mU'} \right) \right]
$$
where the alphabets of variables $X$, $X'$ denote tautological classes on the spaces $\CM_{n,\bu}$, $\CM_{n',\bu'}$, respectively, and $U$, $U'$ denote tautological classes on the middle spaces of the bottom row of \eqref{eqn:big diagram 1}, \eqref{eqn:big diagram 2}, respectively. We may identify:
$$
U = X'+\CL_1+...+\CL_k \qquad \text{and} \qquad U' = X-\CL_1-...-\CL_k
$$
because the classes $U$ and $U'$ do not just live on the product $\CM_{\bullet+k} \times \CM_\bullet$, but on the subvariety $\fZ_k$, which comes endowed with the line bundles $\CL_1,...,\CL_k$ of \eqref{eqn:picard}. Thus: 
$$
\Gamma_k \ \ = \ \ \left(\Id \times \tpi_- \right)_* \left[\Upsilon \cdot  \prod_{i=1}^k \tau \left( \frac {q\CL_i}m \right) \zeta \left(\frac {\CL_i}{mX} \right)\right]
$$
and:
$$
\Gamma'_k = \left( \tpi_+' \times \Id \right)_* \left[\Upsilon  \prod_{i=1}^k \tau' \left( m\CL_i \right)^{-1} \zeta \left(\frac {X'}{m\CL_i} \right)^{-1} \right]
$$
where: 
\begin{equation}
\label{eqn:upsilon}
\Upsilon =  \tau'\left( m{X} \right) \tau\left(\frac {qX'}m \right) \zeta \left( \frac {X'}{mX} \right)
\end{equation}
Note that in each of the formulas above, $\Upsilon$ is pulled back from $\CM_{n,\bu} \times \CM_{n',\bu'}$, so we can slide it in front of the the direct image. Therefore, Proposition \ref{prop:corr} implies:
\begin{equation}
\label{eqn:gam 1}
\Gamma_k \ = \ \Upsilon \cdot \int_{X \cup \{0,\infty\} \succ z_1 \succ ... \succ z_k \succ X'} \frac {\prod_{i=1}^k \overline{ \zeta \left(\frac {z_i}{mX} \right) \zeta \left(\frac {X'}{z_i} \right)^{-1} \tau \left( \frac {qz_i}m \right)  \tau'(z_i)^{-1} }  Dz_i}{\prod_{i=1}^{k-1} \left(1 - \frac {qz_{i+1}}{z_i} \right) \prod_{i < j} \zeta \left( \frac {z_j}{z_i} \right)} \qquad 
\end{equation}
and:
\begin{equation}
\label{eqn:gam 2}
\Gamma'_k = \Upsilon \cdot \int_{X \succ z_1 \succ ... \succ z_k \succ X' \cup \{0,\infty\}} \frac {\prod_{i=1}^k \overline{ \zeta \left(\frac {z_i}{X} \right) \zeta\left( \frac {X'}{mz_i}\right)^{-1} \tau \left( qz_i \right)   \tau'(mz_i)^{-1}} Dz_i}{\prod_{i=1}^{k-1} \left(1 - \frac {qz_{i+1}}{z_i} \right) \prod_{i < j} \zeta \left( \frac {z_j}{z_i} \right)} \qquad
\end{equation}
Let us write:
$$
I_k(z_1,...,z_k) = \frac {\prod_{i=1}^k \overline{ \zeta \left(\frac {z_i}{mX} \right) \zeta \left(\frac {X'}{z_i} \right)^{-1} \tau \left( \frac {qz_i}m \right)  \tau'(z_i)^{-1} }}{\prod_{i=1}^{k-1} \left(1 - \frac {qz_{i+1}}{z_i} \right) \prod_{i < j} \zeta \left( \frac {z_j}{z_i} \right)} 
$$
The change of variables $z_i \mapsto z_i m$ implies that $\Gamma_k - \Gamma'_k = $
\begin{equation}
\label{eqn:van}
= \Upsilon \left[ \int_{X \cup \{0,\infty\} \succ z_1 \succ ... \succ z_k \succ X'} I_k \prod_{i=1}^k Dz_i - \int_{X \succ z_1 \succ ... \succ z_k \succ X' \cup \{0,\infty\}} I_k\prod_{i=1}^k Dz_i \right]
\end{equation}
The two integrals have the same integrand, so their difference consists of the residues when one of the variables $z_1,...,z_k$ passes over either 0 or $\infty$. However, because $\lim_{w\rightarrow 0} \zeta(w) = \lim_{w \rightarrow \infty} \zeta(w) = 1$ we have the limits:
\begin{equation}
\label{eqn:limmy 1}
\lim_{z_1 \rightarrow 0} I_k(z_1,...,z_k) = 0, \qquad \lim_{z_1 \rightarrow \infty} I_k(z_1,...,z_k) = \frac {q^r u'}{m^r u} I_{k-1}(z_2,...,z_k) 
\end{equation}
\begin{equation}
\label{eqn:limmy 2}
\lim_{z_s \rightarrow 0} I_k(z_1,...,z_k) = \lim_{z_s \rightarrow \infty} I_k(z_1,...,z_k) = 0 \ \ \qquad \forall \ s\in \{2,...,k-1\}
\end{equation}
\begin{equation}
\label{eqn:limmy 3}
\lim_{z_k \rightarrow 0} I_k(z_1,...,z_k) = I_{k-1}(z_1,...,z_{k-1}), \ \ \ \qquad \lim_{z_k \rightarrow \infty} I_k(z_1,...,z_k) = 0
\end{equation}
where $u = u_1...u_r$ and $u'=u_1'...u_r'$. Therefore, the only  two residues which appear in the difference \eqref{eqn:van} are when $z_1$ passes over $\infty$ and when $z_k$ passes over 0. We conclude that the difference equals:
$$
\Gamma_k - \Gamma_k' = \Upsilon \left[ \int_{X \cup \{0,\infty\} \succ z_2 \succ ... \succ z_{k} \succ X'} \frac {q^r u'}{m^r u} I_{k-1}(z_2,...,z_k)  \prod_{i=2}^k Dz_i - \right.
$$
$$
\left. - \int_{X \succ z_1 \succ ... \succ z_{k-1} \succ X' \cup \{0,\infty\}} I_{k-1}(z_1,...,z_{k-1}) \prod_{i=1}^{k-1} Dz_i \right] = \frac {q^r u'}{m^r u} \cdot \Gamma_{k-1} - \Gamma'_{k-1}
$$
which is precisely \eqref{eqn:comm 1}. Similarly, we have:
$$
A_m P_{-k,1} = \tp_{1*}(\wGamma_k \cdot p_2^{*}) \qquad \text{and} \qquad P_{-k,1} A_m = \tp'_{1*}(\wGamma'_k \cdot {p_2'}^{*})
$$
where:
$$
\wGamma_k = \left( \Id \times \tpi_- \right)_* \left[\CL_k \otimes \wwedge^\bullet( (\Id \times \pi_+)^* \CE_{n,n'+k} \otimes m)\right]
$$
$$
\wGamma'_k = \left( \tpi_+' \times \Id \right)_* \left[\CL_k \otimes \wwedge^\bullet( (\pi'_- \times \Id)^* \CE_{n-k,n'} \otimes m)\right]
$$
By analogy with \eqref{eqn:gam 1} and \eqref{eqn:gam 2}, we have:
$$
\wGamma_k = \Upsilon \cdot \int_{X \cup \{0,\infty\} \succ z_1 \succ ... \succ z_k \succ X'} z_k \cdot I_k(z_1,...,z_k) \prod_{i=1}^k Dz_i
$$
and:
$$
\wGamma_k' = \Upsilon \int_{X \succ z_1 \succ ... \succ z_k \succ X'  \cup \{0,\infty\}} z_k \cdot I_k(z_1 m,...,z_k m) \prod_{i=1}^k Dz_i = 
$$
$$
= \frac {\Upsilon}m \cdot \int_{X \succ z_1 \succ ... \succ z_k \succ X'  \cup \{0,\infty\}} z_k \cdot I_k(z_1,...,z_k) \prod_{i=1}^k Dz_i
$$
where in the last equality we changed the variable $z_i \mapsto z_i m$. Computing the difference between the expressions above, we see that $\wGamma_k - m \wGamma_k' = $
\begin{equation}
\label{eqn:gro}
= \Upsilon \left[ \int_{X \cup \{0,\infty\} \succ z_1 \succ ... \succ z_k \succ X'} z_k I_k \prod_{i=1}^k Dz_i - \int_{X \succ z_1 \succ ... \succ z_k \succ X' \cup \{0,\infty\}} z_k I_k \prod_{i=1}^k Dz_i \right]
\end{equation}
To compute the difference, we must once again sum the residues of the integrand as one of the variables passes over $0$ or $\infty$. Formulas \eqref{eqn:limmy 1} and \eqref{eqn:limmy 2} continue to hold for $I_k$ replaced by $z_kI_k$, but formula \eqref{eqn:limmy 3} must be replaced by:
\begin{equation}
\label{eqn:limmy 4}
\lim_{z_k \rightarrow 0} z_k I_k  = 0, \qquad \lim_{z_k \rightarrow \infty} z_k I_k(z_1,...,z_k) = - \frac {q^{r-1}u'}{m^ru} z_{k-1}I_{k-1}(z_1,...,z_{k-1})
\end{equation} 
With this in mind, \eqref{eqn:gro} yields:
$$
\wGamma_k - m \wGamma_k' = \Upsilon \left[ \int_{X \cup \{0,\infty\} \succ z_2 \succ ... \succ z_{k} \succ X'} \frac {q^ru'}{m^ru} z_k I_{k-1}(z_2,...,z_k) \prod_{i=2}^k Dz_i -  \right.
$$
$$
\left. \int_{X \succ z_1 \succ ... \succ z_{k-1} \succ X' \cup \{0,\infty\}} \frac {q^{r-1}u'}{m^ru} z_{k-1}I_{k-1}(z_1,...,z_{k-1}) \prod_{i=1}^{k-1} Dz_i \right] = \frac {q^r u'}{m^r u} \left( \wGamma_{k-1} - \frac mq \wGamma_{k-1}' \right)
$$
Note that this relation is precisely \eqref{eqn:comm 3}, written in terms of correspondences. The proofs of \eqref{eqn:comm 2} and \eqref{eqn:comm 4} are completely analogous, so we leave them as exercises to the interested reader. More formally, they follow from \eqref{eqn:comm 1} and \eqref{eqn:comm 3} by transposition, since $h_k, P_{k,1}$ are the adjoints of $h_{-k}, P_{-k,1}$ under the inner product \eqref{eqn:inner 1} and Serre duality implies that $A_{\frac qm}$ is the adjoint of $A_m$. To see the latter claim explicitly, one can also combine formulas \eqref{eqn:orthonorm} and \eqref{eqn:matrix coefficients} to obtain:
$$
\langle \bla' | A_{\frac qm} | \bla \rangle = \langle \bla | A_m | \bla' \rangle \cdot \frac {(|\bla \rangle, |\bla \rangle)}{(|\bla' \rangle, |\bla' \rangle)}
$$


\end{proof}

\subsection{}
\label{sub:work out}

Let us now use formulas \eqref{eqn:Comm 1}--\eqref{eqn:Comm 2} to compute the commutation relations of $A_m$ with the $_{q}W$--algebra currents. We revert to the notation $A_m(x)$ for the generating series \eqref{eqn:gen current}. To assure uniformity in our formulas, it will be convenient to study instead of $A_m(x)$ the operator: 
$$
\Phi_m(x) : K_{\bu'} \rightarrow K_{\bu}
$$
given by:
\begin{equation}
\label{eqn:def phi}
\Phi_m(x) = A_m(x) \exp \left[- \sum_{n=1}^\infty \frac {p_n}{n x^n} \cdot \left(\frac {u'}{m^r u} \right)^n \frac {1-q^n}{(1-q_1^n)(1-q_2^n)} \right]
\end{equation}

\begin{proposition}
\label{prop:vertex}
The operator $\Phi_m(x)$ has the following commutation relations with the first $_qW$--algebra generating current:
\begin{equation}
\label{eqn:phi w1}
\Phi_m(x) W_1(y) \cdot \left(1 - \frac {m^r u x}{q^{r-1} u' y} \right) = W_1(y) \Phi_m(x) \cdot m\left(1 - \frac {m^r u x}{q^{r-1} u' y} \right)
\end{equation}
Moreover, it makes sense to ask for the commutation relations between $\Phi_m$ and the $_q$Heisenberg generators $p_n \in \CA_r^{\emph{ext}}$, which take the form:
\begin{equation}
\label{eqn:phi w0}
[\Phi_m(x), p_{\pm k}] = \pm \Phi_m(x) x^{\pm k} \left[1 - \left(\frac {m^r u}{u'} \right)^{\pm k}\right]
\end{equation}

\end{proposition}

\begin{proof} Formula \eqref{eqn:phi w0} is equivalent to \eqref{eqn:Comm 1} when the sign is $+$. When the sign is $-$, we must supplement \eqref{eqn:Comm 1} with the computation of the commutator of:
\begin{equation}
\label{eqn:def z}
Z_m(x) := \exp \left[- \sum_{n=1}^\infty \frac {p_n}{n x^n} \left(\frac {u'}{m^r u} \right)^n \frac {1-q^n}{(1-q_1^n)(1-q_2^n)} \right] = A_m(x)^{-1} \Phi_m(x)
\end{equation}
with $p_{-n}$. To do so, we directly apply \eqref{eqn:heisenberg}:
$$
\left[Z_m(x) , p_{-k} \right] = Z_m(x) x^{-k} \left[ \left(\frac {u'}{m^r u} \right)^k - \left(\frac {q^r u'}{m^r u} \right)^k \right] 
$$
and then use \eqref{eqn:Comm 1} to obtain \eqref{eqn:phi w0}. As for \eqref{eqn:phi w1}, let us rewrite relation \eqref{eqn:Comm 2} as:
$$
A_m(x) \left( \sum_{k \in \BZ} \frac {P_{k,1}}{y^k} \right) \left(1 - \frac {m^r u x}{q^r u' y} \right)  = \left( \sum_{k \in \BZ} \frac {P_{k,1}}{y^k} \right) A_m(x) \left(\frac mq - \frac {m^{r+1} u x}{q^r u' y} \right) 
$$
Using \eqref{eqn:def phi} and \eqref{eqn:w k}, we obtain:
$$
\Phi_m(x) Z_m(x)^{-1} W_1(y) \left(1 - \frac {m^r u x}{q^r u' y} \right)   =
$$
\begin{equation}
\label{eqn:Comm 5}
= W_1(y) \Phi_m(x) Z_m(x)^{-1} \left(\frac mq - \frac {m^{r+1} u x}{q^r u' y} \right) 
\end{equation}
Since $Z_m(x)$ of \eqref{eqn:def z} is an exponential in the annihilation bosons $p_n$, in any normal ordered expression it should be placed at the very right. This means that in the left-hand side of \eqref{eqn:Comm 5}, we must move this exponential past $W_1(y)$. To this end, recall that formula \eqref{eqn:relation 2} and the fact that $p_n = q^{n(r-1)} P_{n,0}$ imply:
$$
\left[ \sum_{k\in \BZ} \frac {P_{k,1}}{y^k} , p_{n} \right] = (1-q_1^n)(1-q_2^n) q^{n(r-1)} y^{n}  \sum_{k\in \BZ} \frac {P_{k,1}}{y^k}
$$
Exponentiating this relation, we conclude that $Z_m(x)^{-1} W_1(y)$ equals:
$$
W_1(y) Z_m(x)^{-1} \exp \left[ \sum_{n=1}^\infty \frac {y^n}{n x^n} \left(\frac {q^r u'}{m^r u} \right)^n(1-q^{-n}) \right] = W_1(y) Z_m(x)^{-1} \cdot \frac {1- \frac {q^{r-1} u' y}{m^r ux}}{1- \frac {q^r u' y}{m^r u x}}
$$
Plugging this formula into \eqref{eqn:Comm 5} gives us precisely \eqref{eqn:phi w1}. 

\end{proof}



\noindent \textbf{Proof of Theorem \ref{thm:main}: } We will prove \eqref{eqn:phi wk} by induction on $k$, whose base case $k=1$ is precisely \eqref{eqn:phi w1}. For the induction step, assume that \eqref{eqn:phi wk} is proved for some $k$ and let us prove it for $k+1$. Applying relation \eqref{eqn:relations} for $k' = 1$ gives us:
$$
W_k(y')W_1(y)\zeta \left(\frac {y'}{yq^k} \right) - W_1(y) W_k(y') \zeta \left( \frac {y}{y'q} \right) =
$$
\begin{equation}
\label{eqn:home}
= \frac {(1-q_1)(1-q_2)}{1-q} \left[\delta \left( \frac {y}{y'q} \right) W_{k+1}(y) - \delta \left(\frac {y'}{yq^k}\right) W_{k+1}(y') \right]
\end{equation}
Since $\delta(z)(1-z) = 0$, we may isolate $W_{k+1}(y)$ in the right-hand side by multiplying both sides of the expression above with $1- \frac {y'}{yq^k}$:
$$
W_k(y')W_1(y)\zeta \left(\frac {y'}{yq^k} \right)\left(1 - \frac {y'}{yq^k}\right) - W_1(y) W_k(y') \zeta \left( \frac {y}{y'q} \right) \left(1 - \frac {y'}{yq^k}\right) =
$$
\begin{equation}
\label{eqn:home}
= \frac {(1-q_1)(1-q_2)}{1-q} \delta \left( \frac {y}{y'q} \right) W_{k+1}(y)  \left(1-\frac 1{q^{k+1}}\right)
\end{equation}
Therefore, one can obtain the series $W_{k+1}(y)$ by taking the constant term in $y'$ of the series in the left-hand side of formula \eqref{eqn:home}. Take the identity of commutators:
$$
[\Phi_m(x), W_k(y')W_1(y)]_{m^{k+1}} = [\Phi_m(x), W_k(y')]_{m^k} W_1(y) + m^k W_k(y') [\Phi_m(x), W_1(y)]_m
$$
and multiply it by:
\begin{equation}
\label{eqn:rea}
\left(1 - \frac {m^rux}{q^{r-1}u'y} \right) \prod_{i=1}^k \left(1 - \frac {m^rux}{q^{r-i}u'y'} \right)
\end{equation}
Then the induction hypothesis implies that $[\Phi_m(x), W_k(y')W_1(y)]_{m^{k+1}}$ multiplied by \eqref{eqn:rea} vanishes. The same reasoning implies that $[\Phi_m(x), W_1(y)W_k(y')]_{m^{k+1}}$ multiplied by \eqref{eqn:rea} vanishes, so we conclude that the same must be true for the right-hand side of \eqref{eqn:home}:
$$
0 = \frac {(1-q_1)(1-q_2)(1-q^{-k-1})}{1-q} \delta \left( \frac {y}{y'q} \right) [\Phi_m(x), W_{k+1}(y)]_{m_{k+1}} \Big( \text{expression \eqref{eqn:rea}} \Big)
$$
Because of the $\delta$ function, we may replace $y'$ by $\frac yq$ in \eqref{eqn:rea} and obtain:
$$
0 = \frac {(1-q_1)(1-q_2)(1-q^{-k-1})}{1-q} \delta \left( \frac {y}{y'q} \right) [\Phi_m(x), W_{k+1}(y)]_{m_{k+1}} \prod_{i=1}^{k+1} \left(1 - \frac {m^rux}{q^{r-i}u'y} \right) 
$$
Taking the constant term in $y'$ of this expression establishes \eqref{eqn:phi wk} for $k+1$. \\

\begin{remark}
\label{rem:future}

In \cite{AGT}, we will improve formula \eqref{eqn:phi wk} by showing that:
\begin{equation}
\label{eqn:future}
[\Phi_m(x), W_k(y)]_{m^k} \cdot \left(1 - \frac {m^r u x}{q^{r-k} u' y} \right) = 0
\end{equation}
for all $k \geq 1$, and that relations \eqref{eqn:future} uniquely determine the operator $\Phi_m(x)$ up to constant multiple. In fact, we leave it as an exercise to the interested reader to prove \eqref{eqn:future} when $k=r$ using the tools we already have on hand, specifically \eqref{eqn:impose rels bis} and \eqref{eqn:phi w0}, and observe that in this case \eqref{eqn:future} is already stronger than \eqref{eqn:phi wk}. \\

\end{remark}

\section{The quantum Miura transformation}
\label{sec:miura}

\subsection{} 
\label{sub:miura}

The original definition of the $_qW$--algebra is through the \textbf{quantum Miura transformation}. In this Section, we will define the $_qW$--algebra of type $\fgl_r$ by analogy with \cite{AKOS} and \cite{FF}, and show that it matches $\CA_{r}$ of Definition \ref{def:w}. Consider the following deformed Heisenberg algebra of type $\fgl_r$:
$$
\CH_r = \BF \left \langle b_n^i \right \rangle_{n \in \BZ \backslash 0}^{1\leq i \leq r}
$$
modulo the commutation relations:
\begin{equation}
\label{eqn:bosons}
[b_{-n}^i, b_n^j] = n (1-q_1^n)(1-q_2^n) \cdot \begin{cases} 1 - q^{-n} & \text{if } i < j \\ 1 & \text{if } i = j \\ 0 & \text{if } i > j \end{cases}
\end{equation}
for all $n>0$. All other commutators are defined to be zero. Consider the elements:
\begin{equation}
\label{eqn:def pn}
p_{n} = \sum_{i=1}^r b_{n}^i q^{n(i-1)} \in \CH_r
\end{equation}
for all $n \in \BZ \backslash 0$. Let us form currents out of these generators:
$$
b^i(x) = \sum_{n \in \BZ \backslash 0} \frac {b^i_n}{|n| x^n}, \qquad \qquad p(x) = \sum_{n \in \BZ \backslash 0} \frac {p_n}{|n| x^n} = \sum_{i=1}^r b\left(\frac x{q^{i-1}} \right)
$$
and consider the normal ordered exponentials:
$$
\Lambda^i(x) \ = \ u_i : \exp \left[ b^i(x)\right] : \ = \ u_i \exp \left[ \sum_{n = 1}^\infty \frac {b^i_{-n}}{n x^{-n}} \right] \exp \left[ \sum_{n = 1}^\infty \frac {b^i_n}{n x^n} \right]
$$
Then by analogy with \cite{AKOS} and \cite{FF}, we \textbf{define} the $_qW$--algebra currents as the expressions in $\Lambda^1(x),...,\Lambda^r(x)$ given by the following equality of difference operators:
\begin{equation}
\label{eqn:quantum miura}
\sum_{k=0}^\infty (-1)^k W_k(x) D^{r-k}_x = 
\end{equation}
$$
= \ :\left(D_x - \Lambda^1(x) \right)\left(D_x - \Lambda^2 \left( \frac xq \right) \right)...\left(D_x - \Lambda^r \left( \frac x{q^{r-1}} \right) \right) :
$$
Here, $D_x$ denotes the difference operator $f(x) \leadsto f(x q)$, and it commutes past the currents $\Lambda^i(x)$ according to the rule:
$$
D_x \Lambda^i(x) = \Lambda^i (xq) D_x
$$
Then one makes sense of \eqref{eqn:quantum miura} by foiling out the right-hand side and moving all of the difference operators to the right of each summand. After doing so, we are left with the following formulas, which are equivalent to \eqref{eqn:quantum miura}:
\begin{equation}
\label{eqn:quantum miura 2}
W_k(x) = \sum_{1\leq i_1 < ... < i_k \leq r} : \Lambda^{i_1}(x)\Lambda^{i_2}\left(\frac xq\right)... \Lambda^{i_k} \left(\frac x{q^{k-1}} \right) : 
\end{equation}
and the normal ordered product simply means that in all expressions, the creation operators $\{b^i_n\}_{n<0}$ must be placed to the left of the annihilation operators $\{b^i_n\}_{n>0}$. Very roughly, one may interpret the $_qW$--currents of \eqref{eqn:quantum miura 2} as normal ordered ``elementary symmetric functions" in the bosonic fields $\Lambda^i(x)$. \\

\noindent Note that, as a $\BF$--vector space, we have:
\begin{equation}
\label{eqn:complet}
\CH_r = \mathop{\mathop{\bigoplus_{s'_1 \geq ... \geq s_{k'}' > 0}^{t_1 \geq ... \geq t_{l} > 0}}^{\cdots}_{\cdots}}_{t'_1 \geq ... \geq t'_{l'} > 0}^{s_1 \geq ... \geq s_{k} > 0} \BF \cdot b_{-s_1}^1...b_{-s_k}^1 ... b_{-t_1}^r... b_{-t_l}^r b_{s_1'}^1...b_{s_{k'}'}^1 ... b_{t_1'}^r... b_{t_{l'}'}^r
\end{equation}
Let us define:
$$
\wCH_r \supset \CH_r
$$
as the completion consisting of infinite sums of the basis vectors \eqref{eqn:complet}, for finite $k+...+l+k'+...+l'$. Note that the coefficients of $W_{k}[[x^{\pm 1}]]$ lie in this completion. \\

\subsection{} One observes several things from \eqref{eqn:quantum miura 2}. First of all, we have $W_0(x) = 1$, while:
\begin{equation}
\label{eqn:mish}
W_r(x) \ = \ : \Lambda^{1}(x)\Lambda^{2}\left(\frac xq\right)... \Lambda^{r} \left(\frac x{q^{r-1}} \right) : \ = u_1...u_r : \exp \left[  p(x)\right] : 
\end{equation}
and $W_k(x) = 0$ for $k>r$. We give the following definition, by analogy with \cite{AKOS}, \cite{FF}: \\

\begin{definition}
\label{def:w algebra}

Define the deformed $W$--algebra $\CB_r \subset \wCH_r$ to be generated by:
$$
\Big\{ W_{d,k} \Big\}_{d\in \BZ}^{1\leq k \leq r} \qquad \text{where} \qquad W_k(x) = \sum_{d\in \BZ} \frac {W_{d,k}}{x^d}
$$

\end{definition}

\begin{proposition}
\label{prop:sanity check}

The elements $W_{d,k} \in \wCH_r$ and $p_n \in \CH_r$ satisfy relations \eqref{eqn:w rel 0 minus}, \eqref{eqn:w rel 0 plus} and \eqref{eqn:w rel} with $c=q^r$. \\

\end{proposition} 

\begin{proof} Formula \eqref{eqn:bosons} and the definition of $p_n$ in \eqref{eqn:def pn} implies that:
\begin{equation}
\label{eqn:dj}
[b_m^i, p_{-n}] \ = \ - \delta_{m-n}^0 \cdot n (1-q_1^n)(1-q_2^n)
\end{equation}
\begin{equation}
\label{eqn:jd}
[b_m^i, p_n] = \delta_{m+n}^0 n (1-q_1^n)(1-q_2^n) q^{n(r-1)}
\end{equation}
for all $n>0$. An easy consequence of these relations is the fact that:
\begin{equation}
\label{eqn:sc}
\left[\Lambda^i(x), p_{- n} \right] = \ - (1-q_1^n)(1-q_2^n) \cdot x^{-n} \Lambda^i(x)
\end{equation}
\begin{equation}
\label{eqn:cs}
\left[\Lambda^i(x), p_{n} \right] = (1-q_1^n)(1-q_2^n)q^{n(r-1)} x^{n} \Lambda^i(x)
\end{equation}
for all $i$. Applying formula \eqref{eqn:quantum miura 2}, formulas \eqref{eqn:sc}--\eqref{eqn:cs} precisely imply \eqref{eqn:w rel 0 minus}--\eqref{eqn:w rel 0 plus} with $c = q^r$, respectively. Meanwhile, note that \eqref{eqn:bosons} implies the following relations:
$$
\Lambda^i(x) \Lambda^i(y) = u_i^2 \exp \left[ \sum_{n > 0 } \frac {b^i_{-n}}{n x^{-n}} \right] \exp \left[ \sum_{n > 0 } \frac {b^i_n}{n x^n} \right] \exp \left[ \sum_{n > 0 } \frac {b^i_{-n}}{n y^{-n}} \right] \exp \left[ \sum_{n > 0 } \frac {b^i_n}{n y^n} \right] = 
$$
$$
= u_i^2 \exp \left[ \sum_{n > 0 } \frac {b^i_{-n}}{n}\left( x^n + y^n \right) \right] \left[ \sum_{n > 0 } \frac {b^i_n}{n}\left(\frac 1{x^n} + \frac 1{y^n} \right) \right]  \frac 1{\zeta \left(\frac yx\right)} = \ : \Lambda^i(x) \Lambda^i(y): \frac 1{\zeta \left(\frac yx\right)}
$$
while:
$$
\Lambda^i(x) \Lambda^j(y) = u_iu_j \exp \left[ \sum_{n > 0 } \frac {b^i_{-n}}{n x^{-n}} + \frac {b^j_{-n}}{n y^{-n}} \right] \left[ \sum_{n > 0 } \frac {b^i_n}{n x^n} + \frac {b^j_n}{n y^n} \right]  = \ : \Lambda^i(x) \Lambda^j(y):
$$
for $i<j$ and:
$$
\Lambda^i(x) \Lambda^j(y) = u_iu_j \exp \left[ \sum_{n > 0 } \frac {b^i_{-n}}{n x^{-n}} \right] \exp \left[ \sum_{n > 0 } \frac {b^i_n}{n x^n} \right] \exp \left[ \sum_{n > 0 } \frac {b^j_{-n}}{n y^{-n}} \right] \exp \left[ \sum_{n > 0 } \frac {b^j_n}{n y^n} \right] = 
$$
$$
= u_iu_j \exp \left[ \sum_{n > 0 } \frac {b^i_{-n}}{n x^{-n}} + \frac {b^j_{-n}}{n y^{-n}} \right] \left[ \sum_{n > 0 } \frac {b^i_n}{n x^n} + \frac {b^j_n}{n y^n} \right] \frac {\zeta\left(\frac y{xq} \right)}{\zeta\left(\frac y{x} \right)} = \ : \Lambda^i(x) \Lambda^j(y): \frac {\zeta\left(\frac xy \right)}{\zeta\left(\frac y{x} \right)}
$$
for $i>j$. Therefore, we conclude that:
\begin{equation}
\label{eqn:boson comm 1}
\Lambda^i(x) \Lambda^i(y) \zeta \left(\frac yx \right) - \Lambda^i(y) \Lambda^i(x) \zeta \left(\frac xy \right) = 0
\end{equation}
because the normal-ordered product is symmetric under $x \leftrightarrow y$, while:
\begin{equation}
\label{eqn:boson comm 2}
\Lambda^i(x) \Lambda^j(y) \zeta \left(\frac yx  \right) - \Lambda^j(y) \Lambda^i(x) \zeta \left(\frac xy  \right) =
\end{equation}
$$
= \ \left( : \Lambda^i(x) \Lambda^j(y): \zeta \left(\frac yx  \right) \text{ for } |y| \ll |x| \right) - \left( : \Lambda^j(y) \Lambda^i(x) : \zeta \left(\frac yx  \right) \text{ for } |x| \ll |y| \right) \ = 
$$
$$
=  \frac {(1-q_1)(1-q_2)}{1-q} \delta\left(\frac {x}{y} \right) : \Lambda^i(x) \Lambda^j(x): -  \frac {(1-q_1)(1-q_2)}{1-q} \delta\left(\frac {x}{yq} \right) : \Lambda^i(x) \Lambda^j \left(\frac xq \right): 
$$
for $i<j$ and:
\begin{equation}
\label{eqn:boson comm 3}
\Lambda^i(x) \Lambda^j(y) \zeta \left(\frac yx  \right) - \Lambda^j(y) \Lambda^i(x) \zeta \left(\frac xy  \right) = 
\end{equation}
$$
= \left( : \Lambda^i(x) \Lambda^j(y): \zeta \left(\frac xy  \right) \text{ for } |y| \ll |x| \right) - \left( : \Lambda^j(y) \Lambda^i(x) : \zeta \left(\frac xy  \right) \text{ for } |x| \ll |y| \right) = 
$$
$$
= - \frac {(1-q_1)(1-q_2)}{1-q}  \delta\left(\frac {y}{x} \right) : \Lambda^i(y) \Lambda^j(y): + \frac {(1-q_1)(1-q_2)}{1-q}  \delta\left(\frac {y}{xq} \right) : \Lambda^i \left( \frac yq \right) \Lambda^j(y):
$$
for $i>j$. In either formula \eqref{eqn:boson comm 2} or \eqref{eqn:boson comm 3}, the equality between the second and third lines follows from a general fact about rational functions, which we now explain. Since $:\Lambda^i(x) \Lambda^j(y):$ is a Laurent polynomial with coefficients $\in \CB_r$, then:
$$
R(x,y) = \ :\Lambda^i(x) \Lambda^j(y): \zeta \left( \frac xy \right) \in \CB_r[[x^{\pm 1}, y^{\pm 1}]] \cdot \frac {(x-y/q_1)(x-y/q_2)}{(x-y)(x-y/q)}
$$
Therefore, the second line of \eqref{eqn:boson comm 3} is the difference between the expansions of $R(x,y) $ at $|y|\ll|x|$ and at $|x| \ll |y|$. The third line of \eqref{eqn:boson comm 3} is equal to:
$$
\sum_{\alpha \notin \{0,\infty\}} \delta \left( \frac {y\alpha}x \right) \underset{x = y\alpha}{\text{Res}} R(x,y)
$$
and so the left and right-hand sides of \eqref{eqn:boson comm 3} are equal. By \eqref{eqn:quantum miura 2}, we have $W_1(x) = \Lambda^1(x)+...+\Lambda^r(x)$, and so relations \eqref{eqn:boson comm 1}, \eqref{eqn:boson comm 2}, \eqref{eqn:boson comm 3} imply the following relation:
$$
W_1(x)W_1(y)\zeta \left(\frac yx \right) - W_1(y)W_1(x) \zeta\left(\frac xy \right) = \frac {(1-q_1)(1-q_2)}{1-q} \cdot
$$
$$
\sum_{1\leq i < j \leq r} \left[ \delta\left(\frac {y}{xq} \right) : \Lambda^j \left( \frac yq \right) \Lambda^i(y):  - \delta\left(\frac {x}{yq} \right) : \Lambda^i(x) \Lambda^j \left(\frac xq \right): \right] = 
$$
$$
= \ \frac {(1-q_1)(1-q_2)}{1-q} \cdot \left[ \delta\left(\frac {y}{xq} \right) W_2(y)  - \delta\left(\frac {x}{yq} \right) W_2(x) \right]
$$
This is precisely \eqref{eqn:w rel} for $k = k'=1$. As for higher values of $k$ and $k'$, the computation follows the same lines as in the proof above, and is presented in detail in Section 2.2 of \cite{O}. We refer the reader to \loccit for the remainder of the computation, as it is simply an exercise in commuting bosonic currents.

\end{proof}

\begin{proposition}
\label{prop:all of them}

Relations \eqref{eqn:w rel} generate the ideal of relations between the elements $\{W_{d,k}\}_{d\in \BZ}^{k \in \{1,...,r\}} \in \CB_r$. Combining this with Proposition \ref{prop:w gens and rels}, we infer that:
$$
\CB_r \cong \CA_r
$$
hence we refer to either of these algebras as ``the deformed $W$--algebra of type $\fgl_r$". \\

\end{proposition}

\begin{proof} Propositions \ref{prop:w gens and rels} and \ref{prop:sanity check} yield an epimorphism $\CA_r \twoheadrightarrow \CB_r$. As shown in Proposition \ref{prop:w gens and rels}, a basis of $\CA_r$ as a $\BF$--module is given by the elements \eqref{eqn:gente}. Thus it is enough to show that these elements are independent in $\CB_r \subset \wCH_r$. Since $\wCH_r$ is the completion of a free $\BF$--module with basis given by the normal-ordered products of $\{b_n^i\}_{n \in \BZ \backslash 0}^{1 \leq i \leq r}$, it suffices to show that the elements \eqref{eqn:gente} are independent in $\wCH_r$. To this end, we claim that it suffices to show independence when we specialize $q_1 = 1$ and leave $q = q_2$ generic (the reader may object to the latter claim, given that $\CH_r$ is an infinite-dimensional $\BF$--module, but one can repeat the contents of this paragraph with the field $\BF = \BQ(q_1,q_2)$ replaced with the ring $\BZ[q_1^{\pm 1}, q_2^{\pm 1}] \otimes_{\BZ[q]} \BQ(q)$, over which all of our algebras are well-defined; all we are saying in this sentence is that $1 - q_1$ never appears in the denominator of any commutation relation). \\

\noindent Therefore, setting $q_1=1$, the algebras $\CH_r \subset \wCH_r$ become commutative, and:
\begin{equation}
\label{eqn:w specialized}
W_{d,k} = \sum_{1 \leq i_1 < ... < i_k \leq r}^{e_1+...+e_k = d} \lambda^{i_1}_{e_1}... \lambda^{i_k}_{e_k} \cdot q^{\sum_{i=1}^k (i-1)e_i}
\end{equation}
where:
$$
\lambda_e^i = \text{coefficient of }x^{-e} \text{ in } u_i \exp \left[ \sum_{n=1}^\infty  \frac {b_{-n}^i}{nx^{-n}} \right] \exp \left[ \sum_{n=1}^\infty  \frac {b_n^i}{nx^n} \right] 
$$
Since unordered monomials in the $b_e^i$ form a linear basis of the commutative algebra $\CH_r|_{q_1=1}$, unordered monomials in the $\lambda_e^i$ are linearly independent in $\wCH_r|_{q_1=1}$. We will consider unordered products in the $W_{d,k}$ (since the algebra is commutative, the order is immaterial) and define the following total ordering:
\begin{equation}
\label{eqn:mony}
\prod_{i=1}^r W_{d_i^1,i}... W_{d_i^{m_i},i} \qquad \text{is greater than} \qquad \prod_{i=1}^r W_{e_i^1,i}... W_{e_i^{n_i},i}
\end{equation}
(we always write such products assuming that $d_i^1 \leq ... \leq d_i^{m_i}$ and $e_i^1 \leq ... \leq e_i^{n_i}$) if the transpose of the partition $\mu = 1^{m_1}...r^{m_r}$ is greater than the transpose of the partition $\nu = 1^{n_1}...r^{n_r}$ in lexicographic ordering. If the two partitions $\mu$ and $\nu$ are equal, then we impose the order \eqref{eqn:mony} if $(d_r^{m_r},...,d_r^1,....,d_1^{m_1},...,d_1^1)$ is greater than $(e_r^{n_r},...,e_r^1,....,e_1^{n_1},...,e_1^1)$ lexicographically. If there exists a linear relation between the unordered products of $W_{d,k}$ in the algebra $\wCH_r|_{q_1=1}$, we may write it as:
\begin{equation}
\label{eqn:fictitious relation}
\prod_{i=1}^r W_{d_i^1,i}... W_{d_i^{m_i},i} \in \sum^{\text{smaller}}_{\text{products}} \BQ(q) \prod_{i=1}^r W_{e_i^1,i}... W_{e_i^{n_i},i} 
\end{equation}
The contradiction to the existence of such a relation is that the term:
$$
\underbrace{\lambda^1_{d_1^1}...\lambda^1_{d_1^{m_1}} \lambda^1_{-M} ... \lambda^1_{-M}}_{m_1 + ... + m_r \text{ terms}}   \underbrace{\lambda^2_{d_2^1+M}...\lambda^2_{d_2^{m_2}+M} \lambda^{2}_{-M} ... \lambda^{2}_{-M}}_{m_2 + ... + m_r \text{ terms}} \ ... \  \underbrace{\lambda^r_{d_r^1+M(r-1)}...\lambda^r_{d_r^{m_r}+M(r-1)}}_{m_r \text{ terms}}
$$
(for some $-M \leq d_i^j \ \forall i,j$) appears in the left-hand side of \eqref{eqn:fictitious relation} upon applying \eqref{eqn:w specialized}, but does not appear in the right-hand side in virtue of our choice of ordering. 


\end{proof}








\subsection{}
\label{sub:comparison}

Let us recall from \cite{AKOS}, \cite{FF} the construction of the $_qW$--algebra of type $\fsl_r$. Consider the deformed Heisenberg algebra of type $\fsl_r$:
$$
\tCH_r = \BF \left \langle h_n^i \right \rangle_{n \in \BZ \backslash 0}^{1\leq i \leq r}
$$
subject to the commutation relations:
\begin{equation}
\label{eqn:bosons 2}
[h_{-n}^i, h_n^j] = n (1-q_1^n)(1-q_2^n) \cdot \frac {1-q^{(r\delta_i^j-1)n}}{1-q^{rn}}q^{rn\delta_{i>j}}
\end{equation}
for all $n>0$, the fact that all other commutators vanish, and the linear relation: 
$$
\sum_{i=1}^r h^i_n q^{n(i-1)} = 0
$$
for all $n \in \BZ \backslash 0$. Then the $_qW$--algebra of type $\fsl_r$, denoted by $\tCB_r \subset \widehat{\tCH_r}$, is defined to be generated by the coefficients of currents $\wW_k(x)$ from the following formula:
\begin{equation}
\label{eqn:quantum miura 3}
:\left(D_x - \wLambda^1(x) \right)...\left(D_x - \wLambda^r\left( \frac x{q^{r-1}} \right) \right) :\  = \sum_{k=0}^\infty (-1)^k \wW_k(x) D_x^{r-k}
\end{equation}
where:
$$
\wLambda^i(x) = u_i : \exp \left[ h^i(x)\right] : \ = \ u_i \exp \left[ \sum_{n > 0 } \frac {h^i_{-n}}{n x^{-n}} \right] \exp \left[ \sum_{n > 0 } \frac {h^i_n}{n x^n} \right]
$$
By analogy with \eqref{eqn:quantum miura 2}, we may rewrite formula \eqref{eqn:quantum miura 3} as:
\begin{equation}
\label{eqn:quantum miura 4}
\wW_k(x) = \sum_{1\leq i_1 < ... < i_k \leq r} : \wLambda^{i_1}(x) \wLambda^{i_2}\left(\frac xq\right)... \wLambda^{i_k} \left(\frac x{q^{k-1}} \right) : 
\end{equation}
for any $k \in \{1,...,r\}$. \\

\begin{proposition}
\label{prop:iso}

The map $h^i_n \mapsto b^i_n - p_n \cdot \frac {1-q^n}{1-q^{rn}}$ gives rise to a homomorphism:
\begin{equation}
\label{eqn:emb}
\tCB_r \rightarrow \CB_r
\end{equation}
Moreover, the $_q$Heisenberg generators $p_m$ commute with the image of \eqref{eqn:emb}, and therefore we obtain a homomorphism:
\begin{equation}
\label{eqn:emb iso}
\tCB_r \otimes _q\emph{Heisenberg} \rightarrow \CB_r
\end{equation}
In terms of currents, the map \eqref{eqn:emb} takes the form:
\begin{equation}
\label{eqn:jon}
\wW_k(x) = \exp \left[ - \sum_{n=1}^\infty \frac {p_{-n}}{nx^{-n}} \frac {1-q^{-kn}}{1-q^{-rn}} \right] W_k(x) \exp \left[ - \sum_{n=1}^\infty \frac {p_{n}}{nx^{n}} \frac {1-q^{kn}}{1-q^{rn}} \right]
\end{equation}
for all $k \in \{1,...,r\}$. \\

\end{proposition}

\noindent We note a slight imprecision in the definition of the homomorphisms \eqref{eqn:emb} and \eqref{eqn:emb iso}, and we will leave it as is, because we will not revisit the issue in the present paper. These homomorphisms are well-defined only if the $_q$Heisenberg generators $\{p_n\}_{n\in \BZ \backslash 0}$, as well as arbitrary normal-ordered exponentials of these generators, lie in $\CB_r$. This is not automatically obvious, because although the series coefficients $\{W_{r,d}\}_{d\in \BZ}$ of the normal-ordered exponential \eqref{eqn:mish} lie in $\CB_r$, the individual $p_n$ can be obtained from the $W_{r,d}$ only if we suitably complete $\CB_r$. \\

\begin{proof} It is easy to see that if the $b^i_n$ satisfy the commutation relations \eqref{eqn:bosons} and the $p_m$ are defined as the linear combinations \eqref{eqn:def pn}, then: 
\begin{equation}
\label{eqn:ricky}
h_n^i := b^i_n - p_n \cdot \frac {1-q^n}{1-q^{rn}}
\end{equation}
satisfy the commutation relations \eqref{eqn:bosons 2}. This induces an embedding $\tCH_r \subset \CH_r$, and it is clear that $p_m$ commute with the image of this map, because:
$$
[p_m, h_n^i] = \left[ \sum_{j=1}^r b_m^j q^{m(j-1)}, b_n^i - \sum_{j=1}^r b_n^j q^{n(j-1)}\frac {1-q^n}{1-q^{rn}} \right] = \delta_{m+n}^0 n(1-q_1^n)(1-q_2^n) \cdot 
$$
$$
\left(q^{n(1-i)} + \sum_{j=1}^{i-1} (1-q^{-n})q^{n(1-j)} - \sum_{j=1}^r  \frac {1-q^n}{1-q^{rn}} - \sum_{j<j'} (1-q^{-n}) q^{n(j'-j)} \frac {1-q^n}{1-q^{rn}} \right) = 0
$$
Therefore, \eqref{eqn:emb} follows from the fact that formulas \eqref{eqn:quantum miura 2} and \eqref{eqn:quantum miura 4} are connected by \eqref{eqn:jon}. In more detail, we note that as a consequence of \eqref{eqn:quantum miura 4}, we have:
\begin{equation}
\label{eqn:jack}
\wW_k(x) = \sum_{1\leq i_1 < ... < i_k \leq r} : \wLambda^{i_1}(x)\wLambda^{i_2}\left(\frac xq\right)... \wLambda^{i_k} \left(\frac x{q^{k-1}} \right) : \ = \sum_{1\leq i_1 < ... < i_k \leq r}
\end{equation}
$$
: \prod_{j=1}^k \exp \left[ - \sum_{n > 0} \frac {p_{-n}}{n x^{-n}} q^{-n(j-1)} \frac {1-q^{-n}}{1-q^{-rn}} \right] \Lambda^{i_j}\left( \frac x{q^{j-1}} \right) \exp \left[ - \sum_{n > 0} \frac {p_{n}}{n x^{n}} q^{n(j-1)}\frac {1-q^{n}}{1-q^{rn}} \right] : 
$$
$$
= \exp \left[ - \sum_{n>0} \frac {p_{-n}}{nx^{-n}} \frac {1-q^{-kn}}{1-q^{-rn}} \right] \sum_{1\leq i_1 < ... < i_k \leq r} : \prod_{j=1}^k \Lambda^{i_j}\left( \frac x{q^{j-1}} \right) : \exp \left[ - \sum_{n>0} \frac {p_{n}}{nx^{n}} \frac {1-q^{kn}}{1-q^{rn}} \right]
$$
which is precisely the right-hand side of \eqref{eqn:jon}. 

\end{proof}

\subsection{}
\label{sub:stable basis} 

Let us now recall the construction of the stable basis isomorphism introduced by Maulik--Okounkov (\cite{MO}) in the context of symplectic resolutions. In the case of the moduli space of rank $r$ sheaves on the plane, this construction takes the form:
\begin{equation}
\label{eqn:level r iso}
\text{Stab} : F_{u_1} \otimes ... \otimes F_{u_r} \stackrel{\cong}\longrightarrow K
\end{equation}
where the RHS is the $K$--theory group over equivariant parameters $\bu = (u_1,...,u_r)$. The algebra $\CH_r$ acts on the left-hand side via:
\begin{equation}
\label{eqn:boson miura}
b_{-n}^i = p_{-n}^{(i)}, \qquad \qquad b_n^i = p_n^{(i)} + \sum_{j=1}^{i-1} (1-q^{-n}) p_n^{(j)}
\end{equation}
where $p_{\pm n}^{(i)}$ denotes the operators \eqref{eqn:vertex 0} acting in the $i$--th factor of the tensor product in \eqref{eqn:level r iso}. Since the deformed $W$--algebra maps into the completion of $\CH_r$ according to Definition \ref{def:w algebra} and Proposition \ref{prop:all of them}, we obtain actions of the deformed $W$--algebra on both sides of relation \eqref{eqn:level r iso}. The hallmark of \eqref{eqn:level r iso} is that it is an isomorphism of deformed $W$--algebra modules. Let us describe this isomorphism. \\

\noindent Let $\{s_\lambda \in F_{u}\}_{\lambda \text{ partition}}$ denote the basis of plethystically modified Schur functions. We will shortly recall the definition of stable basis $\{s_\bla \in K\}_{\bla \text{ a } r\text{--tuple of partitions}}$ from \loccitt, but let us first state the main idea of the present Subsection: \\


\begin{claim}
\label{claim:stab}

The assignment of \eqref{eqn:level r iso}, which is explicitly defined as:
\begin{equation}
\label{eqn:def stab}
\emph{Stab} \left( s_{\la^1} \otimes ... \otimes s_{\la^r} \right) = s_{(\la^1,...,\la^r)}
\end{equation}
intertwines the deformed $W$--algebra action on the LHS defined by \eqref{eqn:quantum miura 2} and \eqref{eqn:boson miura}, with the action of the same algebra on the RHS that we defined in Section \ref{sec:geom}.

\end{claim}

\tab 
The proof of Claim \ref{claim:stab} involves comparing the Hopf algebra that Maulik--Okounkov associate to the moduli space of rank $r$ sheaves, with the algebra $\CA_r$ that we construct as a subquotient of the completed double shuffle algebra. Their construction is detailed in the cohomological case in \cite{MO}, but the $K$--theoretic version has not yet been published. Therefore, we do not attempt a proof of Claim \ref{claim:stab}, and include its statement only for the sake of our exposition. We believe showing Claim \ref{claim:stab} would be a good challenge for anyone interested in studying \cite{MO}. \\

\noindent Let us recall the definition of stable bases in the greater generality in which they were defined by \cite{MO}. Consider any symplectic resolution $\CM$ acted on by a generic rank $1$ torus $\BC^*$, and suppose for simplicity that the fixed point set $\CM^{\BC^*}$ is isolated and indexed by symbols $\bla$. The fixed point set is partially ordered by the transitive closure of the relation which has $\bmu \unlhd \bla$ if the fixed point $\bmu$ is in the closure of the attracting locus of $\bla$ under the one-parameter subgroup generated by the action $\BC^* \curvearrowright \CM$. Then \loccit define the \textbf{stable basis} corresponding to this data as:
$$
\{ s_\bla \}_{\bla \text{ fixed points}} \in K_{\BC^*}(\CM)
$$
which is upper triangular in the basis of fixed points:
$$
s_\bla = \sum_{\bmu \unlhd \bla} c_\bla^\bmu(t) |\bmu \rangle
$$
where $t$ denotes the equivariant parameter of $\BC^* \curvearrowright \CM$. The coefficients $c_\bla^\bmu(t) \in \BZ[t^{\pm 1}]$ are uniquely determined by the following conditions:
\begin{equation}
\label{eqn:stable 1}
c_\bla^\bla(t) = \ \left[ \wedge^\bullet \Tan_{\bla}^{\text{attr}}\CM \right] \ \in \ K_{\BC^*}(\pt) \ = \ \BZ[t^{\pm 1}]
\end{equation}
for all $\bla$, while for all $\bmu \lhd \bla$ we require that:
\begin{equation}
\label{eqn:stable 2}
c_\bla^\bmu(t) = t^{\#_{\min}} \text{coefficient} + ... + t^{\#_{\max}-1} \text{coefficient} 
\end{equation}
for certain coefficients that do not depend on $t$, where $\#_{\min}$ and $\#_{\max}$ denote the smallest and largest power of $t$ that appears in the Laurent polynomial $\left[ \wedge^\bullet \Tan_{\bmu}^{\text{attr}}\CM \right]$. When $\CM$ is the space of framed rank 1 sheaves (i.e. the Hilbert scheme of points on $\BC^2$), \loccit observe that the $K$--theory classes $s_\la$ thus defined coincide with plethystically modified Schur functions under the isomorphism $\oplus_{n = 0}^\infty K_{\BC^* \times \BC^*}((\BC^2)^{[n]}) \cong F_u$ given by the Bridgeland--King--Reid--Haiman equivalence. In the higher rank case, the stable basis is described by Claim \ref{claim:stab}. \\

\begin{remark}
\label{rem:slope}

The definition of the $K$--theoretic stable basis of Maulik and Okounkov also depends on an extra ``slope" parameter, which in this paper is equal to 0. In the setup of the moduli spaces of rank $r$ sheaves on the plane, the slope is a real number with which we shift the numbers $\#_{\min}$ and $\#_{\max}$ in relation \eqref{eqn:stable 2}. This results in an infinite family of stable bases, as were studied in \cite{GN} and \cite{Pieri}. \\

\end{remark}

\noindent Finally, let us reiterate the fact that the $K$--theoretic stable basis would make \eqref{eqn:level r iso} into an isomorphism of deformed $W$--algebra representations (see \cite{MO} for the cohomological version). The corresponding coproduct on the algebra $\CA$, as well as the action of the $W$--currents on the left-hand side of \eqref{eqn:level r iso}, have been studied in \cite{Kernel}. In the present paper, we take the opposite point of view from \loccit and define the level $r$ representation as the right-hand side of \eqref{eqn:level r iso}. \\

\section{The classical limit}
\label{sec:classical}

\subsection{} The usual $W$--algebra is the ``classical limit" of the $_qW$--algebra. Specifically, this limit means that we specialize the parameters studied in this paper to:
$$
q_1 = e^{\e \hbar_1}, \quad q_2 = e^{\e \hbar_2}, \quad q = e^{\e \hbar} \quad \text{ where } \hbar = \hbar_1+\hbar_2
$$
$$
\bu = e^{\e \barbu}, \qquad u_i = e^{\e \baru_i}, \qquad m = e^{\e \barm}, \qquad y = e^{\e \bary}
$$
and we take the leading term of all our formulas in the limit $\e \rightarrow 0$. This is quite natural, since many of our formulas are sums of products of expressions:
$$
1 - q_1^{-k_1}q_2^{-k_2}... =  1- e^{-\e(k_1\hbar_1+k_2\hbar_2+...)} = \e(k_1\hbar_1+k_2\hbar_2+...) + O(\e^2)
$$
for various integers $k_1,k_2,...$. As a matter of notation, a bar over a symbol means the classical limit of that symbol from the $q$--deformed theory. Let us recall the additive shuffle algebra studied in \cite{Ext}:
$$
\barCS = \bigoplus_{k=0}^\infty \left \{ \text{symmetric rational functions } \frac {r(\barz_1,...,\barz_k)}{\prod_{1\leq i \neq j \leq k} (\barz_i - \barz_j - \hbar)} \right \}
$$
that satisfy the wheel conditions $r|_{\{\barz_1 - \barz_2, \barz_2 - \barz_3,\barz_3 - \barz_1\} = \{\hbar_1, \hbar_2,-\hbar\}} = 0$. The shuffle product is defined as in \eqref{eqn:mult}, with the rational function $\zeta$ replaced by its additive analogue:
\begin{equation}
\label{eqn:bar zeta}
\barzeta(\barz) = \lim_{\e \rightarrow 0} \zeta(e^{\e\barz}) = \frac {(\barz+\hbar_1)(\barz+\hbar_2)}{\barz(\barz+\hbar)}
\end{equation}
The analogue of our level $r$ modules in the case of the additive shuffle algebra is the equivariant cohomology of the same moduli spaces $\CM_n$ of rank $r$ sheaves on $\BP^2$:
$$
H = \bigoplus_{n=0}^\infty H^*_T(\CM_{n}) \bigotimes_{H^*_T(\pt)} \text{Frac}(H^*_T(\pt))
$$
This vector space has a basis of fixed points $|\barbla \rangle$, also indexed by $r$--tuples of partitions. By analogy with Theorem \ref{thm:action}, we have two opposite $\barCS$ actions on $H$:
\begin{equation}
\label{eqn:fixed minus clas}
\langle \barbla | \barR^\leftarrow | \barbmu \rangle = \barR(\blamu) \prod_{\bsq \in \blamu} \left[ \frac {\hbar_1\hbar_2}{-\hbar} \barzeta\left( \barchi_\bsq - \barchi_{\bmu} \right) \bartau \left( \barchi_\bsq + \hbar \right) \right]
\end{equation}
\begin{equation}
\label{eqn:fixed plus clas}
\langle \barbmu | \barR^\rightarrow |\barbla \rangle = \barR(\blamu) \prod_{\bsq \in \blamu} \left[ \frac {\hbar_1\hbar_2}{-\hbar} \barzeta\left( \barchi_\bla - \barchi_\bsq \right)^{-1} \bartau \left( \barchi_\bsq  \right)^{-1} \right]
\end{equation}
where $\barR(\blamu) = \barR(..., \barchi_\sq,...)_{\sq \in \blamu}$ and:
$$
\barchi_\sq = \baru_k + i \hbar_1 + j \hbar_2
$$
for a square $\sq$ lying at coordinates $(i,j)$ in the $k$--th constituent partition of the $r$--tuple $\bla$. Moreover, we write $\bartau$ for the additive analogue of the function $\tau$ of \eqref{eqn:def tau}:
\begin{equation}
\label{eqn:bar tau}
\bartau(\barz) = \prod_{i=1}^r (\baru_i - \barz)
\end{equation}
Finally, we have a series of diagonal operators on $H$, with matrix coefficients:
\begin{equation}
\label{eqn:fixed diag clas}
\left \langle \barbmu | \barE(\bary) | \barbnu \right \rangle = \delta_\bnu^\bmu \prod_{i=1}^r (\bary - \baru_i) \prod_{\sq \in \bnu} \barzeta (\barchi_\sq - \bary)
\end{equation}

\subsection{}
\label{sub:degeneration}

To take the classical limit $\CS \leadsto \barCS$ at the shuffle algebra level, simply set the variables to $z_i = e^{\e \barz_i}$ and take the leading order term as $\e \rightarrow 0$. Explicitly, for a shuffle element $R$ of the form $\eqref{eqn:our r}$, its classical limit is:
\begin{equation}
\label{eqn:our r bar}
R(z_1,...,z_k) \leadsto \barR(\barz_1,...,\barz_k) = \sym \left[ \frac {\text{l.o. } \rho\left(e^{\e \barz_1},...,e^{\e \barz_k} \right)}{\prod_{i=1}^{k-1} \left(\barz_i - \barz_{i+1} - \hbar \right)} \prod_{i < j} \barzeta (\barz_i - \barz_j) \right]
\end{equation}
where ``l.o." stands for the leading order term in $\e$. Note that we do not claim that the assignment \eqref{eqn:our r bar} has any property whatsoever (it is not a homomorphism). In particular, the classical limit of the shuffle elements \eqref{eqn:pam pam} is: 
\begin{equation}
\label{eqn:awesome shuffle 3}
T(x,y) \leadsto \barT(x,\bary) = \sum_{d=0}^\infty \sym \left[ \frac {x^d}{(\barz_d - \bary) \prod_{i=1}^{d-1} \left(\barz_i - \barz_{i+1} - \hbar \right)} \prod_{i < j} \barzeta (\barz_i - \barz_j) \right] 
\end{equation}
To give meaning to the procedure $\leadsto$ outlined above, we will connect the action of $T(x,y)$ on $K$ with the action of $\barT(x,\bary)$ on $H$. Define the classical limit $K \leadsto H$ by:
\begin{equation}
\label{eqn:renormalization}
| \bla \rangle \leadsto \lim_{\e \rightarrow 0} \ \e^{r|\bla|} |\bla \rangle  = : |\barbla\rangle
\end{equation}
Because of this renormalization of the basis vectors, the classical limit of Theorem \ref{thm:ldu} is given by the following limit of the operators \eqref{eqn:w currents}: \\

\begin{proposition}
\label{prop:class lim}

Let:
\begin{equation}
\label{eqn:def def}
\bar{W}(x,\bary) = \bar{T}(x,y^{-1})^\leftarrow \cdot \bar{E}(\bary) \cdot \bar{T}(x,y)^\rightarrow
\end{equation}
Then we have $W(x,yD_x) \leadsto \bar{W}(x,\bary)$, by which we mean the fact that:
\begin{equation}
\label{eqn:classical limit}
\langle \bmu | W(x,y D_x )  |\bla \rangle = \e^r \langle \barbmu | \bar{W}(x,\bar{y})  |\barbla \rangle + O(\e^{r+1})
\end{equation}
$$$$
\end{proposition}

\begin{proof} By applying \eqref{eqn:w currents}, we see that:
$$
\langle \bmu | W(x,y D_x )  |\bla \rangle = \langle \bmu | W(x,e^{\e\bar{y}+\e \hbar x\partial_x} )  |\bla \rangle = 
$$
$$
= \sum_{\bnu \subset \bla \cap \bmu} \langle \bmu |T(x^{-1},e^{\e\bar{y}+\e \hbar x\partial_x})^\leftarrow| \bnu \rangle \langle \bnu | E(e^{\e\bar{y}+\e \hbar x \partial_x}) | \bnu \rangle \langle \bnu | \barT(x e^{\e\hbar},e^{\e\bar{y}+\e \hbar x\partial_x})^\rightarrow | \bla \rangle =
$$
$$
= \e^r \sum_{\bnu \subset \bla \cap \bmu} \langle \barbmu | \barT(x^{-1},\bary)^\leftarrow| \barbnu \rangle \langle \barbnu | \barE(\bary) | \barbnu \rangle \langle \barbnu | \barT(x,\bary)^\rightarrow | \barbla \rangle + O(\e^{r+1})  = \text{RHS of \eqref{eqn:classical limit}}
$$
The next-to-last equality holds because the leading powers of $\langle \bmu |T(x^{-1},e^{\e\bar{y}+\e \hbar x\partial_x})^\leftarrow| \bnu \rangle$, $\langle \bnu | E(e^{\e\bar{y}+\e \hbar x\partial_x}) | \bnu \rangle$, $\langle \bnu | \barT(x e^\hbar,e^{\e\bar{y}+\e \hbar x\partial_x})^\rightarrow | \bla \rangle$ are $\e^0$, $\e^r$, $\e^0$, respectively, due to the normalization \eqref{eqn:renormalization}. The leading order terms are given by the corresponding coefficients of $\barT$, $\barE$, $\barT$, respectively, as in the equation above.

\end{proof}

\subsection{}
\label{sub:classical bosons} 

Formulas \eqref{eqn:fixed minus clas}--\eqref{eqn:fixed plus clas} may be interpreted as the classical limit of the algebra $\CA_r$ acting on $H$, but this action had already been presented in a different language in \cite{MO} and \cite{SV2}. In \cite{Ext}, we asked about a geometric interpretation of the action of the usual $W$--currents on the cohomology group $H$. Since the operators $\barT(x,\bary)^\leftarrow$ and $\barT(x,\bary)^\rightarrow$ are given by the correspondences $\fZ_d$ as $d$ ranges over $\BN$ (this is the content of the additive version of Theorem \ref{thm:corr}, see \cite{Ext}), formulas \eqref{eqn:def def} and \eqref{eqn:classical limit} would give a manifestly geometric answer, if we could relate the usual $W$--currents with $\bar{W}(x,\bary)$. We will prove this in Proposition \ref{prop:classical limit} and Theorem \ref{thm:ldu clas}, but let us start by recalling the classical limit of the quantum Miura transformation \eqref{eqn:quantum miura}. Specifically, set the bosons of Subsection \ref{sub:miura} equal to:
$$
b^i(x) \ = \ \e x \partial_x ( \barb^i(x) ) + O(\e^2)
$$
in such a way that relation \eqref{eqn:bosons} yields in the limit $\e \rightarrow 0$:
\begin{equation}
\label{eqn:bosons classical}
[\barb_{m}^i, \barb_n^j] = \delta_j^i \delta_{m+n}^0 m \hbar_1 \hbar_2
\end{equation}
Note that:
$$
\Lambda^i(x) = 1 + \e \barLambda^i(x) + O(\e^2) \qquad \text{where} \qquad \bar\Lambda^i(x) = \baru_i -  \sum_{n \in \BN} \frac {\barb^i_n}{x^n} + \sum_{n \in \BN} \frac {\barb^i_{-n}}{x^{-n}}
$$
Then the quantum Miura transformation \eqref{eqn:quantum miura} reads:
\begin{equation}
\label{eqn:quantum miura 5}
\sum_{k=0}^r (-1)^k W_k(x) e^{\e \hbar (r-k)x\partial_x} = \ : \prod_{i=1}^r \left(e^{\e \hbar x \partial_x} - 1 - \e \bar\Lambda^i(x) - O(\e^2)\right)  : \ =  
\end{equation}
$$
= \e^r : \prod_{i=1}^r (\hbar x \partial_x - \barLambda^i(x)) : + O(\e^{r+1}) =  \e^r \sum_{k=0}^r (-1)^k \barW_k(x) (\hbar x \partial_x)^{r-k} + O(\e^{r+1})
$$
The right-hand side of \eqref{eqn:quantum miura 5} \textbf{defines} the currents of the $W$--algebra of type $\fgl_r$ (our convention differs from the usual one by a sign). We expect this algebra to be the tensor product of a Heisenberg algebra and the usual $W$--algebra of type $\fsl_n$ (defined in \cite{FLu}, see also \cite{BoS} for reference). Specifically, the Heisenberg subalgebra is generated by:
$$
\barW_1(x) = \sum_{i=1}^r \barLambda^i(x) = \baru_1 + ... + \baru_r -  \sum_{n \in \BN} \frac {B_n}{x^n} + \sum_{n \in \BN} \frac {B_{-n}}{x^{-n}}
$$
Explicitly, we have $B_n = \barb^1_n+...+\barb^r_n$ and so $[B_m, B_n] = \delta_{m+n}^0 r m \hbar_1\hbar_2$. \\


\subsection{} 
\label{sub:q to clas}

Formula \eqref{eqn:quantum miura 5} has an interesting consequence: the left-hand side vanishes to order $\e^r$. Identifying the coefficients of $(\hbar x\partial_x)^{r-i} \e^{r-i+j}$ for all $0 \leq j \leq i \leq r$ in the two sides of equation \eqref{eqn:quantum miura 5} gives us:
$$
\text{coefficient of }\e^j \text{ in } \sum_{k=0}^r (-1)^k W_k(x) \frac {(r-k)^{r-i}}{(r-i)!} = \delta_j^i (-1)^i \barW_{i}(x)
$$
$$
\Longrightarrow \quad \sum_{k=0}^r (-1)^k W_k(x) \frac {(r-k)^{r-i}}{(r-i)!}  = \e^i (-1)^i  \barW_i(x) + O(\e^{i+1})
$$
Since this relation holds for all $0 \leq i \leq r$, we infer that:
$$
\sum_{k=0}^r (-1)^k W_k(x) \pi(r-k)  = \e^i (-1)^i  \barW_i(x) + O(\e^{i+1})
$$
for any polynomial $\pi$ of degree exactly $r-i$, with leading term $1/(r-i)!$ We will take this polynomial to be $\pi(n) = {n \choose r-i}$, so we obtain the relation:
\begin{equation}
\label{eqn:limitation}
\barW_i(x) = \lim_{\e \rightarrow 0} \e^{-i} \sum_{k=0}^{i} (-1)^k W_k(x) {r-k \choose r-i}
\end{equation}
Recall the notation $W(x,y) = \sum_{k=0}^r W_k(x) (-y)^{-k}$ from Subsection \ref{sub:q-W currents}. With it, we conclude the following formula for the classical limit of $_qW$--currents: \\

\begin{proposition}
\label{prop:classical limit}

The currents of the usual $W$--algebra are given by:
\begin{equation}
\label{eqn:classical current}
\barW_i(x) = \lim_{\e \rightarrow 0} \e^{-i} \cdot \frac { \partial_y^{r-i}}{(r-i)!} \Big[ y^r W(x,y) \Big] \Big|_{y=1} 
\end{equation}
$$$$
\end{proposition}

\noindent Set $y = e^{\e\bary}$ and recall from \eqref{eqn:classical limit} that $W(x,y)$ is of order $\e^r$ when degenerating the representation on $K$--theory into the representation on cohomology $H$ (the operator $D_x$ is of order 0 in $\e$, so it does not contribute anything to the leading order term). The $r-i$ derivatives in \eqref{eqn:classical current} can bring down no more than $r-i$ powers of $\e^{-1}$, so we conclude that only the leading term of \eqref{eqn:classical limit} contributes to the limit \eqref{eqn:classical current}. Explicitly, we obtain: \\

\begin{theorem}
\label{thm:ldu clas}

The action of the $W$--algebra on the level $r$ module $H$ is given by:
\begin{equation}
\label{eqn:factor clas}
\barW_i(x) = \frac {\partial_{\bary}^{r-i}}{(r-i)!} \left[ \bar{W}(x,\bary) \right] \Big|_{\bary = 0} 
\end{equation}
Combining this result with \eqref{eqn:def def}, we obtain the following factorization for the currents of the classical $W$--algebra in the level $r$ representation $H$:
\begin{equation}
\label{eqn:factor clas 2}
\barW_i(x) = \frac {\partial_{\bary}^{r-i}}{(r-i)!} \left[ \barT(x^{-1}, \bary)^\leftarrow \cdot \barE(\bary) \cdot \barT(x, \bary)^\rightarrow \right] \Big|_{\bary = 0}
\end{equation}
The expression $\bar{W}(x,\bary)$ is interpreted as a rational function in $\bary$, via \eqref{eqn:def def}. \\

\end{theorem}




\subsection{} The classical analogue of $A_m(x)$ is given by the Chern polynomial of the Ext bundle, used as a correspondence on cohomology (\cite{Ext}). We will denote it by:
$$
\barA_{\barm}(x) : H_{\barbu'} \longrightarrow H_{\barbu}
$$
where $H_{\barbu}$ and $H_{\barbu'}$ are the cohomology groups of two different moduli spaces of sheaves, acted on by two different rank $r$ tori, with equivariant parameters given by $\barbu$ and $\barbu'$, respectively. All we need to use is the classical limit of formula \eqref{eqn:matrix coefficients}:
\begin{equation}
\label{eqn:callatis}
\langle \bla | A_m(x) |\bla' \rangle = \langle \barbla | \barA_{\barm}(x) |\barbla' \rangle + O(\e)
\end{equation}
since the operator $\barA_{\barm}(x)$ can be easily seen to have the following matrix coefficients (recall the notation of \eqref{eqn:bar zeta}, \eqref{eqn:bar tau} and \eqref{eqn:renormalization}):
$$
\langle \barbla | \barA_{\barm}(x) |\barbla' \rangle = \frac {\prod_{\sq \in \bla} \bartau'(\barm + \barchi_{\sq}) \prod_{\sq' \in \bla'} \bartau(\barchi_{\sq'} + \hbar - \barm) \prod^{\sq \in \bla}_{\sq' \in \bla'} \barzeta \left( \barchi_\sq' - \barchi_\sq - \barm \right)}{x^{|\bla'| - |\bla|} \prod_{\sq \in \bla'} \bartau'(\barchi_\sq) \prod_{\sq' \in \bla'} \bartau'(\barchi_{\sq'} + \hbar) \prod^{\sq \in \bla'}_{\sq' \in \bla'} \barzeta \left( \barchi_\sq' - \barchi_\sq \right)}
$$
As in relation \eqref{eqn:def phi}, we will replace the study of $\barA_{\barm}(x)$ with the closely related operator $\barPhi_{\barm}(x) : H_{\barbu'} \rightarrow H_{\barbu}$, defined by the formula:
$$
\barPhi_{\barm}(x) = \barA_{\barm}(x) \exp \left[\sum_{n=1}^\infty \frac {\barp_n}{x^n} \cdot \frac {\hbar}{\hbar_1 \hbar_2} \right]
$$
where $\barp_n = -\frac {B_n}n$ is the leading order term of the classical limit of the bosons $p_n$. The operator $\barPhi_{\barm}(x)$ was connected with vertex operators for the Toda conformal field theory in \cite{FL}. In \cite{Ext}, we used geometric techniques similar to those of Section \ref{sec:ext} to show that when $r=2$, the operator $\barPhi_{\barm}(x)$ is (up to conjugation by certain explicit exponentials in the Heisenberg subalgebra) equal to the Liouville vertex operator. In arbitrary rank $r$, we may obtain a similar result by taking the classical limit of formula \eqref{eqn:phi wk}. Specifically, this formula implies that:
$$
\Phi_m(x) \left[\sum_{k=0}^i (-1)^k W_k(y) {r-k \choose r-i} \right] \prod_{j=1}^i \left(1 - \frac {m^r u x}{q^{r-j} u' y} \right) = 
$$
$$
= \left[\sum_{k=0}^i (-1)^k m^k W_k(y) {r-k \choose r-i} \right] \Phi_m(x) \prod_{j=1}^i \left(1 - \frac {m^r u x}{q^{r-j} u' y} \right)
$$
By \eqref{eqn:limitation} and \eqref{eqn:callatis}, the leading order term (namely $\e^i$) of the above relation is:
\begin{equation}
\label{eqn:classical commutation}
\barPhi_{\barm}(x) \barW_i(y) (x-y)^i = \barW_i(y) \barPhi_{\barm}(x) (x-y)^i
\end{equation}
which yields the locality of the fields $\bar\Phi_{\barm}(x)$ and $\barW_i(y)$ (note that $m^k = 1 + O(\e)$ in the right-hand side does not contribute anything to the leading order term). Up to the fact that our conventions are somewhat non-standard, formula \eqref{eqn:classical commutation} is of the same nature as the realization of $\barPhi_{\barm}(x)$ as a vertex operator in \cite{FL}. \\ 


\section{Appendix}
\label{sec:appendix}

\noindent \textbf{Proof of formulas \eqref{eqn:def h}, \eqref{eqn:def e}, \eqref{eqn:def q}:} Fix any pair of coprime integers $a,b$, and recall that the subalgebra $\Lambda_{b/a} = \BF[P_{a,b},P_{2a,2b},...]$ is commutative. Moreover, as shown in \cite{Shuf}, $\Lambda_{b/a}$ is actually a bialgebra with respect to the coproduct:
$$
\Delta_{b/a}(R) = \sum_{i=0}^{na} \lim_{\xi \rightarrow \infty} \frac {R(z_1,..., z_i \otimes \xi z_{i+1},..., \xi z_{na})}{\xi^{\frac {b(na-i)}a}}
$$
for any $R \in \CS_{na,nb} \cap \Lambda_{b/a}$. The fact that $P_{na,nb}$ is primitive with respect to $\Delta_{b/a}$:
$$
\Delta_{b/a}(P_{na,nb}) = P_{na,nb} \otimes 1 + 1 \otimes P_{na,nb}
$$
was established in \loccit Moreover, it was shown therein that any other primitive element of $\Lambda_{b/a}$ is a constant multiple of $P_{na,nb}$. Because of this, the elements $H_{na,nb}$, $E_{na,nb}$, $Q_{na,nb}$ are exponentials in the shuffle elements $P_{na,nb}$ (times constant multiples) if and only if they are group-like with respect to $\Delta_{b/a}$:
\begin{equation}
\label{eqn:cop h}
\Delta_{b/a}(H_{na,nb}) = \sum_{m=0}^n H_{ma,mb} \otimes H_{(n-m)a,(n-m)b} 
\end{equation}
and the analogous formulas for $E_{na,nb}$ and $Q_{na,nb}$. These formulas follow from relation (6.10) of \loccitt, in the notation of which we have:
\begin{align}
&H_{na,nb} = X_{na,nb}^{(0,...,0)} \label{eqn:s h} \\
&E_{na,nb} =  (-q)^{n-1} X_{na,nb}^{(1,...,1)} \label{eqn:s e} \\
&Q_{na,nb} = \left(1 - \frac 1q\right) \left( X_{na,nb}^{(0,...,0)} + X_{na,nb}^{(0,...,0,1)} + ... + X_{na,nb}^{(0,1...,1)} + X_{na,nb}^{(1,...,1)} \right) \label{eqn:s q}
\end{align}
Indeed, the fact that \eqref{eqn:s h} implies \eqref{eqn:cop h} is an immediate application of (6.10) of \loccitt, as is the analogous statement for $E_{na,nb}$. The fact that the same holds for \eqref{eqn:s q} is proved analogously with Proposition 6.5 of \loccitt, and we leave the details to the interested reader.\footnote{The proof is actually in Subsection 5.7 of the first version of the paper \loccitt, which is unpublished, but can be found at ar$\chi$iv:1209.3349v1} It is a well-known fact concerning the bialgebra structure on the ring of symmetric functions that relation \eqref{eqn:cop h} implies that:
\begin{equation}
\label{eqn:truck}
\sum_{n=0}^\infty \frac {H_{an,bn}}{x^n} = \exp \left( \sum_{n=1}^\infty \frac {P_{an,bn}}{n x^n}  \cdot \alpha_n \right)
\end{equation}
for some constants $\alpha_1,\alpha_2,...$. Similarly, we have formulas analogous to \eqref{eqn:truck} for $E_{na,nb}$ and $Q_{na,nb}$, with respect to other sets of constants $\{\alpha_n\}_{n \in \BN}$. Therefore, to prove \eqref{eqn:def h}, \eqref{eqn:def e}, \eqref{eqn:def q}, it remains to compute these constants. As in \loccitt, we will do so by considering the linear functional:
$$
\phi : \Lambda_{b/a} \rightarrow \BF, \quad \phi(R(z_1,...,z_{na})) = \frac {R(1,q_1^{-1}...,q_1^{-na+1})}{\prod_{1\leq i < j \leq na} \zeta(q_1^{j-i})} \cdot q_1^{\frac {n^2ab-nb+na-n}2} (1-q_2)^{na}
$$
It is shown in (6.12) of \loccit that the functional $\phi$ is multiplicative. Moreover, from \eqref{eqn:pkd}--\eqref{eqn:qkd}, we see that:
\begin{align*}
&\phi(P_{na,nb}) = 1-q_2^n & &\phi(H_{na,nb}) = 1-q_2 \\
&\phi(E_{na,nb}) = (1 - q_2)(-q_2)^{n-1} & &\phi(Q_{na,nb}) = \frac {(1-q^{-1}) (1-q_2)(1-q_1^{-n})}{1-q_1^{-1}} 
\end{align*}
The reason for these very simple formulas is that $\zeta(q_1^{-1}) = 0$, and so only the identity permutation survives evaluation at $z_i = q_1^{-i+1}$ in formulas \eqref{eqn:pkd}--\eqref{eqn:qkd}. For brevity, we leave out the elementary manipulations involving integer parts that have produced the formulas above. Then applying the multiplicative function $\phi$ to \eqref{eqn:truck}, we conclude that:
$$
1 + \sum_{n=1}^\infty \frac {1-q_2}{x^n} = \exp \left( \sum_{n=1}^\infty \frac {\alpha_n (1-q_2^n)}{n x^n} \right) \quad \stackrel{\log}\Longrightarrow \quad \alpha_n = 1
$$
for all $n$. Then \eqref{eqn:truck} leads to \eqref{eqn:def h}. Formulas \eqref{eqn:def e}--\eqref{eqn:def q} are proved similarly. \\

\noindent \textbf{Proof of Proposition \ref{prop:explicit shuffle}:} Let us recall the following pairing on $\CS$ from \cite{Shuf}:
\begin{equation}
\label{eqn:shuffle pairing}
\langle \cdot, \cdot \rangle : \CS \otimes \CS \rightarrow \BF
\end{equation}
$$
\langle R, R' \rangle = \ :\int: \frac {R(z_1,...,z_k)R'\left(\frac 1{z_1}, ..., \frac 1{z_k} \right)}{\prod_{1\leq i \neq j \leq n} \zeta\left(\frac {z_i}{z_j} \right)} Dz_1... Dz_k
$$
The normal-ordered integral $:\int:$ is defined for all Laurent polynomials $\rho$ and:
$$
R' = \sym \left[\rho(z_1,...,z_k) \prod_{1\leq i < j \leq k} \zeta \left(\frac {z_i}{z_j} \right) \right]
$$
by the formula:
\begin{equation}
\label{eqn:normal shuffle pairing}
\langle R, R' \rangle := \int_{|z_1| \gg ... \gg |z_k|} \frac {R(z_1,...,z_k) \rho\left(\frac 1{z_1},..., \frac 1{z_k} \right)}{\prod_{1\leq i < j \leq n} \zeta\left(\frac {z_i}{z_j} \right)} Dz_1... Dz_k
\end{equation}
for all $R \in \CS$. Let us now consider the isomorphism $\CS \cong \CA^\uparrow$, under which the shuffle element $z^d$ in a single variable corresponds to the generator $P_{d,1}$. The situation we will mostly be concerned with is when: 
$$
\rho(z_1,...,z_k) = \delta \left(\frac {z_1}x \right) ... \delta \left( \frac {z_k}{xq^{1-k}} \right)
$$
in which case \eqref{eqn:normal shuffle pairing} implies:
\begin{equation}
\label{eqn:delta shuffle pairing 0}
\Big \langle R (z_1,...,z_k), W_k(x) \Big \rangle = \frac {R\left(\frac 1x, \frac q{x},..., \frac {q^{k-1}}x \right)}{\prod_{1\leq i < j \leq k} \zeta(q^{i-j})}
\end{equation}
because integrating against a $\delta$ function is equivalent to evaluation. A priori, the right-hand side of \eqref{eqn:delta shuffle pairing 0} is an undefined expression of the form $\frac {\infty}{\infty}$, but we make it precise as the linear functional $\ph^k_x : \CS_k \longrightarrow \BF [x^{\pm 1}]$ defined by:
\begin{equation}
\label{eqn:evaluation 0}
\ph^k_x(R) = \lim_{z_k \mapsto \frac {q^{k-1}}x} \left( ... \lim_{z_1 \mapsto \frac 1x} \left( \frac {R(z_1,...,z_k)}{\prod_{1\leq i < j \leq k} \zeta \left( \frac {z_i}{z_j} \right)} \right) ... \right)
\end{equation}
Indeed, each successive limit is well-defined for all rational functions $R$ of the form \eqref{eqn:shuf}, because at the $i$--th step, the pole $z_i - q^{i-1}/x$ of $R$ is precisely canceled by a zero of $\zeta^{-1}$. Therefore, \eqref{eqn:delta shuffle pairing 0} and \eqref{eqn:evaluation 0} imply that:
\begin{equation}
\label{eqn:delta shuffle pairing 1}
\Big \langle R (z_1,...,z_k), W_k(x) \Big \rangle = \ph^k_x(R)
\end{equation}
$$$$

\begin{claim}
\label{claim:1}
The functional $\ph_x$ satisfies the following multiplicativity property:
\begin{equation}
\label{eqn:pseudo}
\ph^{k_1+k_2}_x(R_1 * R_2) = \ph^{k_1}_{x}(R_1) \ph^{k_2}_{x/q^{k_1}}(R_2)
\end{equation}
for any shuffle elements $R_1,R_2$ of degrees $k_1,k_2$. \\
\end{claim}

\begin{claim}
\label{claim:2}
For any $k>0$ and $d\in \BZ$ with greatest common divisor $n$, we have:
\begin{equation}
\label{eqn:ph r}
\ph^k_x(P_{d,k}) = \frac {q^{\alpha(k,d)}}{x^d} \cdot \frac {(1-q_1^{n})(1-q_2^{n})(1-q^{-1})^{k}}{(1-q_1)^{k}(1-q_2)^{k}(1-q^{-n})}
\end{equation}
where $\alpha(k,d) = \frac {kd+k-d-n}2$. We think of $P_{d,k}$ as lying in $\CA^\uparrow \cong \CS$. \\
\end{claim}


\noindent We will prove these two claims after we show how they allow us to complete the proof of Proposition \ref{prop:explicit shuffle}. For any ordered collection $v$ as in \eqref{eqn:sequence} and $P_v$ defined as in \eqref{eqn:compositions}, relations \eqref{eqn:pseudo} and \eqref{eqn:ph r} imply that:
\begin{equation}
\label{eqn:ph rr}
\ph^k_x(P_v) = \frac {q^{\alpha(v)}}{x^d}  \prod_{i=1}^t \frac {(1-q_1^{n_i})(1-q_2^{n_i})(1-q^{-1})^{k_i}}{(1-q_1)^{k_i}(1-q_2)^{k_i}(1-q^{-n_i})}
\end{equation}
where we write $n_i = \gcd(k_i,d_i)$, $k = k_1+...+k_t$, $d = d_1+...+d_t$, and set:
\begin{equation}
\label{eqn:alpha}
\alpha(v) = \sum_{1\leq i < j \leq t} k_i d_j + \sum_{i=1}^t \frac {k_i d_i + k_i - d_i - n_i}2
\end{equation}
The products $P_v$ for ordered sequences $v$ are known (\cite{BS}, \cite{Shuf}) to form an orthogonal basis of the algebra $\CS$ with respect to the inner product \eqref{eqn:shuffle pairing}:
\begin{equation}
\label{eqn:basis}
\left \langle P_v, P_{v'} \right \rangle = \delta_{v'}^v \cdot z_v \prod_{i=1}^t \frac {(1-q_1^{n_i})(1-q_2^{n_i})(q^{-1}-1)^{k_i}}{(1-q_1)^{k_i}(1-q_2)^{k_i}(q^{-n_i}-1)}
\end{equation}
where:
\begin{equation}
\label{eqn:def zv}
z_v = \prod_{(k,d)} \left( \# \text{ of }(k,d) \text{ in }v \right)! \prod_{i=1}^t n_i
\end{equation}
Therefore, \eqref{eqn:delta shuffle pairing 1}, \eqref{eqn:ph rr} and \eqref{eqn:basis} imply the following formula:
\begin{equation}
\label{eqn:power}
W_k(x) = \sum_{t \geq 1} (-1)^{k-t} \mathop{\sum^{k_1+...+k_t = k}_{v = \left\{ \frac {d_1}{k_1} \leq ... \leq \frac {d_t}{k_t} \right\}}}^{k_i \in \BN, d_i \in \BZ}  \frac {P_v}{x^d} \cdot \frac { q^{\alpha(v)}}{z_v}
\end{equation}
Using relation \eqref{eqn:relation 1}, we conclude that the generators $P_{na,nb}$ for fixed coprime $(a,b)$ and all $n>0$ commute. Therefore, the multiplicative map that assigns to $P_{na,nb}$ the $n$--th power sum function $p_n$ also assigns to $E_{na,nb}$ the $n$--th elementary symmetric function $e_n$. Therefore, the well-known identity of symmetric functions:
$$
e_n = \sum^{\lambda \vdash n}_{\lambda = (n_1 \geq ... \geq n_t)} (-1)^{n-t} \frac {p_\lambda}{z_\lambda} \qquad \text{where } p_\lambda = p_{n_1}...p_{n_t}
$$
implies the following identity of shuffle elements for any slope $\frac ba$ with $\gcd(a,b) = 1$:
\begin{equation}
\label{eqn:dt}
E_{na,nb} = \sum^{\lambda \vdash n}_{\lambda = (n_1 \geq ... \geq n_t)} (-1)^{n-t} \frac {P_{\lambda(a,b)}}{z_\lambda}, \text{ where } P_{\lambda(a,b)} = P_{n_1a,n_1b}...P_{n_ta,n_tb}
\end{equation}
where we use the the well-known constant $z_\lambda = \prod_i (\# \text{ of }i \text{ in }\lambda)!\prod_{i=1}^t n_i$ defined for any partition $\lambda$ by analogy with \eqref{eqn:def zv}. Formula \eqref{eqn:dt} allows us to collect all terms $P_{d,k}$ of the same slope together in formula \eqref{eqn:power}, and rewrite it as:
\begin{equation}
\label{eqn:elementary}
W_k(x) = \mathop{\sum^{k_1+...+k_t = k}_{v = \left\{ \frac {d_1}{k_1} < ... < \frac {d_t}{k_t} \right\}}}^{t, k_i \in \BN, d_i \in \BZ}  \frac {E_v}{x^d} \cdot q^{\alpha(v)}(-1)^{\sum_{i=1}^t k_i  - \gcd(k_i,d_i)}
\end{equation}
where we still denote $E_v = E_{d_1,k_1}... E_{d_t,k_t}$. Since $E_{d,k}$ acts with degree $-d$ in any good representation, we may consider the LDU decomposition of \eqref{eqn:elementary}:
$$
W_k(x) = \sum^{d_\leftarrow, d_\rightarrow \geq 0}_{k_\leftarrow + k_0 + k_\rightarrow = k} \frac {L_{d_\leftarrow, k_\leftarrow} \cdot E_{0,k_0} \cdot U_{d_\rightarrow, k_\rightarrow}}{x^{d_\rightarrow - d_\leftarrow}} \cdot q^{(k_\leftarrow + k_0) d_\rightarrow}
$$
where we set $L_{0,k} = U_{0,k} = \delta_k^0$ and:
\begin{align}
&L_{d,k} = \mathop{\sum^{k_1+...+k_t = k}_{d_1+...+d_t = d}}^{t, k_i, d_i \in \BN}_{v = \left\{ -\frac {d_1}{k_1} < ... < -\frac {d_t}{k_t} \right\}}  E_v \cdot q^{\alpha(v)}(-1)^{\sum_{i=1}^t k_i  - \gcd(k_i,d_i)} \label{eqn:def for l} \\
&U_{d,k} = \mathop{\sum^{k_1+...+k_t = k}_{d_1+...+d_t = d}}^{t, k_i, d_i \in \BN}_{v = \left\{\frac {d_1}{k_1} < ... < \frac {d_t}{k_t} \right\}}  E_v \cdot q^{\alpha(v)}(-1)^{\sum_{i=1}^t k_i  - \gcd(k_i,d_i)} \label{eqn:def for u}
\end{align}
for $d>0$. The crucial fact is that $L_{d,k} \in \CA^\uparrow \cap \CA^\leftarrow$ and $U_{d,k} \in \CA^\uparrow \cap \CA^\rightarrow$. Up until now, we have thought about these operators as lying in the upper shuffle algebra, but from now on we wish to think about them in the left/right shuffle algebras: \\

\begin{claim}
\label{claim:3}

For any $k,d > 0$, we have the following equalities in the left and right shuffle algebras $\CA^\leftarrow \cong \CS$ and $\CA^\rightarrow \cong \CS^{\emph{\op}}$, respectively:
\begin{equation}
\label{eqn:lower upper}
L_{d,k} = T_{d,k}^\leftarrow, \qquad \qquad U_{d,k} = q^{d(k-1)} T_{d,k}^\rightarrow 
\end{equation}
where $T_{d,k}$ is the shuffle element of \eqref{eqn:awesome shuffle}.


\end{claim}

\tab 
This completes the proof of Proposition \ref{prop:explicit shuffle}, modulo Claims \ref{claim:1}, \ref{claim:2} and \ref{claim:3}, to which we now turn. To establish the first of these claims, note that the shuffle product \eqref{eqn:mult} implies that:
\begin{equation}
\label{eqn:weird} 
\ph_x^{k_1+k_2}(R_1*R_2)  = 
\end{equation}
$$
\sum_{A_1 \sqcup A_2 = \{1,...,k_1+k_2\}}^{|A_1|=k_1, |A_2|=k_2} R_1\left( \left\{\frac {q^{i-1}}x\right\}_{i \in A_1} \right) R_2\left( \left\{ \frac {q^{j-1}}x \right\}_{j \in A_2} \right) \frac {\prod^{i \in A_1}_{j \in A_2} \zeta (q^{i-j})}{\prod_{1 \leq i < j \leq k_1+k_2} \zeta(q^{i-j})}
$$
Since $R_i$ is of the form \eqref{eqn:shuf} for each of $i \in \{1,2\}$, the number of poles produced by the evaluation of $R_i$ as in \eqref{eqn:weird} equals the number of pairs of consecutive indices in $A_i$. Therefore, the total number of poles in each summand of \eqref{eqn:weird} equals:
$$
\# \left( \text{pairs of consecutive numbers in }A_1 \right) + \#  \left( \text{pairs of consecutive numbers in }A_2 \right) - 
$$
\begin{equation}
\label{eqn:lana's number}
- \# \left(i\in \{1,...,k_1+k_2-1\} \text{ and }i \notin A_1 \text{ or } i+1 \notin A_2 \right)
\end{equation}
It is elementary to see that the number in \eqref{eqn:lana's number} is always strictly negative, unless $A_1 = \{1,...,k_1\}$ and $A_2 = \{k_1+1,...,k_1+k_2\}$, in which case it is equal to zero. In the former case, the corresponding term in \eqref{eqn:weird} vanishes. In the latter case, it produces precisely the right-hand side of \eqref{eqn:pseudo}, completing the proof of Claim \ref{claim:1}.


\tab 
Let us now prove Claim \ref{claim:2} by induction on $k$. The task reduces to proving the claims contained in the two bullets below: \\

\begin{itemize}

\item the induction hypothesis implies the formula:
\begin{equation}
\label{eqn:t mob}
\ph^k_x(Q_{d,k}) = \frac {q^{\alpha(k,d)}}{x^d} \cdot \frac {(1-q^{-1})^{k}}{(1-q_1)^{k}(1-q_2)^{k}} \cdot \frac {(1-q_1)(1-q_2)(1-q^n)}{1-q}
\end{equation}

\item formula \eqref{eqn:t mob} and the induction hypothesis imply \eqref{eqn:ph r} \\

\end{itemize}

\noindent Let us prove the first bullet. Choose numbers $k_1+k_2 = k$ and $d_1+d_2 = d$ such that $k_1d_2 - k_2d_1 = n$, and $\gcd(k_1,d_1) = \gcd(k_2,d_2) = 1$ (the existence of such numbers comes down to the existence of a lattice triangle of minimal area with $(0,0), (k,d)$ as an edge, which is always the case). Then \eqref{eqn:relation 2} and the linearity of $\ph_x^k$ imply:
$$
\ph^k_x(Q_{d,k}) = \frac {1-q^{-1}}{(1-q_1)(1-q_2)} \left[ \ph^k_x(P_{d_1,k_1} * P_{d_2,k_2}) - \ph^k_x(P_{d_2,k_2} * P_{d_1,k_1}) \right]
$$
Then formula \eqref{eqn:pseudo} and the induction hypothesis of \eqref{eqn:ph r} imply that $\ph^k_x(Q_{d,k}) = $
$$
\frac {1-q^{-1}}{(1-q_1)(1-q_2)} \cdot \frac {q^{\alpha(k_1,d_1)+\alpha(k_2,d_2)}}{x^d} \cdot \frac {(1-q_1)^2(1-q_2)^2(1-q^{-1})^k}{(1-q_1)^k(1-q_2)^k(1-q^{-1})^2} \cdot (q^{k_1d_2}-q^{d_1k_2}) 
$$
which proves \eqref{eqn:t mob} (note that one needs the elementary formulas $k_1d_2 = k_2 d_1 + n$ and $\alpha(k,d) = \alpha(k_1,d_1)+\alpha(k_2,d_2)+ k_2d_1 + 1$). Let us now prove the second bullet, to which end we write $d = na$, $k=nb$ for $a$ and $b$ coprime. Formula \eqref{eqn:def q} implies:
$$
Q_{na,nb} = \sum^{\lambda \vdash n}_{\lambda = (n_1 \geq ... \geq n_t)} \frac {P_{\lambda(a,b)}}{z_\lambda} \prod_{i=1}^t (1-q^{-n_i})
$$
where we use the notation in \eqref{eqn:dt}. Applying the linear map $\ph_x^{nb}$ gives us:
$$
\ph^{nb}_x(Q_{na,nb}) = \sum^{\lambda \vdash n}_{\lambda = (n_1 \geq ... \geq n_t)} \frac {q^{\sum_{1 \leq i < j \leq t} n_in_j ab}}{z_\lambda} \prod_{i=1}^t \ph^{n_ib}_x(P_{n_ia,n_ib}) (1-q^{-n_i})
$$
where we used $\ph_x^{nb}(P_{n_1a,n_1b}...P_{n_ta,n_tb}) = \ph^{n_1b}_x(P_{n_1a,n_1b})...\ph^{n_t b}_x(P_{n_ta,n_tb}) q^{\sum_{i < j} n_in_j ab}$, which is a consequence of \eqref{eqn:pseudo}. Since the second bullet tells us that we may assume formula \eqref{eqn:t mob} for the left-hand side and formula \eqref{eqn:ph r} for all summands in the right-hand side except for $\lambda = (n)$, proving formula \eqref{eqn:ph r} for $P_{na,nb}$ boils down to establishing the identity:
$$
\frac {q^{\alpha(nb,na)}}{x^{na}} \cdot \frac {(1-q^{-1})^{nb}}{(1-q_1)^{nb}(1-q_2)^{nb}} \cdot \frac {(1-q_1)(1-q_2)(1-q^n)}{1-q} = 
$$
$$
\sum^{\lambda \vdash n}_{\lambda = (n_1 \geq ... \geq n_t)}  \frac {q^{\sum_{1 \leq i < j \leq t} n_in_j ab}}{z_\lambda} \prod_{i=1}^t \left[ \frac {q^{\alpha(n_ib,n_ia)}}{x^{n_ia}} \frac {(1-q_1^{n_i})(1-q_2^{n_i})(1-q^{-1})^{n_ib}}{(1-q_1)^{n_i b}(1-q_2)^{n_i b}(1-q^{-n_i})} (1-q^{-n_i}) \right]
$$
Note that $\alpha(nb,na) = \sum_{i=1}^t \alpha(n_ib,n_ia) + \sum_{1 \leq i < j \leq t} n_in_j ab$, as follows from \eqref{eqn:alpha}. After canceling common factors, the required equality states:
$$
\frac {(1-q_1)(1-q_2)(1-q^n)}{1-q}  = \sum^{\lambda \vdash n}_{\lambda = (n_1 \geq ... \geq n_t)} \frac {\prod_{i=1}^t (1-q_1^{n_i})(1-q_2^{n_i})}{z_\lambda}
$$
for all $n>0$. We will prove this by establishing the identity of generating series:
$$
1 + \sum_{n=1}^\infty \frac {(1-q_1)(1-q_2)(1-q^n)y^n}{1-q}  = \sum_{\lambda = (n_1 \geq ... \geq n_t)} \frac {\prod_{i=1}^t (1-q_1^{n_i})(1-q_2^{n_i})y^{n_i}}{z_\lambda}
$$
Summing the two sides of the equation above gives us:
$$
\frac {(1-q_1y)(1-q_2y)}{(1-y)(1-qy)} = \exp \left[ \sum_{n=1}^\infty \frac {(1-q_1^n)(1-q_2^n)y^n}{n} \right]
$$
which is a straightforward identity. \\

\noindent Finally, let us prove Claim \ref{claim:3}. We will only prove the statement for $L_{d,k}$, as the statement for $U_{d,k}$ follows by applying the anti-automorphism $P_{-d,k} \mapsto P_{d,k}$ (the power of $q$ in the formula for $U_{d,k}$ stems from the difference between $\alpha(v)$ and $\alpha(v^\dagger)$, where for a collection of lattice points $v$, its reflection across the vertical axis is denoted by $v^\dagger$). From formula \eqref{eqn:ekd}, we see that:
$$
E_{-d_1,k_1}...E_{-d_t,k_t} (-1)^{\sum_{i=1}^t k_i - n_i} = q^{\sum_{i=1}^t n_i - d_i} (-1)^{\sum_{i=1}^t k_i - d_i} \cdot 
$$
$$
\sym \left[ \frac {\prod_{s=1}^t \left( \prod_{i=1}^{d_s} z_{i+d_1+...+d_{s-1}}^{\left \lceil \frac {ik_s}{d_s} \right \rceil - \left \lceil \frac {(i-1)k_s}{d_s} \right \rceil} \right) \prod_{s=1}^{t-1} \left(1 - \frac {z_{d_1+...+d_s}}{qz_{d_1+...+d_s+1}} \right)}{\prod_{i=1}^{d-1} \left(1 - \frac {z_i}{qz_{i+1}} \right)} \prod_{1 \leq i < j \leq d} \zeta \left( \frac {z_i}{z_j} \right) \right]^\leftarrow
$$
where $d = d_1+...+d_t$ and $n_s = \gcd(k_s,d_s)$. Meanwhile, it is elementary to prove that $\alpha((-d_1,k_1),...,(-d_t,k_t))$ from formula \eqref{eqn:alpha} equals:
$$
 = -k(d-1) + \sum_{s=1}^t \left[ d_s - n_s+ \sum_{i=1}^{k_s} (i+d_1+...+d_{s-1}-1) \left(\left \lceil \frac {ik_s}{d_s} \right \rceil - \left \lceil \frac {(i-1)k_s}{d_s} \right \rceil \right) \right]
$$
If we let $w_i = z_i q^{i-1}$, then the two formulas above together imply that:
$$
L_{d,k} = q^{-k(d-1)} (-1)^{k-d} \mathop{\sum^{k_1+...+k_t = k}_{d_1+...+d_t = d}}^{k_i, d_i \in \BN}_{v = \left\{- \frac {d_1}{k_1} < ... < - \frac {d_t}{k_t} \right\}}  \sym \left[ \frac {\prod_{1 \leq i < j \leq d} \zeta \left( \frac {w_i q^j}{w_j q^i} \right)}{\prod_{i=1}^{d-1} \left(1 - \frac {w_i}{w_{i+1}} \right)}  \right.
$$
$$
\left. \prod_{s=1}^t \left( \prod_{i=1}^{d_s} w_{i+d_1+...+d_{s-1}}^{\left \lceil \frac {ik_s}{d_s} \right \rceil - \left \lceil \frac {(i-1)k_s}{d_s} \right \rceil} \right) \prod_{s=1}^{t-1} \left(1 - \frac {w_{d_1+...+d_s}}{w_{d_1+...+d_s+1}} \right) \right]^\leftarrow
$$
Therefore, formula \eqref{eqn:lower upper} follows from the identity of Laurent polynomials:
\begin{equation}
\label{eqn:iden}
\mathop{\sum^{k_1+...+k_t = k}_{d_1+...+d_t = d}}^{k_i, d_i \in \BN}_{v = \left\{ \frac {k_1}{d_1} < ... < \frac {k_t}{d_t} \right\}} \prod_{1\leq s \leq t}^{1 \leq i \leq d_s} w_{i+d_1+...+d_{s-1}}^{\left \lceil \frac {ik_s}{d_s} \right \rceil - \left \lceil \frac {(i-1)k_s}{d_s} \right \rceil} \prod_{s=1}^{t-1} \left(1 - \frac {w_{d_1+...+d_s}}{w_{d_1+...+d_s+1}} \right) = w_1w_d^{k-1}
\end{equation}
Note that the summands in the left-hand side are indexed by convex piecewise--linear paths $(0,0),(d_1,k_1),(d_1+d_2,k_1+k_2),...,(d,k)$ between the origin and the point $(d,k)$, which lie in the first quadrant minus the coordinate axes. Since we will encounter this terminology often, we will refer to such paths as \textbf{broken paths}. We will prove \eqref{eqn:iden} by counting how many times a monomial $w_1^{a_1}...w_d^{a_d}$ with $a_1+...+a_d = k$ appears in the left-hand side (the point is to show that the answer should be 0, unless the monomial is $w_1w_d^{k-1}$, when the answer should be 1). Such monomials are in one--to--one correspondences with non-decreasing collections:
$$
0 \leq s_1 \leq ... \leq s_{d-1} < k, \qquad \text{where} \qquad s_i+1 = a_1+...+a_i
$$
(we also make the convention that $s_0 = -1$ and $s_d = k-1$). Note that we can assume that $a_1 \geq 1$, because all monomials in \eqref{eqn:iden} have $w_1$ raised to positive powers. We identify $(s_1,...,s_{d-1})$ with the collection of lattice points $(1,s_1),...,(d-1,s_{d-1})$, and refer to this set of points as a \textbf{collection of bullets}. \\

\begin{claim}
\label{claim:broken}

The coefficient of $\prod_i w_i^{s_i-s_{i-1}}$ in the left-hand side of \eqref{eqn:iden} counts the number of broken paths $P$ between the lattice points $(0,0)$ and $(d,k)$, which: \\

\begin{enumerate}

\item intersect each line $x=i$ between heights $y = s_i$ and $y = s_{i}+1$, $\forall \ 1 \leq i < d$\\

\item can only pass through the point $(i,s_i)$ if the broken path bends at this point; in this case, we will call such an $(i,s_i)$ a \textbf{hinge}, and note that this terminology depends on both the broken path and the collection of bullets. \\ 

\item are counted with sign $(-1)^{\# \emph{ hinges}}$ in the left-hand side of \eqref{eqn:iden}

\end{enumerate}

\tab 
For a given collection of bullets $\CC$, we denote by $\CP(\CC)$ the collection of broken paths with the above properties. Then identity \eqref{eqn:iden} reduces to the fact that:
\begin{equation}
\label{eqn:collection}
\sum_{P \in \CP(\CC)} (-1)^{\# \emph{ hinges of }P} = 0
\end{equation}

\end{claim}

\begin{remark} Based on \eqref{eqn:collection}, we need to explain why the right-hand side of \eqref{eqn:iden} is $1 \cdot w_1w_d^{k-1}$ instead of $0 \cdot w_1w_d^{k-1}$. There are only two broken paths corresponding to the collection of bullets  $\{(1,0),...,(d-1,0)\}$: one is $(0,0), (d-1,1), (d,k)$ and the other is $(0,0), (1,0), (d-1,1), (d,k)$. The first contributes sign $+$ to \eqref{eqn:collection} and the second contributes sign $-$ to \eqref{eqn:collection}. However, only the first path contributes to the left-hand side of \eqref{eqn:iden}, because of the fact that $k_1,..., k_t$ therein must be $>0$. 

\end{remark}

\tab 
Claim \ref{claim:broken} is an elementary bijection between the sums in \eqref{eqn:iden} and \eqref{eqn:collection}. Indeed, as we have mentioned immediately after \eqref{eqn:iden}, there is a bijection between summands in the left-hand side of \eqref{eqn:iden} and broken paths. Moreover, to every subset $S = \{u_1,u_2,...\} \subset \{1,...,t-1\}$, we may associate the monomial:
$$
\prod_{1 \leq s \leq t}^{1 \leq i \leq d_s} w_{i+d_1+...+d_{s-1}}^{\left \lceil \frac {ik_s}{d_s} \right \rceil - \left \lceil \frac {(i-1)k_s}{d_s} \right \rceil} \prod_{u \in S} \left( - \frac {w_{d_1+...+d_u}}{w_{d_1+...+d_u+1}} \right)
$$
to the collection of bullets $(1,s_1),...,(d-1,s_{d-1})$ with:
$$
s_{d_1+...+d_{u-1}+i} = k_1+...+ k_{u-1} + \left \lceil \frac {ik_u}{d_u} \right \rceil - 1 +\delta_{u\in S}\delta_{d_u}^i \qquad \forall i \in \{1,...,d_u\}, \ \forall u\in \{1,...,t\}
$$
To any monomial in \eqref{eqn:iden}, this procedure associates a term $\pm 1$ in \eqref{eqn:collection}, and we leave it to the interested reader to show that this assignment is a bijection. The challenging part is the proof of statement \eqref{eqn:collection}, to which we now turn. For any collection of bullets $\CC$, consider the assignment:
$$
\CP(\CC) \stackrel{\Xi_\CC}\longrightarrow \Big \{ \text{subsets of }\{1,...,d-1\} \Big\}
$$
which sends a path to the set of $i$'s such that $(i,s_i)$ is a hinge. \\

\begin{claim}
\label{claim:inj}

The assignment $\Xi_\CC$ is injective. Its image is the collection of subsets $H \subset \{1,...,d-1\}$ with the property that for all $0 \leq a < b < c \leq d$ we have:
\begin{equation}
\label{eqn:inj 1}
\frac {s_c-s_b+1}{c-b} > \frac {s_b - s_a-1}{b-a}
\end{equation}
\begin{equation}
\label{eqn:inj 2}
\frac {s_c-s_b+1}{c-b} > \frac {s_b - s_a}{b-a} \quad \text{if } a \in H
\end{equation}
\begin{equation}
\label{eqn:inj 3}
\frac {s_c-s_b}{c-b} > \frac {s_b - s_a - 1}{b-a} \quad \text{if } c \in H
\end{equation}
\begin{equation}
\label{eqn:inj 4}
\frac {s_c-s_b}{c-b} \ > \ \frac {s_b - s_a}{b-a} \quad \text{if } a,c \in H
\end{equation}
By convention, we write $s_0 = -1$ and $s_d = k-1$. The four conditions above can be represented pictorially by requiring that the following 4 situations do not occur:

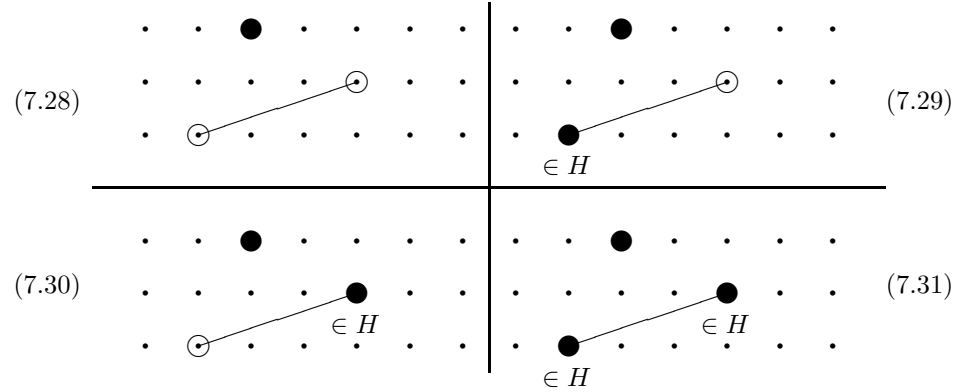
\begin{picture}(100,170)(-40,-25)

\put(0,0){\circle*{2}}\put(20,0){\circle*{2}}\put(40,0){\circle*{2}}\put(60,0){\circle*{2}}\put(80,0){\circle*{2}}\put(100,0){\circle*{2}}\put(120,0){\circle*{2}}\put(140,0){\circle*{2}}\put(160,0){\circle*{2}}\put(180,0){\circle*{2}}\put(200,0){\circle*{2}}\put(220,0){\circle*{2}}\put(240,0){\circle*{2}}\put(260,0){\circle*{2}}

\put(0,20){\circle*{2}}\put(20,20){\circle*{2}}\put(40,20){\circle*{2}}\put(60,20){\circle*{2}}\put(80,20){\circle*{2}}\put(100,20){\circle*{2}}\put(120,20){\circle*{2}}\put(140,20){\circle*{2}}\put(160,20){\circle*{2}}\put(180,20){\circle*{2}}\put(200,20){\circle*{2}}\put(220,20){\circle*{2}}\put(240,20){\circle*{2}}\put(260,20){\circle*{2}}

\put(0,40){\circle*{2}}\put(20,40){\circle*{2}}\put(40,40){\circle*{2}}\put(60,40){\circle*{2}}\put(80,40){\circle*{2}}\put(100,40){\circle*{2}}\put(120,40){\circle*{2}}\put(140,40){\circle*{2}}\put(160,40){\circle*{2}}\put(180,40){\circle*{2}}\put(200,40){\circle*{2}}\put(220,40){\circle*{2}}\put(240,40){\circle*{2}}\put(260,40){\circle*{2}}


\put(0,80){\circle*{2}}\put(20,80){\circle*{2}}\put(40,80){\circle*{2}}\put(60,80){\circle*{2}}\put(80,80){\circle*{2}}\put(100,80){\circle*{2}}\put(120,80){\circle*{2}}\put(140,80){\circle*{2}}\put(160,80){\circle*{2}}\put(180,80){\circle*{2}}\put(200,80){\circle*{2}}\put(220,80){\circle*{2}}\put(240,80){\circle*{2}}\put(260,80){\circle*{2}}

\put(0,100){\circle*{2}}\put(20,100){\circle*{2}}\put(40,100){\circle*{2}}\put(60,100){\circle*{2}}\put(80,100){\circle*{2}}\put(100,100){\circle*{2}}\put(120,100){\circle*{2}}\put(140,100){\circle*{2}}\put(160,100){\circle*{2}}\put(180,100){\circle*{2}}\put(200,100){\circle*{2}}\put(220,100){\circle*{2}}\put(240,100){\circle*{2}}\put(260,100){\circle*{2}}

\put(0,120){\circle*{2}}\put(20,120){\circle*{2}}\put(40,120){\circle*{2}}\put(60,120){\circle*{2}}\put(80,120){\circle*{2}}\put(100,120){\circle*{2}}\put(120,120){\circle*{2}}\put(140,120){\circle*{2}}\put(160,120){\circle*{2}}\put(180,120){\circle*{2}}\put(200,120){\circle*{2}}\put(220,120){\circle*{2}}\put(240,120){\circle*{2}}\put(260,120){\circle*{2}}


\put(130,-10){\line(0,1){140}}
\put(-20,60){\line(1,0){300}}

\put(-50,90){\eqref{eqn:inj 1}}
\put(-50,20){\eqref{eqn:inj 3}}
\put(280,90){\eqref{eqn:inj 2}}
\put(280,20){\eqref{eqn:inj 4}}

\put(20,80){\circle{8}}
\put(80,100){\circle{8}}
\put(40,120){\circle*{8}}
\put(20,80){\line(3,1){60}}

\put(160,80){\circle*{8}}
\put(220,100){\circle{8}}
\put(180,120){\circle*{8}}
\put(160,80){\line(3,1){60}}

\put(20,0){\circle{8}}
\put(80,20){\circle*{8}}
\put(40,40){\circle*{8}}
\put(20,0){\line(3,1){60}}

\put(160,0){\circle*{8}}
\put(220,20){\circle*{8}}
\put(180,40){\circle*{8}}
\put(160,0){\line(3,1){60}}

\put(150,65){$\in H$}\put(150,-15){$\in H$}\put(70,5){$\in H$}\put(210,5){$\in H$}


\end{picture}

\noindent where the full circles depict the lattice points $(i,s_i)$ and the hollow circles depict the lattice points $(i,s_i+1)$. We conclude that $\emph{Im } \Xi_\CC$ is an \textbf{abstract simplicial complex}, namely a collection of sets with the property that if some $H$ belongs to the collection, so do all of the subsets of $H$. Formula \eqref{eqn:collection} then becomes a statement about the reduced Euler characteristic of this abstract simplicial complex:
\begin{equation}
\label{eqn:collection 2}
\sum_{H \in \emph{Im } \Xi_\CC} (-1)^{|H|} = 0
\end{equation}

\end{claim}

\begin{proof} It is clear that properties \eqref{eqn:inj 1}--\eqref{eqn:inj 4} are necessary in order for a broken path passing through the set of hinges $H$ to be convex. Conversely, let us consider a set of hinges $H$ with properties \eqref{eqn:inj 1}--\eqref{eqn:inj 4}, and let us show that they correspond to a broken path $P$ with the correct properties. First of all, the hinges $\{(i,s_i), i\in H\}$ must themselves form a convex path, otherwise we would violate condition \eqref{eqn:inj 1}. In between two consecutive hinges $(i,s_i)$ and $(j,s_j)$, property (1) of Claim \ref{claim:broken} forces the piecewise--linear path $P$ to be the convex hull of the points:
\begin{equation}
\label{eqn:hull}
(i+1, s_{i+1}+1),...,(j-1,s_{j-1}+1)
\end{equation}
In particular, the path $P$ is uniquely determined in between any two consecutive hinges, which implies the injectivity of $\Xi_\CC$. The path $P$ thus constructed is convex between any two consecutive hinges, so we must also show that it is convex at each hinge $(i,s_i)$. In other words, we must show that the part of $P$ to the right of the hinge has slope greater than the part of $P$ that is to the left. This is guaranteed by formulas \eqref{eqn:inj 1}--\eqref{eqn:inj 4}.

\tab 
Finally, we must show that the convex piecewise--linear path $P$ thus constructed satisfies property (1) of Claim \ref{claim:broken} (since property (2) holds automatically). Assume that the path $P$ intersects the vertical line $x=b$ at height $y$. If $y>s_b+1$, then we contradict the fact that $P$ is the convex hull of the points \eqref{eqn:hull} between any two consecutive hinges. if $y \leq s_b$, then we contradict one of \eqref{eqn:inj 1}--\eqref{eqn:inj 4}.

\end{proof}

\noindent Therefore, we need to prove formula \eqref{eqn:collection 2} for the abstract simplicial complex consisting of subsets $H \subset \{1,...,d-1\}$ satisfying properties \eqref{eqn:inj 1}--\eqref{eqn:inj 4}. We may assume that \eqref{eqn:inj 1} is never violated, otherwise the abstract simplicial complex would be empty and \eqref{eqn:collection 2} would hold trivially. Therefore, let us consider the set:
\begin{equation}
\label{eqn:h0}
H_0 := \Big\{ 1,...,d-1 \Big\} \ \backslash \ \Big( \{a\text{'s as in \eqref{eqn:inj 2}}\} \cup \{c\text{'s as in \eqref{eqn:inj 3}}\} \Big)
\end{equation}
We claim that for all $a<b<c$ with $a,c\in H_0$, we have: 
\begin{equation}
\label{eqn:ineq h0}
\frac {s_c-s_b}{c-b} \geq \frac {s_b-s_a}{b-a}
\end{equation}
Indeed, assume for the purpose of contradiction that the opposite inequality holds. Then because $a,c\in H_0$, the points $(a,s_a)$ and $(c,s_c)$ cannot be in the situations of \eqref{eqn:inj 2} and \eqref{eqn:inj 3}. Therefore, the lattice point $(b,s_b)$ must be strictly inside the bottom triangle in the picture below. We assume $b$ is chosen such that the distance from the point $(b,s_b)$ to the line $l = \{(a,s_a), (c,s_c)\}$ is maximal.

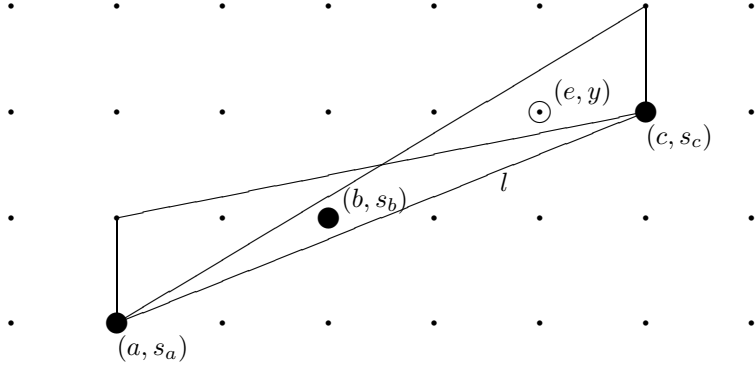
\begin{picture}(100,170)(-30,-25)

\put(0,0){\circle*{2}}
\put(40,0){\circle*{2}}
\put(80,0){\circle*{2}}
\put(120,0){\circle*{2}}
\put(160,0){\circle*{2}}
\put(200,0){\circle*{2}}
\put(240,0){\circle*{2}}
\put(280,0){\circle*{2}}

\put(0,40){\circle*{2}}
\put(40,40){\circle*{2}}
\put(80,40){\circle*{2}}
\put(120,40){\circle*{2}}
\put(160,40){\circle*{2}}
\put(200,40){\circle*{2}}
\put(240,40){\circle*{2}}
\put(280,40){\circle*{2}}

\put(0,80){\circle*{2}}
\put(40,80){\circle*{2}}
\put(80,80){\circle*{2}}
\put(120,80){\circle*{2}}
\put(160,80){\circle*{2}}
\put(200,80){\circle*{2}}
\put(240,80){\circle*{2}}
\put(280,80){\circle*{2}}

\put(0,120){\circle*{2}}
\put(40,120){\circle*{2}}
\put(80,120){\circle*{2}}
\put(120,120){\circle*{2}}
\put(160,120){\circle*{2}}
\put(200,120){\circle*{2}}
\put(240,120){\circle*{2}}
\put(280,120){\circle*{2}}

\put(40,0){\circle*{8}}
\put(120,40){\circle*{8}}
\put(240,80){\circle*{8}}
\put(200,80){\circle{8}}

\put(40,0){\line(5,2){200}}
\put(40,40){\line(5,1){200}}
\put(40,0){\line(5,3){200}}
\put(40,0){\line(0,1){40}}
\put(240,80){\line(0,1){40}}

\put(40,-12){$(a,s_a)$}
\put(125,44){$(b,s_b)$}
\put(240,68){$(c,s_c)$}
\put(205,85){$(e,y)$}
\put(185,50){$l$}

\end{picture}

\tab 
Assume without loss of generality that $b$ is closer to $a$ than to $c$, and consider the lattice point $(e,y)$ with $e = 2b-a$ and $y = 2s_b-s_a$. We have three options: \\

\begin{itemize}

\item If $s_e < y$ then we contradict the fact that $a \in H_0$ because the triple $(a,b,e)$ violates condition \eqref{eqn:inj 2}. \\

\item If $s_e = y$ and $(e,y)$ lies in the bottom triangle, then we violate the choice of $(b,s_b)$ as having maximal possible distance from the line $l$. \\

\item If $s_e = y$ and $(e,y)$ lies in the right triangle, or if $s_e>y$, then we contradict the fact that $c \in H_0$ because the triple $(a,e,c)$ violates condition \eqref{eqn:inj 3}. 

\end{itemize}

\tab 
Inequality \eqref{eqn:ineq h0} proves that the lattice points $\{(i,s_i) , i\in H_0\}$ form a convex path $R$ (we will call these lattice points \textbf{marks}) and moreover, that there are no other lattice points $(j,s_j)$ strictly above $R$. Recall from Claim \ref{claim:inj} that the abstract simplicial complex $\Xi_\CC$ consists of all subsets $H \subset H_0$ such that situation \eqref{eqn:inj 4} does not occur. Because there are no points $(j,s_j)$ above the convex path $R$, we conclude that the abstract simplicial complex consists of those subsets $H$ of the set of marks $H_0$, such that $H$ does not contain non-adjacent marks on a side of the convex path $R$. Then \eqref{eqn:collection 2} follows from the more general statement below: \\

\begin{claim}
\label{claim:finally}

For any convex path $R$ with marked points $p_1,...,p_n$, the abstract simplicial complex consisting of $H \subset \{1,...,n\}$ such that $H$ does not contain non-adjacent marked points on a side of $R$, has reduced Euler characteristic $\sum_H (-1)^{|H|} = 0$.

\end{claim}

\tab 
The claim is proved by induction on the number $n$. Assume that the rightmost marked point $p_n$ is on the last edge $E \subset R$. Then the set of subsets in question can be partitioned into groups: \\

\begin{itemize}

\item those $H$'s which do not contain $p_n$ \\

\item those $H$'s which contain $p_n$ and no other points on the edge $E$ \\

\item thise $H$'s which contain $p_n$ and $p_{n-1}$ and no other points on the edge $E$ \\

\end{itemize}

\noindent By the induction hypothesis, each of the three collections of $H$'s above has Euler characteristic zero, unless the convex path $R$ only consists of the edge $E$. If it happens that $R = E$, then the second and third bullets contribute 1 and $-1$ to the reduced Euler characteristic, and hence their contributions cancel each other out. \\

\noindent \textbf{Proof of Proposition \ref{prop:explicit shuffle 2}:} By the property of $\delta$ functions, observe that:
$$
W_k(x) * W_{k'}(y) =  \sym \left[\delta \left(\frac {z_1}x \right) ... \delta \left(\frac {z_k}{xq^{1-k}} \right) \delta \left( \frac {z_{k+1}}y \right) ...  \delta \left(\frac {z_{k+k'}}{y q^{1- k'}} \right)  \prod_{i < j} \zeta \left(\frac {z_i}{z_j} \right) \right]
$$
and therefore, the analogue of formula \eqref{eqn:delta shuffle pairing 1} states that:
\begin{equation}
\label{eqn:delta shuffle pairing 2}
\left \langle R (z_1,...,z_{k+k'}), W_k(x) * W_{k'}(y) \right \rangle = \ph_{x,y}^{k,k'}(R)
\end{equation}
where the functional $\ph_{x,y}^{k,k'} : \CS_{k+k'} \longrightarrow \BF(x,y)$ is defined by:
\begin{equation}
\label{eqn:phi kl}
\ph_{x,y}^{k,k'}(R) = \frac {R\left(\frac 1x,..., \frac {q^{k-1}}x, \frac 1y,..., \frac {q^{k'-1}}y \right)}{ \prod_{1\leq i < i' \leq k} \zeta (q^{i-i'}) \prod_{1\leq j < j' \leq k'} \zeta (q^{j-j'})  \prod^{1\leq i \leq k}_{1\leq j \leq k'} \zeta\left(\frac {yq^i}{xq^j} \right)} := 
\end{equation}
$$
:= \lim_{z_{k+k'} \mapsto \frac {q^{k'-1}}y} \left( ... \lim_{z_{k+1} \mapsto \frac 1y} \left( \lim_{z_k \mapsto \frac {q^{k-1}}x} \left( ... \lim_{z_1 \mapsto \frac 1x} \left( \frac {R(z_1,...,z_{k+k'})}{\prod_{1\leq i < j \leq k+k'} \zeta \left( \frac {z_i}{z_j} \right)} \right) ... \right) \right) ... \right)
$$
By analogy with the proof of Proposition \ref{prop:explicit shuffle}, we claim that it suffices to prove the following statement. For any $k,k' \geq 0$ and any shuffle element $R$ of degree $k+k'$, we have:
\begin{equation}
\label{eqn:reduces}
\ph_{x,y}^{k,k'}(R) = \prod_{i=\max(0,k'-k)+1}^{k'} \frac 1{y-xq^i} \prod_{i=-k+1}^{\min(0,k'-k)} \frac {y-xq^i}{(y - x q_1 q^{i-1})(y - x q_2 q^{i-1})} ...
\end{equation}
where $...$ stands for a Laurent polynomial in $x,y$. Indeed, combined with Claim \ref{claim:2} and \eqref{eqn:basis}, formula \eqref{eqn:reduces} applied to $R = P_v$ implies that the basis element $P_v$ of the shuffle algebra appears in the decomposition of $W_k(x) * W_{k'}(y)$ with a coefficient which is a rational function of $x,y$ as in the right-hand side of \eqref{eqn:reduces}. \\

\noindent Therefore, it remains to prove \eqref{eqn:reduces}. By the argument in Proposition 2.9 of \cite{Shuf}, which closely follows a key Lemma of \cite{FHHSY}, the wheel conditions \eqref{eqn:wheel} imply that the evaluation:
\begin{equation}
\label{eqn:evaluation}
R\left(\frac 1x,..., \frac {q^{k-1}}x, \frac 1y,..., \frac {q^{k'-1}}y \right)
\end{equation}
is divisible by the linear factors:
$$
\begin{cases} \frac {q^{a-1}}x = \frac {q_1 q^{b-1}}y \text{ and } \frac {q^{a-1}}x = \frac {q_2 q^{b-1}}y \text{ for } 1 \leq a \leq k, \ 1 \leq b < k' & \text{ if } k\leq k' \\
\frac {q^{b-1}}y = \frac {q_1 q^{a-1}}x \text{ and } \frac {q^{b-1}}y = \frac {q_2 q^{a-1}}x \text{ for } 1 \leq a < k, \ 1 \leq b \leq k' & \text{ if } k\geq k' \end{cases}
$$
Dividing \eqref{eqn:evaluation} by $ \prod^{1\leq i \leq k}_{1\leq j \leq k'} \zeta \left( \frac {yq^i}{xq^j} \right)$ as in \eqref{eqn:phi kl} yields:
$$
\ph_{x,y}^{k,k'}(R) = \frac {...}{\prod_{i=-k + 1}^{\min(0,k'-k)} (y - x q_1 q^{i-1})(y - x q_2 q^{i-1})}
$$
where ... refers to a Laurent polynomial times linear factors other than those of the form $y-xq_1q^*$ and $y-xq_2q^*$. Meanwhile, as far as the factors $y-xq^*$ are concerned, let us write $R$ as a shuffle element \eqref{eqn:shuf}:
$$
\ph_{x,y}^{k,k'}(R) = \frac {r \left(\frac 1x,..., \frac {q^{k-1}}x, \frac 1y,..., \frac {q^{k'-1}}y \right)}{\prod^{1\leq i \leq k}_{1\leq j \leq k'} (y-xq^{j-i-1})(y-xq^{j-i+1})} \prod^{1\leq i \leq k}_{1\leq j \leq k'} \frac {(y-xq^{j-i})(y-xq^{j-i-1})}{ (y-xq_1q^{j-i-1})(y-xq_2q^{j-i-1})} 
$$
$$
\cdot \ ... = ... \cdot \left( \begin{cases} \prod_{i=1}^k \frac {y-xq^{1-i}}{y-xq^{k'+1-i}} & \text{ if } k \leq k' \\
\prod_{j=1}^{k'} \frac {y-xq^{j-k}}{y-xq^{j}}  & \text{ if } k \geq k' \end{cases} \right) = ... \cdot \frac {\prod_{i=-k+1}^{\min(0,k'-k)} (y-xq^i)}{\prod_{i=\max(0,k'-k)+1}^{k'} (y-xq^i)}
$$
where ... denote Laurent polynomials and linear factors other than those of the form $y-xq^*$. This completes the proof of \eqref{eqn:reduces}, and with it, Proposition \ref{prop:explicit shuffle 2}. \\

\noindent \textbf{Proof of Proposition \ref{prop:w}:} As a consequence of \eqref{eqn:relation 2}, we have:
$$
[P_{d,1},P_{\pm n,0}] = \pm (1-q_1^n)(1-q_2^n) P_{d \pm n,1} \quad \Rightarrow 
$$
\begin{equation}
\label{eqn:nice comm}
\quad \Rightarrow \left[\delta \left(\frac zx \right)^\uparrow, P_{\pm n,0} \right] = \pm (1-q_1^n)(1-q_2^n) \cdot x^{\pm n} \delta \left(\frac zx \right)^\uparrow
\end{equation}
which proves \eqref{eqn:w rel 0 minus}--\eqref{eqn:w rel 0 plus}. Let $\beta_n = \frac {(1-q_1^n)(1-q_2^n)}n$ and consider the functions:
$$
f_{kk'} (z) = \prod^{k-1}_{i = \max(0,k-k')} \zeta (z q^i) =  \exp \left[ \sum_{n=1}^\infty z^n \beta_n \cdot \frac {q^{\max(0,k-k')n} - q^{kn}}{1-q^n} \right]
$$
$$
g_{kk'} (z) = \prod^{1\leq i \leq k}_{1 \leq j \leq k'} \zeta \left( z q^{i-j-1} \right) = \exp \left[\sum_{n=1}^\infty \frac {z^n \beta_n}{q^n} \frac {(1-q^{kn})(1-q^{-k' n})}{(1-q^n)(1-q^{-n})} \right]
$$
$$
h_{kk'} (z) = f_{kk'} (z) g_{kk'}(z) = \exp \left[\sum_{n=1}^\infty z^n \beta_n \frac {q^{\max(0,k-k')n} + q^{\min(0,k-k')n-n} - q^{kn} - q^{-(k'+1)n}}{(1-q^n)(1-q^{-n})} \right]
$$
Recall formula \eqref{eqn:def explicit 2}:
\begin{equation}
\label{eqn:aaa}
W_k(x) W_{k'}(y) = S_{k,k'}(x,y) g_{kk'} \left(\frac yx\right) 
\end{equation}
where:
\begin{equation}
\label{eqn:pressure}
S_{k,k'}(x,y) = \eta_k \eta_{k'} \sym \left[ \delta \left(\frac {z_1}x \right) ... \delta \left( \frac {z_k}{xq^{1-k}} \right) \delta \left(\frac {z_{k+1}}y \right) ... \delta \left( \frac {z_{k+k'}}{yq^{1-k'}} \right) \right]^\uparrow  
\end{equation}
Moreover, formula \eqref{eqn:explicit shuffle 2} claims that the expression:
\begin{equation}
\label{eqn:bbb}
W_k(x)W_{k'}(y) f_{kk'}\left(\frac yx \right) = S_{k,k'}(x,y) h_{kk'} \left(\frac yx \right)
\end{equation}
is a rational function with poles given by \eqref{eqn:poles 1}--\eqref{eqn:poles 2}. Note that we have the equality of formal sums $S_{k,k'}(x,y) = S_{k',k}(y,x)$ by symmetry, while:
$$
h_{kk'} \left(\frac yx \right) = h_{k'k} \left(\frac xy \right)
$$
by using formulas \eqref{eqn:inv zeta} and \eqref{eqn:zeta exponential}. Therefore, we can switch $(k,x) \leftrightarrow (k',y)$ in \eqref{eqn:bbb} at the cost of picking up the residues at the aforementioned poles:
\begin{equation}
\label{eqn:s}
S_{k,k'}(x,y) h_{kk'}\left(\frac yx \right) - S_{k',k}(y,x) h_{k'k}\left(\frac xy \right) =
\end{equation}
$$
= \sum_{i=\max(0,k'-k)+1}^{k'} \delta\left(\frac {y}{xq^i} \right) \underset{x = \frac y{q^i}}{\res} \left[ S_{k,k'}(x,y) h_{kk'}\left( \frac yx \right) \frac {dx}x \right] - 
$$
$$
- \sum^{k}_{i = \max(0,k-k')+1} \delta\left(\frac {x}{yq^i} \right) \underset{y = \frac x{q^i}}{\res} \left[ S_{k',k}(y,x) h_{k'k} \left(\frac xy \right) \frac {dy}y \right]
$$
From the definition of $S_{k,k'}(x,y)$ in \eqref{eqn:pressure}, we see that:
$$
S_{k,k'}(x,y) \Big|_{x = \frac y{q^i}} = \frac {\eta_k \eta_{k'}}{\eta_{k+i} \eta_{k'-i}} \cdot  S_{k'-i,k+i} (x,y) \Big|_{x = \frac y{q^i}}
$$
and so we may rewrite \eqref{eqn:s} as:
\begin{equation}
\label{eqn:ss}
S_{k,k'}(x,y) h_{kk'}\left(\frac yx \right) - S_{k',k}(y,x) h_{k'k}\left(\frac xy \right) =
\end{equation}
$$
= \sum_{i=\max(0,k'-k)+1}^{k'} \delta\left(\frac {y}{xq^i} \right) \underset{x = \frac y{q^i}}{\res} \left[ S_{k'-i,k+i} (x,y) h_{kk'} \left( \frac yx \right) \frac {dx}x \right] \frac {\eta_k \eta_{k'}}{\eta_{k+i} \eta_{k'-i}}  - 
$$
$$
- \sum^{k}_{i = \max(0,k-k')+1} \delta\left(\frac {x}{yq^i} \right) \underset{y = \frac x{q^i}}{\res} \left[ S_{k-i,k'+i} (y,x) h_{k'k}\left( \frac xy \right) \frac {dy}y \right] \frac {\eta_k \eta_{k'}}{\eta_{k-i} \eta_{k'+i}}
$$
Let us rewrite this equality as:
\begin{equation}
\label{eqn:sss}
S_{k,k'}(x,y) h_{kk'}\left(\frac yx \right) - S_{k',k}(y,x) h_{k'k}\left(\frac xy \right) =
\end{equation}
$$
= \sum_{i=\max(0,k'-k)+1}^{k'} \delta\left(\frac {y}{xq^i} \right) \underset{x = \frac y{q^i}}{\res} \left[ S_{k'-i,k+i} (x,y) h_{k'-i,k+i} \left( \frac yx \right) \frac {h_{kk'} \left(\frac yx \right) \eta_k \eta_{k'} \cdot \frac {dx}x}{h_{k'-i,k+i} \left(\frac yx \right) \eta_{k+i} \eta_{k'-i}} \right]  
$$
$$
- \sum^{k}_{i = \max(0,k-k')+1} \delta\left(\frac {x}{yq^i} \right) \underset{y = \frac x{q^i}}{\res} \left[ S_{k-i,k'+i} (y,x) h_{k-i,k'+i}\left( \frac xy \right)  \frac {h_{k'k} \left( \frac xy \right)\eta_k \eta_{k'} \cdot \frac {dy}y}{h_{k-i,k'+i} \left( \frac xy \right) \eta_{k-i} \eta_{k'+i}} \right] 
$$
Using the definition of $h_{kk'}(z)$ and $\eta_k$, we conclude that $\frac {h_{kk'}(q^i)\eta_k \eta_{k'}}{h_{k'-i,k+i}(q^i) \eta_{k+i} \eta_{k'-i}} $ equals: 
$$
\exp \left[\sum_{n=1}^\infty\beta_n \left(  \frac {q^{[\max(0,k-k')+i]n} + q^{[\min(0,k-k')+i-1]n} - q^{(k+i)n} - q^{(i-k'-1)n}}{(1-q^n)(1-q^{-n})} -\right. \right.
$$
$$
 \left. \left. - \frac {q^{in} + q^{(k'-k-i-1)n} - q^{k'n} - q^{-(k+1)n}}{(1-q^n)(1-q^{-n})}  - \frac {(q^{k'n} - q^{(k+i)n})(1-q^{-in})}{(1-q^n)(1-q^{-n})} \right) \right] =
$$
\begin{equation}
\label{eqn:snake}
= \exp \left(\sum_{n=1}^\infty \beta_n \cdot \frac {1 - q^{\min(i,k-k'+i)n}}{1-q^n} \right) = \prod_{s=0}^{\min(i,k-k'+i)-1} \zeta(q^s)
\end{equation}
Note that in the first equality in \eqref{eqn:snake}, we have made generous use of the identity \eqref{eqn:inv zeta}, which in exponential notation takes the form:
$$
\exp \left(\sum_{n=1}^\infty \beta_n \cdot q^{na} \right) = \exp \left( \sum_{n=1}^\infty \beta_n \cdot q^{(-a-1)n} \right) 
$$
for any integer $a$. Also note that $\zeta(1)$ has a factor of $1-1$ in the denominator, and this factor is precisely the reason why the summands in the right-hand side of \eqref{eqn:sss} have a simple pole at $y = x q^i$ and $x = y q^i$, respectively. Without this factor, expression \eqref{eqn:snake} equals $\theta(\min(i,k-k'+i))$ of \eqref{eqn:def theta}. Therefore, changing all $S$'s to $W$'s via relation \eqref{eqn:bbb} converts formula \eqref{eqn:sss} into \eqref{eqn:w rel}.


\begin{thebibliography}{XXX}

\bibitem{Aga} Aganagic M., Haouzi N., Shakirov S., {\em $A_n$--triality}, ar$\chi$iv:1403.3657

\bibitem{AFHKSY} Awata H., Feigin B., Hoshino A., Kanai M., Shiraishi J., Yanagida S., {\em Notes on Ding-Iohara algebra and AGT conjecture}, ar$\chi$iv:1106.4088.

\bibitem{AKOS} Awata H., Kubo H., Odake S. and Shiraishi J., {\em Quantum $W_N$ algebras and Macdonald polynomials} \textbf{Comm. Math. Phys. } 179 (1996), no.2, 401--416

\bibitem{AY} Awata H., Yamada Y., {\em Five-dimensional AGT Relation and the Deformed $\beta$--ensemble}, \textbf{Prog.Theor.Phys. } 124 (2010), 227--262

\bibitem{Ba} Baranovsky V., {\em Moduli of sheaves on surfaces and action of the oscillator algebra}, \textbf{J. Diff. Geom.} 55 (2000), no. 2, 

\bibitem{Bo0} Bourgine J.-E., Matsuo Y., Zhang H., {\em Holomorphic field realization of $\emph{SH}_c$ and quantum geometry of quiver gauge theories}, \textbf{J. High Energy Phys.}, vol. 167 (2016)

\bibitem{Bo} Bourgine J.-E., Fukuda M., Matsuo Y., Zhang H., Zhu R.-D.,  {\em Coherent states in quantum $\CW_{1+\infty}$ algebra and qq--character for 5d Super Yang-Mills}, ar$\chi$iv:1606.08020

\bibitem{BoS} Bouwknegt P., Schoutens K., {\em W symmetry in conformal field theory} \textbf{Phys.Rept.} 223 (1993) 183-276

\bibitem{BS} Burban I., Schiffmann O., {\em On the Hall algebra of an elliptic curve I} \textbf{Duke Math. J.} 161 (2012), no. 7, 1171-1231

\bibitem{CNO} Carlsson E., Nekrasov N., Okounkov A., {\em Five dimensional gauge theories and vertex operators}, \textbf{Mosc. Math. J. } 14 (2014), no. 1, 39–61, 170.

\bibitem{CO} Carlsson E., Okounkov A., {\em Exts and vertex operators}, \textbf{Duke Math. J.} 161 (2012), no. 9, 1797-1815


\bibitem{FL} Fateev V., Litvinov A. {\em Integrable structure, W-symmetry and AGT relation}, \textbf{J. High Energ. Phys.} (2012) 2012: 51

\bibitem{FLu} Fateev V., Lukyanov S. {\em The Models of Two-Dimensional Conformal Quantum Field Theory with $Z(n)$ Symmetry}, \textbf{Int. J. Mod. Phys.} A3 (1988) 507

\bibitem{FF} Feigin B. and Frenkel E., {\em Quantum $W$--algebras and elliptic algebras}, \textbf{Comm. Math. Phys. }, 178 (1996), no. 3, 653--678

\bibitem{FHHSY} Feigin B., Hashizume K., Hoshino A., Shiraishi J., Yanagida S., {\em A commutative algebra on degenerate $\BC \BP^1$ and MacDonald polynomials},  \textbf{J. Math. Phys.} 50 (2009), no. 9

\bibitem{Kernel} Feigin B., Hoshino A., Shibahara J., Shiraishi J., Yanagida S. {\em Kernel function and quantum algebras}, ar$\chi$iv: 1002.2485

\bibitem{FJMM} Feigin B., Jimbo M., Miwa T., Mukhin E., {\em Quantum toroidal $\mathfrak{gl}_1$ algebra: plane partitions}, \textbf{Kyoto J. Math.} 52 (2012), no 3, 621--659

\bibitem{FO} Feigin B., Odesskii A., {\em Vector bundles on elliptic curve and Sklyanin algebras}, Topics in Quantum Groups and Finite-Type Invariants, \textbf{Amer. Math. Soc. Transl. Ser. 2} 185 (1998), Amer. Math. Soc., 65-–84

\bibitem{FT} Feigin B., Tsymbaliuk A., {\em Heisenberg action in the equivariant $K-$theory of Hilbert schemes via Shuffle Algebra}, \textbf{Kyoto J. Math.} 51 (2011), no. 4

\bibitem{GN} Gorsky E., Negu\cb t A., {\em Infinitesimal change of stable basis}, \textbf{Selecta Math.}, July 2017, Volume 23, Issue 3, pp 1909–1930


\bibitem{G} Grojnowski I., {\em Instantons and Affine Algebras I. The Hilbert Scheme and Vertex Operators}, \textbf{Math. Res. Lett.} 3 (1996), no. 2


\bibitem{KP} Kimura T., Pestun V., {\em Quiver $W$--algebras}, ar$\chi$iv:1512.08533


\bibitem{MO} Maulik D., Okounkov A., {\em Quantum groups and quantum cohomology}, ar$\chi$iv:1211.1287

\bibitem{O} Odake S., {\em Comments on the deformed $W_N$ algebra}, \textbf{Int. J. Mod. Phys. B}, 16, 2055 (2002)


\bibitem{ON} Okounkov A., Nekrasov N., {\em Seiberg-Witten Theory and Random Partitions}, \textbf{Progress in Mathematics},  The Unity of Mathematics (244), 525--596



\bibitem{NY} Nakajima H., Yoshioka K., {\em Instanton counting on blowup. II. K-theoretic partition function}, \textbf{Transform. Groups } 10 (2005), no. 3-4, 489-–519.

\bibitem{Ext} Negu\cb t A., {\em Exts and the AGT Relations}, \textbf{Lett. Math. Phys.}, 106 (2016), no. 9, 1265–1316

\bibitem{Shuf} Negu\cb t A., {\em The shuffle algebra revisited},  \textbf{Int. Math. Res. Not.} 22 (2014), 6242--6275

\bibitem{Mod} Negu\cb t A., {\em Moduli of flags of sheaves and their $K$--theory}, \textbf{Algebraic Geometry} 2 (2015), 19--43

\bibitem{Pieri} Negu\cb t A., {\em The $\frac mn$ Pieri rule}, \textbf{Int. Math. Res. Not.}, 2016 (1): 219--257

\bibitem{Ops} Negu\cb t A., {\em Operators on symmetric polynomials}, ar$\chi$iv:1310.3515

\bibitem{W gen} Negu\cb t A., {\em $W$--algebras associated to surfaces}, ar$\chi$iv:1710.03217

\bibitem{AGT} Negu\cb t A., {\em AGT relations for sheaves on surfaces}, ar$\chi$iv:1711.00390


\bibitem{Nek} Nekrasov N., {\em Seiberg-Witten prepotential from instanton counting}, \textbf{Adv. Theor. Math. Phys.} 7 (2003), no. 5, 831-864

\bibitem{NS} Nekrasov N., Shadchin S., {\em ABCD of instantons}, \textbf{Commun. Math. Phys.} 252 (2004) 359--391

\bibitem{NPS} Nekrasov N., Pestun V., Shatashvili S., {\em Quantum geometry and quiver gauge theories}, ar$\chi$iv:1312.6689 

\bibitem{SV} Schiffmann O., Vasserot E., {\em The elliptic Hall algebra and the equivariant $K-$theory of the Hilbert scheme of ${\mathbb{A}}^2$}, \textbf{Duke Math. J.} 162 (2013), no. 2, 279--366

\bibitem{SV2} Schiffmann O., Vasserot E., {\em Cherednik algebras, $W$--algebras and the equivariant cohomology of the moduli space of instantons on $\BA^2$}, \textbf{Publ. Math. Inst. Hautes Etud. Sci.}, 118 (2013), Issue 1, 213-–342

\bibitem{Ta} Tachikawa Y., {\em A brief review of the 2d/4d correspondences}, \textbf{Journal of Physics A: Mathematical and Theoretical}, Volume 50, Number 44

\bibitem{T} Taki M., {\em On AGT-W Conjecture and q-Deformed W-Algebra}, ar$\chi$iv:1403.7016

\bibitem{Y} Yanagida S., {\em Whittaker vector of deformed Virasoro algebra and Macdonald symmetric functions}, \textbf{Lett. Math. Phys.} 106 (2016), no. 3, 395--431

\end{thebibliography}
\end{document}